\documentclass[twoside,11pt]{article}

\usepackage{jmlr2e}

\RequirePackage[OT1]{fontenc}
\RequirePackage{hypernat}
\usepackage{mathscinet}
\usepackage{enumerate}
\usepackage[utf8]{inputenc} % Any characters can be typed directly from the keyboard, eg éçñ
\usepackage{natbib} % use author/date bibliographic citations
\usepackage{amsmath}  % Better maths support & more symbols
\usepackage[pdftex,bookmarks,colorlinks,breaklinks]{hyperref}  % PDF hyperlinks, with coloured links
\usepackage{url}
\usepackage{multirow}
\usepackage{verbatim}
\usepackage{afterpage}
\usepackage{chngcntr}
\hypersetup{linkcolor=blue,citecolor=blue,filecolor=cyan,urlcolor=blue} % coloured links

\usepackage{caption}
\counterwithout{figure}{section}
\counterwithout{figure}{subsection}
\makeatletter
\newcommand*{\rom}[1]{\expandafter\@slowromancap\romannumeral #1@}
\makeatother

\def\Stas{\textcolor{magenta}}

\newcommand{\dotp}[2]{\left\langle #1, #2\right\rangle}
\newcommand{\argmin}{\mathop{\rm argmin~}}

\newcommand{\vol}{\mathop{\rm Vol}{}_{\mathcal M}}

\newcommand{\proj}{\mathop{{\rm Proj}}}
\newcommand{\var}{\mathop{{\rm Var}}}

\newcommand{\diag}{\mathop{\rm diag}}
\newcommand{\card}{\mathop{\rm card}}
\newcommand{\Id}{\mathop{\rm Id}}

\def\Vol{{\rm Vol}}

\def\eps{\varepsilon}
\def\m{\mathcal}
\def\mb{\mathbb}

\def\card{{\rm Card}}

\def\supp{{\rm supp}}
\def\tr{{\rm tr\,}}
\def\l{\left}
\def\r{\right}

\ShortHeadings{Multiscale Dictionary Learning}{Maggioni, Minsker, and Strawn}
\firstpageno{1}

\begin{document}

\title{Multiscale Dictionary Learning: Non-Asymptotic Bounds and Robustness}

\author{\name Mauro Maggioni \email mauro@math.duke.edu \\
       \addr Departments of Mathematics, Electrical and Computer Engineering, and Computer Science\\
       Duke University\\
      Durham, NC 27708, USA
       \AND
       \name Stanislav Minsker \email minsker@usc.edu \\
       \addr Department of Mathematics\\
       University of Southern California\\
       Los Angeles, CA 90089, USA
       \AND
       \name Nate Strawn \email nate.strawn@georgetown.edu \\
       \addr Department of Mathematics and Statistics\\
       Georgetown University\\
     Washington D.C., 20057, USA
     }

\editor{}

\maketitle

%\author{\fnms{Mauro} \snm{Maggioni}\thanksref{t1,a1,a2,a3}\ead[label=e1]{mauro@math.duke.edu}}
%\author{\fnms{Stanislav} \snm{Minsker}\thanksref{t1,a1,a4}\ead[label=e2]{sminsker@math.duke.edu}}
%\author{\fnms{Nate} \snm{Strawn}\thanksref{t1,a1,a2}\ead[label=e3]{nstrawn@math.duke.edu}}

%\thankstext{t1}{The authors gratefully acknowledge support from NSF DMS-0847388, NSF DMS-1045153, ATD-1222567, CCF-0808847, AFOSR FA9550-14-1-0033, DARPA N66001-11-1-4002}
%\runauthor{M. Maggioni, S. Minsker, and N. Strawn}
%\runtitle{Multiscale Dictionary Learning}
%\affiliation{Duke University\thanksmark{m1}}

%\address{Department of Mathematics\thanksmark{a1}, Electrical and Computer Engineering\thanksmark{a2}, Computer Science\thanksmark{a3} and Statistics\thanksmark{a4}\\
%Duke University\\
%\phantom{{E}-mail:\ }\printead{e1},
%\printead*{e2},
%\printead*{e3}}
%\end{aug}
%\date
%\maketitle

\begin{abstract}
\hspace*{5pt}High-dimensional datasets are well-approximated by low-dimensional structures. Over the past decade, this empirical observation motivated the investigation of detection, measurement, and modeling techniques to exploit these low-dimensional intrinsic structures, yielding numerous implications for high-dimensional statistics, machine learning, and signal processing. Manifold learning (where the low-dimensional structure is a manifold) and dictionary learning (where the low-dimensional structure is the set of sparse linear combinations of vectors from a finite dictionary) are two prominent theoretical and computational frameworks in this area. Despite their ostensible distinction, the recently-introduced Geometric Multi-Resolution Analysis (GMRA) provides a robust, computationally efficient, multiscale procedure for simultaneously learning manifolds and dictionaries. \\
\hspace*{10pt}In this work, we prove non-asymptotic probabilistic bounds on the approximation error of GMRA for a rich class of data-generating statistical models that includes ``noisy'' manifolds, thereby establishing the theoretical robustness of the procedure and confirming empirical observations. In particular, if a dataset aggregates near a low-dimensional manifold, our results show that the approximation error of the GMRA is completely independent of the ambient dimension. Our work therefore establishes GMRA as a provably fast algorithm for dictionary learning with approximation and sparsity guarantees. We include several numerical experiments confirming these theoretical results, and our theoretical framework provides new tools for assessing the behavior of manifold learning and dictionary learning procedures on a large class of interesting models.
\end{abstract}


%In that work, authors exhibited asymptotic error bounds for the GMRA procedure contingent on a strict manifold model for the data, and empirically established the robustness of the procedure.

\tableofcontents

\section{Introduction}
%!TEX root = JMLR_revised.tex

In many high-dimensional data analysis problems, existence of {\em{efficient data representations}} can dramatically boost the statistical performance and the computational efficiency of learning algorithms.
Inversely, in the absence of efficient representations, the curse of dimensionality implies that required sample sizes must grow exponentially with the ambient dimension, which ostensibly renders many statistical learning tasks completely untenable. Parametric statistical modeling seeks to resolve this difficulty by restricting the family of candidate distributions for the data to a collection of probability measures indexed by a finite-dimensional parameter. By contrast, nonparametric statistical models are more flexible and oftentimes more precise, but usually require data samples of large sizes unless the data exhibits some simple latent structure (e.g., some form of sparsity). Such structural considerations are essential for establishing convergence rates, and oftentimes these structural considerations are geometric in nature.

One classical geometric assumption asserts that the data, modeled as a set of points in $\mathbb{R}^D$, in fact lies on (or perhaps very close to) {\em{a single $d$-dimensional affine subspace}} $V\in\mathbb{R}^D$ where $d\ll D$. 
Tools such as PCA \citep[see][]{Pearson_PCA,Hotelling_PCA1,Hotelling_PCA2} estimate $V$ in a stable fashion under suitable assumptions.
Generalizing this model, one may assert that the data lies on a union of several low-dimensional affine subspaces instead of just one, and in this case the estimation of the {\em multiple affine subspaces} from data samples already inspired intensive research due to its subtle complexity 
\citep[e.g., see][]{MSL-SK-IEICE04,GPCA-VMS-PAMI05,LSA-YP-ECCV06,CM:CVPR2011,
GPCA-MYDF-SiamRev08,spectral_theory,SSC-EV-CVPR09,LBF-ZSWL-CVPR10,LRR-LLY-ICML10,ALC-MDHW-PAMI07,RANSAC-FB-ACM81,EM-TB-NeuComp99,KS-Ho-CVPR03}.
A widely used form of this model is that of $k$-sparse data, where there exists a dictionary (i.e., a collection of vectors) $\Phi=\{\varphi_i\}_{i=1}^m\subset\mathbb{R}^D$ such that each observed data point $x\in\mathbb{R}^d$ may be expressed as a linear combination of at most $k\ll D$ elements of $\Phi$. 
These \textit{sparse representations} offer great convenience and expressivity for signal processing tasks \citep[such as in][]{doi:10.1117/12.731851,peyre:sparsetexture}, compressive sensing, statistical estimation, and learning \citep[e.g., see][among others]{Lewicki98learningovercomplete,Kreutz-Delgado:2003:DLA:643335.643340,DBLP:journals/tit/MaurerP10,chen:33,DD:CompressedSensing,Aharon05ksvd,Tao:DantzigEstimator}, and even exhibits connections with representations in the visual cortex \citep[see][]{OF:SparseCodingV1}. 
In geometric terminology, such sparse representations are generally attainable when the local \textit{intrinsic dimension} of the observations is small. 
For these applications, the dictionary is usually assumed to be known a priori, instead of being learned from the data, but it has been recognized in the past decade that data-dependent dictionaries may perform significantly better than generic dictionaries even in classical signal processing tasks. 

The $k$-sparse data model motivates a large amount of research in dictionary learning, where $\Phi$ is learned from data rather than being fixed in advance: given $n$ samples $X_1,\dots,X_n$ from a probability distribution $\mu$ in $\mathbb{R}^D$ representing the training data, an algorithm ``learns'' a dictionary $\widehat{\Phi}$ which provides sparse representations for the observations sampled from $\mu$. 
This problem and its optimal algorithmic solutions are far from being well-understood, at least compared to the understanding that we have for classical dictionaries such as Fourier, wavelets, curvelets, and shearlets. These dictionaries arise in computational harmonic analysis approaches to image processing, and \cite{Donoho99wedgelets:nearly-minimax} (for example) provides rigorous, optimal approximation results for simple classes of images.
The work of \citet{gribonval:hal-00918142} present general bounds for the complexity of learning the dictionaries 
\citep[see also][and references therein]{Vainsencher:2011:SCD:1953048.2078210,DBLP:journals/tit/MaurerP10}. 
The algorithms used in dictionary learning are often computationally demanding, and many of them are based on high-dimensional non-convex optimization \citep{Mairal:OnlineLearningSparseCoding}.
The emphasis of existing work is often made on the generality of the approach, where minimal assumptions are made on geometry of the distribution from which the sample is generated. These ``pessimistic'' techniques incur bounds dependent upon the ambient dimension $D$ in general (even in the standard case of data lying on one hyperplane). 

A different type of geometric assumption on the data gives rise to manifold learning, where the observations aggregate on a {\em suitably regular manifold $\mathcal{M}$} of dimension $d$ isometrically embedded in $\mathbb{R}^D$ 
\citep[notable works include][among others]{isomap,RSLLE,BN,DoGri:WhenDoesIsoMap,DG_HessianEigenmaps,ZhaZha,DiffusionPNAS,DiffusionPNAS2,jms:UniformizationEigenfunctions,jms:UniformizationEigenfunctions2,CMDiffusionWavelets,Wasserman:ManifoldEstimationHausdorff,MM:MultiscaleDimensionalityEstimationAAAI,LMR:MGM1,TestingManifoldHyp}. This setting has been recognized as useful in a variety of applications \citep[e.g.][]{CLMKSWZ:GeometrySensorOutputs,MM:EEG,Donoho:MultiscaleManifoldValued}, influencing work in the applied mathematics and machine learning communities during the past several years. 
It has also been recognized that in many cases the data does not naturally aggregate on a smooth manifold \citep[as in][]{Wakin:MultiscaleStructureNonDifferentiableImageManifolds,MM:MultiscaleDimensionalityEstimationAAAI,LMR:MGM1}, with examples arising in imaging that contradict the smoothness conditions.
While this phenomenon is not as widely recognized as it probably could be, we believe that it is crucial to develop methods (both for dictionary and manifold learning) that are robust not only to noise, but also to modeling error. 
Such concerns motivated the work on intrinsic dimension estimation of noisy data sets \citep[see][]{MM:MultiscaleDimensionalityEstimationAAAI,LMR:MGM1}, where smoothness of the underlying distribution of the data is not assumed, but only certain natural conditions (possibly varying with the scale of the data) are imposed. 
The central idea of the aforementioned works is to perform the multiscale singular value decomposition (SVD) of the data, an approach inspired by the works of \citet{DS} and \citet{Jones-TSP} in classical geometric measure theory. 
These techniques were further extended in several directions in the papers by \citet{MM:WaveletsMultiscale2010,FFT2011,CM:CVPR2011}, while \citet{CM:geometricwaveletsciss,CM:MGM2} built upon this work to construct multiscale dictionaries for the data based on the idea of Geometric Multi-Resolution Analysis (GMRA). 

Until these recent works introduced the GMRA construction, connections between dictionary learning and manifold learning had not garnered much attention in the literature. These papers showed that, for intrinsically low-dimensional data, one may perform dictionary learning very efficiently by exploiting the underlying geometry, thereby illuminating the relationship between manifold learning and dictionary learning.
In these papers, it was demonstrated that, in the infinite sample limit and under a manifold model assumption for the distribution of the data (with mild regularity conditions for the manifold), the GMRA algorithm efficiently learns a dictionary in which the data admits sparse representations. 
More interestingly, the examples in that paper show that the GMRA construction succeeds on real-world data sets which do not admit a structure consistent with the smooth manifold modeling assumption, suggesting that the GMRA construction exhibits robustness to modeling error. 
This desirable behavior follows naturally from design decisions; GMRA combines two elements that add stability: a multiscale decomposition and localized SVD.
Similar ideas appeared in work applying dictionary learning to computer vision problems, for example in the paper by \cite{YuLCC}, where local linear approximations are used to create dictionaries. These techniques appeared at roughly the same time as GMRA (\cite{CM:geometricwaveletsciss}), but were not multiscale in nature, and the selection of the local scale is crucial in applications. These techniques also lacked any finite or infinite sample guarantees, nor considered the effect of noise. They were however successfully applied in computer vision problems, most notably in the Pascal 2007 challenge.

In this paper, we analyze the finite sample behavior of (a slightly modified version of) that construction, and prove strong finite-sample guarantees for its behavior under general conditions on the geometry of a probability distribution generating the data. In particular, we show that these conditions are satisfied when the probability distribution is concentrated ``near'' a manifold, which robustly accounts for noise and modeling errors. In contrast to the pessimistic bounds mentioned above, the bounds that we prove only depend on the ``intrinsic dimension'' of the data. It should be noted that our method of proof produces non-asymptotic bounds, and requires several explicit geometric arguments not previously available in the literature (at least to the best of our knowledge). Some of our geometric bounds could be of independent interest to the manifold learning community. 

The GMRA construction is therefore proven to simultaneously ``learn'' manifolds (in sense that it outputs a suitably close approximation to points on a manifold) and dictionaries in which data are represented sparsely. Moreover, the construction is guaranteed to be robust with respect to noise and to the ``perturbations'' of the manifold model. 
The GMRA construction is fast, linear in the size of the data matrix, inherently online, does not require nonlinear optimization, and is not iterative. 
Finally, our results may be combined with recent GMRA compressed sensing techniques and algorithms presented in \citet{IM:GMRA_CS}, yielding both a method to learn a dictionary in a stable way on a finite set of training data, and a way of performing compressive sensing and reconstruction (with guarantees) from a small number of (suitable) linear projections (again without the need for expensive convex optimization).

This paper is organized as follows: Section \ref{sec:def} introduces the main definitions and notation employed throughout the paper. 
Section \ref{sec:contr} explains the main contributions, formally states the results and provides comparison with existing literature. 
%Section \ref{sec:numerical} contains the summary of numerical analysis based on proposed methods. 
Finally, Sections \ref{sec:prelim} and \ref{sec:proofs} are devoted to the proofs of our main results, Theorem \ref{thm:FSB} and Theorem \ref{thm:MFD}.

\section{Geometric Multi-Resolution Analysis (GMRA)}
\label{sec:def}

This section describes the main results of the paper, starting in a somewhat informal form. The statements will be made precise in the course of the exposition. 
In the statements below, ``$\gtrsim$'' and ``$\lesssim$'' denote inequalities up to multiplicative constants and logarithmic factors. 
\vskip 0.1in
{\noindent \textbf{Statement of results.}} \emph{
%\label{t:informalthm}
Let $\sigma\geq 0$ be a fixed small constant, and let $\eps\gtrsim \sigma$ be given. 
Suppose that $n\gtrsim \eps^{-(1+d/2)}$, and let $\m X_n=\{X_1,\ldots, X_n\}$ be an i.i.d. sample from $\Pi$, a probability distribution with density supported in a tube of radius $\sigma$ around a smooth closed $d$-dimensional manifold $\mathcal{M}\hookrightarrow\mathbb{R}^D$, with $d>1$. There exists an algorithm that, given $\m X_n$, outputs the following objects: 
%(see definition \ref{d:tausigma}). 
\begin{itemize}
\item[\tiny$\bullet$]a dictionary 
$\widehat{\Phi}_\eps=\{\widehat{\varphi}_i\}_{i\in \m J_\eps}\subset\mathbb{R}^D$;
\item[\tiny$\bullet$]  a nonlinear ``encoding" operator 
$\widehat{\mathcal{D}}_\eps:\mathbb{R}^D\rightarrow\mathbb{R}^{\m J_\eps}$ which takes $x\in \mb R^D$ and returns the coefficients of its approximation by the elements of $\widehat{\Phi}_\eps$;
\item[\tiny$\bullet$]  a ``decoding" operator $\widehat{\mathcal{D}}_\eps^{-1}:\mathbb{R}^{\m J_\eps}\rightarrow\mathbb{R}^D$ which maps a sequence of coefficients to an element of $\mb R^D$. 
\end{itemize}
Moreover, the following properties hold with high probability:
\begin{enumerate}
\item[i.] $\card(\m J_\eps)\lesssim\eps^{-d/2}$;
\item[ii.] the image of $\widehat{\m D}_\eps$ is contained in the set $S_{d+1}\subset \mb R^{\m J_\eps}$ of all $(d+1)$ - sparse vectors (i.e., vectors with at most $d+1$ nonzero coordinates);
\item[iii.]  the reconstruction error satisfies
\begin{equation*}
\begin{aligned}
\sup_{x\in{\rm support}(\Pi)}\|x-\widehat{\mathcal{D}}_\eps^{-1}\widehat{\mathcal{D}}_\eps(x)\|&\lesssim \eps;
\end{aligned}
\end{equation*}
\item[iv.] the time complexity for computing
\begin{itemize}
\item[$\bullet$] $\widehat{\Phi}_\eps$ is $O(C^d (D+d^2)\eps^{-(1+\frac d2)}\log(1/\eps))$, where $C$ is a universal constant;
\item[$\bullet$] $\widehat{\mathcal{D}}_\eps(x)$ is $O(d(D+d\log(1/\eps)))$, and for $\widehat{\mathcal{D}}_\eps^{-1}(x)$ is $O(d(D+\log(1/\eps)))$.
\end{itemize}
If a new observation $X_{n+1}$ from $\Pi$ becomes available, $\widehat{\Phi}_\eps$ may be updated in time $O(C^d (D+d^2)\log(1/\eps))$.
\end{enumerate}
}

In other words, we can construct a data-dependent dictionary $\widehat{\Phi}_\eps$ of cardinality $O(\eps^{-d/2})$ by looking at $O(\eps^{-1-\frac{d}{2}})$ data points drawn from $\Pi$, such that $\widehat{\Phi}_\eps$ provides both $(d+1)$-sparse approximations to data and has expected ``reconstruction error'' of order $\eps$ (with high probability). 
Note that the cost of encoding the $(d+1)$ non-zero coefficients requires $O((d+1)\log(\card(\m J_\eps)))=O(d^2\log(1/\epsilon))$.
Moreover, the algorithm producing this dictionary is fast and can be quickly updated if new points become available. 
We want to emphasize that the complexity of our construction only depends on the desired accuracy $\eps$, and is independent of the total number of samples (for example, it is enough to use only the first $\simeq \eps^{-(1+d/2)}$ data points). Many existing techniques in dictionary learning cannot guarantee a requested accuracy, or a given sparsity, and a certain computational cost as a function of the two. Our results above completely characterize the tradeoffs between desired precision, dictionary size, sparsity, and computational complexity for our dictionary learning procedure.
%\Stas{I would skip the following..}: 
%at most linear in $n$ if $n=O(\eps^{-1-\frac2d})$ but in general sub linear in $n$ and only dependent on the requested precision, and so is the application of the map from data points to coefficients and its (approximate) inverse. Finally, the dictionary is quickly updated if new points become available.\

We also remark that a suitable version of compressed sensing applies to the dictionary representations used in the theorem: we refer the reader to the work by \citet{IM:GMRA_CS}, and its applications to hyperspectral imaging by \citet{6410789}.

\subsection{Notation} 

For $v\in\mb R^D$, $\|v\|$ denotes the standard Euclidean norm in $\mb R^D$.
$B_d(0,r)$ is the Euclidean ball in $\mb R^d$ of radius $r$ centered at the origin, and we let $B(0,r):=B_D(0,r)$. 
$\proj_V$ stands for the orthogonal projection onto a linear subspace $V\subseteq \mb R^D$, $\dim(V)$ for its dimension and $V^\perp$ for its orthogonal complement.
For $x\in \mb R^D$, let $\proj_{x+V}$ be the affine projection onto the affine subspace $x+V$ defined by
$\proj_{x+V}(y)=x+\proj_V(y-x)$, for $y\in\mb R^D$.

Given a matrix $A\in \mb R^{k\times l}$, we write $A=[a_1|\cdots |a_l]$, where $a_i$ stands for the $i$th column of $A$. 
The operator norm is denoted by $\| A \|$, the Frobenius norm by $\|A\|_F$ and the matrix transpose by $A^T$. 
If $k=l$, $\tr (A)$ denotes the trace.  
For $v\in \mb R^k$, let $\diag(v)$ be the $k\times k$ diagonal matrix with $(\diag(v))_{ii}=v_i, \ i=1,\ldots,k$. 
Finally, we use ${\rm span}\{a_i\}_{i=1}^l$ to denote the linear span of the columns of $A$. 

Given a $C^2$ function $f:\mathbb{R}^l\rightarrow\mathbb{R}^k$, let $f_i$ denote the $i$th coordinate of the function $f$ for $i=1,\ldots k$, 
$Df(v)$ the Jacobian of $f$ at $v\in\mathbb{R}^l$, and $D^2 f_i(v)$ the Hessian of the $i$th coordinate at $v$. 

We shall use $d\Vol$ to denote Lebesgue measure on $\mathbb{R}^D$, and if $U\subset\mathbb{R}^D$ is Lebesgue measurable, $\Vol(U)$ stands for the Lebesgue measure of $U$. 
We will use $\vol$ to denote the volume measure on a $d$-dimensional manifold $\m M$ in $\mb R^D$ (note that this coincides with the $d$-dimensional Hausdorff measure for the subset $\m M$ of $\mb R^D$), $U_\m M$ - the uniform distribution over $\m M$, and $d_{\m M}(x,y)$ to denote the geodesic distance between two points $x,y\in\m M$. 
For a probability measure $\Pi$ on $\mb R^D$,
$
\supp(\Pi):=\cap_{C \text{ closed},\Pi(C)=1}C
$
stands for its support. Finally, for $x,y\in\mb R$, $x\vee y:=\max(x,y)$.

\subsection{Definition of the geometric multi-resolution analysis (GMRA)}
We assume that the data are identically, independently distributed samples from a Borel probability measure $\Pi$ on $\mb R^D$. 
Let $1\leq d\leq D$ be an integer. 
A GMRA with respect to the probability measure $\Pi$ consists of a collection of (nonlinear) operators $\{P_j: \mb R^D\to\mb R^D\}_{j\ge 0}$. 
For each ``resolution level'' $j\geq 0$, $P_j$ is uniquely defined by a collection of pairs of subsets and affine projections, $\{(C_{j,k},P_{j,k})\}_{k=1}^{N(j)}$, where the subsets $\{C_{j,k}\}_{k=1}^{N(j)}$ form a measurable partition of $\mb R^D$ (that is, members of $\{C_{j,k}\}_{k=1}^{N(j)}$ are pairwise disjoint and the union of all members is $\mb R^D$). 
$P_j$ is constructed by piecing together local affine projections. Namely, let
$$
P_{j,k}(x) := c_{j,k}+\proj{}_{V_{j,k}} (x-c_{j,k}),
$$
where $c_{j,k}\in \mb R^D$ and $V_{j,k}$ are defined as follows. 
Let $\mb E_{j,k}$ stand for the expectation with respect to the conditional distribution $d\Pi_{j,k}(x)=d\Pi(x|x\in C_{j,k})$. 
Then
% ($\Pi$ conditioned on the event $C_{j,k}$)
\begin{align}
\label{eq:a10}
c_{j,k}&=\mb E_{j,k}x,\\
\label{eq:a20}
V_{j,k}&=\argmin_{\dim(V)=d} \mb E_{j,k}\left\|x-c_{j,k}-\proj {}_V(x-c_{j,k})\right\|^2,
\end{align}
where the minimum is taken over all linear spaces $V$ of dimension $d$. 
In other words, $c_{j,k}$ is the conditional mean and $V_{j,k}$ is the subspace spanned by eigenvectors corresponding to $d$ largest eigenvalues of the conditional covariance matrix
\begin{align}
\label{eq:a25}
\Sigma_{j,k}=\mb E_{j,k}[(x-c_{j,k})(x-c_{j,k})^T]\,. 
\end{align}
Note that we have implicitly assumed that such a subspace $V_{j,k}$ is unique, which will always be the case throughout this paper. Given such a $\{(C_{j,k},P_{j,k})\}_{k=1}^{N(j)}$, we define
$$
P_j(x):=\sum_{k=1}^{N(j)} I\{x\in C_{j,k}\} P_{j,k}(x)
$$
where $I\{x\in C_{j,k}\}$ is the indicator function of the set $C_{j,k}$. 

It was shown in the paper by \cite{CM:MGM2} that if $\Pi$ is supported on a smooth, closed $d$-dimensional submanifold $\m M\hookrightarrow \mb R^D$, and if the partitions $\{C_{j,k}\}_{j=1}^{N(j)}$ satisfy some regularity conditions for each $j$, then, for any $x\in \m M$, $\|x-P_j(x)\|\leq C(\m M)2^{-2j}$ for all $j\geq j_0(\m M)$. 
This means that the operators $P_j$ provide an efficient ``compression scheme'' $x\mapsto P_j(x)$ for $x\in \m M$, in the sense that every $x$ can be well-approximated by a linear combination of at most $d+1$ vectors from the dictionary $\Phi_{2^{-2j}}$ formed by $\{c_{j,k}\}_{k=1}^{N(j)}$ and the union of the bases of $V_{j,k}, \ k=1\ldots N(j)$. 
Furthermore, operators efficiently encoding the ``difference'' between $P_j$ and $P_{j+1}$ were constructed, leading to a multiscale compressible representation of $\mathcal{M}$.

In practice, $\Pi$ is unknown and we only have access to the \textit{training data} $\m X_n=\{X_1,\ldots,X_n\}$, which are assumed to be i.i.d. with distribution $\Pi$. In this case, operators $P_j$ are replaced by their estimators 
$$
\widehat{P}_j(x) : = \sum_{k=1}^{N(j)} I\{x\in C_{j,k}\} \widehat{P}_{j,k}(x)
$$
where $\{C_{j,k}\}_{k=1}^{N(j)}$ is a suitable partition of $\mb R^D$ {\em{obtained from the data}},
\begin{align}
\label{eq:a30}
&\widehat{P}_{j,k}(x):= \widehat{c}_{j,k}+\proj{}_{\widehat{V}_{j,k}}(x-\widehat{c}_{j,k}),\\
\nonumber
&\widehat{c}_{j,k} := \frac{1}{\vert\m X_{j,k}\vert}\sum_{x\in \m X_{j,k}}x,\\
\nonumber
&\widehat{V}_{j,k} :=\argmin_{\dim(V)=d} \frac{1}{\vert \m X_{j,k}\vert}\sum_{x\in\m X_{j,k}}\left\|x-\widehat{c}_{j,k}-\proj {}_V(x-\widehat{c}_{j,k})\right\|^2,
\end{align}
$\m X_{j,k}=C_{j,k}\cap \m X_n$, and $\vert \m X_{j,k}\vert$ denotes the number of elements in $\m X_{j,k}$. We shall call these $\widehat{P}_j$ the {\it empirical GMRA}. 

Moreover, the dictionary $\widehat \Phi_{2^{-2j}}$ is formed by $\{\hat c_{j,k}\}_{k=1}^{N(j)}$ and the union of bases of 
$\hat V_{j,k}, \ k=1\ldots N(j)$. 
The ``encoding'' and ``decoding'' operators $\widehat{\m  D}_{2^{-2j}}$ and $\widehat{\m D}_{2^{-2j}}^{-1}$ mentioned above are now defined in the obvious way, so that $\widehat{\m D}_{2^{-2j}}^{-1}\widehat{\m D}_{2^{-2j}}(x)=\widehat{P}_{j,k}(x)$ for any $x\in C_{j,k}$.

We remark that the ``intrinsic dimension'' $d$ is assumed to be known throughout this paper. 
In practice, it can be estimated within the GMRA construction using the ``multiscale SVD'' ideas of \cite{MM:MultiscaleDimensionalityEstimationAAAI,LMR:MGM1}. 
The estimation technique is based on inspecting (for a given point $x\in C_{j,k}$) the behavior of the singular values of the covariance matrix $\Sigma_{j,k}$ as $j$ varies. 
For alternative methods, see \citet[][]{Levina2004Maximum-likelih00,camastra2001intrinsic} and references therein and in the review section of \cite{LMR:MGM1}.

\section{Main results}
\label{sec:contr}

Our main goal is to obtain \textit{probabilistic}, non-asymptotic bounds on the performance of the \textit{empirical} GMRA under certain structural assumptions on the underlying distribution of the data. 
In practice, the data rarely belongs precisely to a smooth low-dimensional submanifold. 
One way to relax this condition is to assume that it is ``sufficiently close'' to a reasonably regular set.  
Here we assume that the underlying distribution is supported in a thin tube around a manifold. We may interpret the displacement from the manifold as noise, in which case we are making no assumption on distribution of the noise besides boundedness. 
Another way to model this situation is to allow \textit{additive noise}, whence the observations are assumed to be of the form $X=Y+\xi$, where
$Y$ belongs to a submanifold of $\mb R^D$, $\xi$ is independent of $Y$, and the distribution of $\xi$ is  known. 
This leads to a singular deconvolution problem \citep[see][]{KoltchinskiiGeometry,Wasserman:ManifoldEstimationHausdorff}.
Our assumptions however may also be interpreted as relaxing the ``manifold assumption'': even in the absence of noise we do allow data to be not exactly supported on a manifold. Our results will elucidate how the error of sparse approximation via GMRA depends on the ``thickness'' of the tube, which quantifies stability and robustness properties of our algorithm.

As we mentioned before, our GMRA construction is entirely data-dependent: it takes the point cloud of cardinality $n$ as an input and for every 
$j\in \mb Z_+$ returns the partition $\{C_{j,k}\}_{k=1}^{N(j)}$ and associated affine projectors $\widehat{P}_{j,k}$.
We will measure performance of the empirical GMRA by the $L_2(\Pi)$-error 
\begin{equation}
\mb E\left\|X-\widehat P_j(X)\right\|^2:=\int\limits_{\supp(\Pi)}\left\|x-\widehat P_j(x)\right\|^2 d\Pi(x)
\label{e:errdef}
\end{equation}
or by the $\|\cdot\|_{\infty,\Pi}$-error defined as
\begin{align}
\label{eq:sup-norm}
&
\left\|\Id-\widehat P_j\right\|_{\infty,\Pi}:=\sup_{x\in \supp(\Pi)}\left\|x-\widehat P_j(x)\right\|,
\end{align}
where $\widehat P_j$ is defined by (\ref{eq:a30}). Note, in particular, that these errors are ``out-of-sample'', i.e. measure the accuracy of the GMRA representations on all possible samples, not just those used to train the GMRA, which would not correspond to a learning problem.
%\Nate{I understand what is meant by the sup norm error here, but the ``$x$" place holder in the notation is confusing because we are essentially considering an operator norm. Instead, consider
%$$
%\left\| Id-\widehat P_j\right\|_{\Pi}:=\sup_{x\in \supp(\Pi)}\left\|x-\widehat P_j(x)\right\|.
%$$}
%\Mauro{I agree with Nate's comment but fine with either notation; still I think we need an $\infty$ there, so perhaps
%$$
%\left\| Id-\widehat P_j\right\|_{L^\infty(\Pi)}:=\sup_{x\in \supp(\Pi)}\left\|x-\widehat P_j(x)\right\|.
%$$
%which is in fact in a way a standard notation, since $Id$ and $\hat P_j$ are indeed functions on $\Pi$.}
%\Stas{As far as I remember, $L_\infty(\Pi)$ is a slightly different norm defined through ${\rm ess\,sup}$, so I used a third variant.}

The presentation is structured as follows: we start from the natural decomposition
\begin{equation*}
\left\|x-\widehat P_j(x)\right\|\le \underbrace{\left\|x- P_j(x)\right\|}_{\text{approximation error}}+\underbrace{\left\|P_j(x)-\widehat P_j(x)\right\|}_{\text{random error}}
\end{equation*}
and state the general conditions on the underlying distribution and partition scheme that suffice to guarantee that 
\begin{enumerate}
\item the distribution-dependent operators $P_j$ yield good approximation, as measured by $\mb E\left\|x-P_j(x)\right\|^2$: this is the bias (squared) term, which is non-random; 
\item the empirical version $\widehat{P}_j$ is with high probability close to $P_j$, so that $\mb E\left\|\widehat{P}_j(x)-P_j(x)\right\|^2$ is small (with high probability): this is the variance term, which is random.
\end{enumerate}
This leads to our first result, Theorem \ref{thm:FSB}, where the error $\mb E\left\|x-\widehat P_j(x)\right\|^2$ of the empirical GMRA is bounded with high probability.

We will state this first result in a rather general setting (assumptions A1-A4) below), and after developing this general result, we consider the special but important case where the distribution $\Pi$ generating the data is supported in thin tube around a smooth submanifold, and
for a (concrete, efficiently computable, online) partition scheme we show that the conditions of Theorem \ref{thm:FSB} are satisfied.
This is summarized in the statement of Theorem \ref{thm:MFD}, that may be interpreted as proving finite-sample bounds for our GMRA-based dictionary learning scheme for high-dimensional data that suitably concentrates around a manifold.
It is important to note that most of the constants in our results are explicit. 
The only geometric parameters involved in the bounds are the dimension $d$ of the manifold (but not the ambient dimension $D$), its \textit{reach} (see  $\tau$ in (\ref{eq:reach1})) and the ``tube thickness'' $\sigma$. 

Among the existing literature, the papers \cite{CM:MGM2,6410789} introduced the idea of using multiscale geometric decomposition of data to estimate the distribution of points sampled in high-dimensions. However in the first paper no finite sample analysis was performed, and in the second the connection with geometric properties of the distribution of the data is not made explicit, with the conditions are expressed in terms of certain approximation spaces within the space of probability distributions in $\mathbb{R}^D$, with Wasserstein metrics used to measure distances and approximation errors. 

The recent paper by \citet{Canas2012Learning-Manifo00} is close in scope to our work; its authors present probabilistic guarantees for approximating a manifold with a global solution of the so-called $k$-flats \citep{Bradley2000K-Plane-cluster00} problem in the case of distributions supported on manifolds. 
It is important to note, however, that our estimator is explicitly and efficiently computable, while exact global solution of $k$-flats is usually unavailable and certain approximations have to be used in practice, with convergence to a global minimum is conditioned on suitable unknown initializations.  In practice it is often reported that there exist many local minima with very different values, and good initializations are not trivial to find. In this work we obtain better convergence rates, with fast algorithms, and we also seamlessly tackle the case of noise and model error, which is beyond what was studied previously. We consider this development extremely relevant in applications, both because real data is corrupted by noise and the assumption that data lies exactly on a  smooth manifold is often unrealistic.
A more detailed comparison of theoretical guarantees for $k$-flats and for our approach is given after we state the main results in Subsection \ref{sec:case} below. 

Another body of literature connected to this work studies the complexity of dictionary learning. 
For example, \citet{gribonval:hal-00918142} present general bounds for the convergence of global minimums of empirical risks in dictionary learning optimization problems (those results build on and generalize the works of \citet{DBLP:journals/tit/MaurerP10,Vainsencher:2011:SCD:1953048.2078210}, among several others). While the rates obtained in those works seem to be competitive with our rates in certain regimes, the fact that their bounds must hold over entire families of dictionaries implies that those error rates generally involve a scaling constant of the order $\sqrt{Dk}$, where $D$ is the ambient dimension and $k$ is the number of ``atoms" in the dictionary. Our bounds are independent of the ambient dimension $D$ but implicitly include terms which depend upon the number of our ``atoms.'' It should be noted that the number of atoms in the dictionary learned by GMRA increase so as to approximate the fine structure of a dataset with more precision. As such, our attainment of the minimax lower bounds for manifold estimation in the Hausdorff metric (obtained in \citep{Genovese:2012:MME:2503308.2343687}) should be expected. While dictionaries produced from dictionary learning should reveal the fine structure of a dataset through careful examination of the representations they induce, these representations are often ambiguous unless additional structure is imposed on both the dictionaries and the datasets. On the other hand, the GMRA construction induces completely unambiguous sparse representations that can be used in regression and classification tasks with confidence. 
% and yield rates of approximation optimal in a minimax sense.

In the course of the proof, we obtain several results that should be of independent interest. 
In particular, Lemma \ref{lem:volUB} gives upper and lower bounds for the volume of the tube around a manifold in terms of the reach (\ref{eq:reach}) and tube thickness. While the exact tubular volumes are given by Weyl's tube formula \citep[see][]{Gray2004Tubes00}, our bound are exceedingly easy to  state in terms of simple global geometric parameters. 

For the details on numerical implementation of GMRA and its modifications, see the works by \citet{CM:MGM2,CM:geometricwaveletsciss}. 

\subsection{Finite sample bounds for empirical GMRA}

In this section, we shall present the finite sample bounds for the empirical GMRA described above. 
For a fixed resolution level $j$, we first state sufficient conditions on the distribution $\Pi$ and the partition $\{C_{j,k}\}_{k=1}^{N(j)}$ for which these $L_2(\Pi)$-error bounds hold (see Theorem \ref{thm:FSB} below). 

%\Stas{I would replace the following phrase by a remark giving the intuition behind the assumptions..}
%Each of these conditions indicates a level of control over the concentration of $\Pi$ on each $C_{j,k}$.
Suppose that for all integers $j_{\min} \leq j\leq j_{\max}$ the following is true:
%(\Stas{it is important that the range of $j$ is bounded - in the corollary for manifold, $j_{\min}$ and $j_{\max}$ are specified precisely}):
\begin{description}
\item[(A1)] There exists an integer $1\leq d\leq D$ and a positive constant $\theta_1=\theta_1(\Pi)$ such that for all $k=1,\ldots, N(j)$,
%\leq \theta_2=\theta_2(\Pi)$
$$
\Pi(C_{j,k})\geq \theta_1 2^{-jd}\,.
%\leq \theta_2 2^{-jd}.
$$
\item[(A2)] There is a positive constant $\theta_2=\theta_2(\Pi)$ such that for all $k=1,\ldots, N(j)$, if $X$ is drawn from $\Pi_{j,k}$ then, $\Pi$ - almost surely,
$$
\Vert X - c_{j,k}\Vert\leq \theta_2 2^{-j}\,.
$$
\item[(A3)] Let $\lambda^{j,k}_1\ge\ldots\ge\lambda^{j,k}_D\ge0$ denote the eigenvalues of the covariance matrix $\Sigma_{j,k}$ (defined in (\ref{eq:a25})). 
Then there exist $\sigma=\sigma(\Pi)\ge0$, $\theta_3=\theta_3(\Pi)$, $\theta_4=\theta_4(\Pi)>0$, and some $\alpha>0$ such that for all $k=1\ldots N(j)$, 
\begin{align*}
\lambda^{j,k}_d&\geq \theta_3 \frac{2^{-2j}}{d} \quad\mbox{ and }\quad
\sum\limits_{l=d+1}^D\lambda^{j,k}_l \leq \theta_4(\sigma^2+2^{-2(1+\alpha)j})\leq \frac{1}{2}\lambda^{j,k}_d.
\end{align*}
\end{description}
If in addition
\begin{description}
\item[(A4)]
There exists $\theta_5=\theta_5(\Pi)$ such that 
%for all $k=1\ldots N(j)$
\begin{align*}
&
\left\|\Id-P_j\right\|_{\infty,\Pi}\leq \theta_5\left(\sigma+2^{-(1+\alpha)j}\right),
%\left(1+\theta_6 2^{-j}\right),
\end{align*}
\end{description}
then the bounds are also guaranteed to hold for the $\|\cdot\|_{\infty,\Pi}$-error (\ref{eq:sup-norm}).
\begin{remark}
\label{rem:1}\hfill \\
\begin{enumerate}
\item[i.] Assumption {\em(A1)} entails that the distribution assigns a reasonable amount of probability to each partition element, assumption {\em(A2)} ensures that samples from partition elements are always within a ball around the centroid, and assumption {\em(A3)} controls the effective dimensionality of the samples within each partition element. Assumption {\em(A4)} just assumes a bound on the error for the theoretical GMRA reconstruction.
\item[ii.] Note that the constants $\theta_i, \ i=1\ldots 4$, are independent of the resolution level $j$. 
\item[iii.] It is easy to see that Assumption (A3) implies a bound on the ``local approximation error'': since $P_j$ acts on $C_{j,k}$ as an affine projection on the first $d$ ``principal components'', we have  
\begin{align*}
\mb E_{j,k}\|x-P_j(x)\|^2&=\tr\left[\mb E_{j,k}\left(x-c_{j,k}-\proj{}_{V_{j,k}}(x))(x-c_{j,k}-\proj{}_{V_{j,k}}(x)\right)^T\right]\\
&
=\sum\limits_{l=d+1}^D\lambda^{j,k}_l\leq \theta_4 (\sigma^2+2^{-2(1+\alpha)j}).
\end{align*}
\item[iv.] The parameter $\sigma$ is introduced to cover ``noisy'' models, including the situations when $\Pi$ is supported in a thin tube of width $\sigma$ around a low-dimensional manifold $\m M$. 
Whenever $\Pi$ is supported on a smooth $d$-dimensional manifold, $\sigma$ can be taken to be $0$.
\item[v.] The stipulation
\begin{align*}
\theta_4(\sigma^2+2^{-2(1+\alpha)j})\leq \frac{1}{2}\lambda^{j,k}_d
\end{align*}
guarantees that the spectral gap $\lambda^{j,k}_d-\lambda^{j,k}_{d+1}$ is sufficiently large.
%is useful for the proof, and may be verified using the assumed lower bound for $\lambda^{j,k}_d$. Of course, this will hold for all sufficiently small $\sigma$ and sufficiently large $j$.
\end{enumerate}
\end{remark}

\noindent We are in position to state our main result. 

\begin{theorem}\label{thm:FSB}
Suppose that {\bf\em(A1)}-{\bf\em(A3)} are satisfied, let $X,X_1,\ldots,X_n$ be an i.i.d. sample from $\Pi$, and set $\bar d:=4d^2\theta_2^4/\theta_3^2$.  
Then for any $j_{\min} \leq j\leq j_{\max}$ and any $t\geq 1$ such that $t+\log(\bar d\vee 8)\leq \frac{1}{2}\theta_1 n 2^{-jd}$,  
\[
\mb E \|X-\widehat P_j(X)\|^2\leq 2\theta_4\left(\sigma^2+2^{-2j(1+\alpha)}\right)+c_1 2^{-2j}\frac{(t+\log(\bar d\vee 8))d^2}{n2^{-jd}},
\]
and if in addition {\bf\em(A4)} is satisfied,
\[
\left\|\Id-\widehat P_j\right\|_{\infty,\Pi}\leq \theta_5\left(\sigma+2^{-(1+\alpha)j}\right)
+\sqrt{\frac{c_1}{2} 2^{-2j}\frac{(t+\log(\bar d\vee 8))d^2}{n2^{-jd}}}
\]
%\int_{\mb R^D}\|x-\widehat P_j(x)\|^2_2 d\Pi(x)
with probability  $\geq 1-\frac{2^{jd+1}}{\theta_1}\l(e^{-t}+e^{-\frac{\theta_1}{16}n2^{-jd}}\r)$, where 
$c_1=2\l(12\sqrt{2}\frac{\theta_2^3}{\theta_3\sqrt{\theta_1}}+4\sqrt{2}\frac{\theta_2}{d\sqrt{\theta_1}}\r)^2$. 
\end{theorem}

\subsection{Distributions concentrated near smooth manifolds}
\label{sec:case}
%Partition guarantees for geometric models

Of course, the statement of Theorem \ref{thm:FSB} has little value unless assumptions {\bf(A1)}-{\bf(A4)} can be verified for a rich class of underlying distributions. 
We now introduce an important class of models and an algorithm to construct suitable partitions $\{C_{j,k}\}$ which together satisfy these assumptions. 
Let $\m M$ be a smooth (or at least $C^2$, so changes of coordinate charts admit continuous second-order derivatives), closed $d$-dimensional submanifold of $\mb R^D$. 
We recall the definition of the \textit{reach} \citep[see][]{federer1959curvature}, an important global characteristic of $\m M$. Let 
\begin{align}
\label{eq:reach}
D(\m M)&=\{y\in \mb R^D: \exists ! x\in \m M \text{ s.t. } \|x-y\|=\inf_{z\in \m M}\|z-y\|\}, 
\\ 
\label{eq:tube}
\m M_r&=\{y\in \mb R^D: \ \inf_{x\in \m M}\|x-y\| < r\}.
\end{align}
Then 
\begin{align}
\label{eq:reach1}
{\rm reach}(\m M):=\sup\{r\geq 0: \ \m M_r\subseteq D(\m M)\},
\end{align} 
and we shall always use $\tau$ to denote the reach of the manifold $\m M$.

\begin{definition}
Assume that $0\leq\sigma < \tau$. 
We shall say that the distribution $\Pi$ satisfies the {\bf{$(\boldsymbol\tau,\boldsymbol\sigma)$-model assumption}} if there exists a smooth (or at least $C^2$), compact submanifold 
$\m M\hookrightarrow \mb R^D$ with reach $\tau$ such that $\supp(\Pi)=\m M_\sigma$, 
$\Pi$ and $\mathcal{U}_{\m M_\sigma}$ (the uniform distribution on $\m M_\sigma$) are absolutely continuous with respect to each other,  and so Radon-Nikodym derivative $\frac{d\Pi}{dU_{\m M_\sigma}}$ satisfies
\begin{align}
\label{eq:tau-sigma}
0<\phi_1\leq \frac{d\Pi}{d \mathcal{U}_{\m M_\sigma}} \leq \phi_2<\infty\qquad \mathcal{U}_{\m M_\sigma} \text{- almost surely}\,.
\end{align} 
%Here, $M_\sigma$ is the tubular neighborhood of $\m M$ with radius $\sigma$ (the Minkowski sum of $\m M$ and the ball of radius $\sigma$ centered at $0$), and $U_{\m M_\sigma}$ is normalized restriction of Lebesgue measure on $\mb R^D$ to $\m M_\sigma$ (the uniform measure on $\m M_\sigma$).  
\label{d:tausigma}
\end{definition}

\begin{example}
Consider the unit sphere of radius $R$ in $\mathbb{R}^D$, $\mathcal{S}_R$. Then $\tau=R$ for this manifold, and for any $\sigma<R$, the uniform distribution on the set $B(0,R+\sigma)\setminus B(0,R-\sigma)$ satisfies the $(\sigma,\tau)$-model assumption. On the other hand, taking the uniform distribution on a $\sigma$-thickening of the union of two line segments emanating from the origin produces a distribution which does not satisfy the $(\sigma,\tau)$ model assumption. In particular, $\tau=0$ for the underlying manifold.
\end{example}

\begin{remark}
We will implicitly assume that constants $\phi_1$ and $\phi_2$ do not depend on the ambient dimension $D$ (or depend on a slowly growing function of $D$, such as $\log D$)  - the bound of Theorem \ref{thm:MFD} shows that this is the ``interesting case''. 
On the other hand, we often do not need the full power of $(\tau,\sigma)$ - model assumption, see the Remark \ref{rmk:tau-sigma} after Theorem \ref{thm:MFD}.
%This model avoids the somewhat more complicated cases in which the noise level is ``non-uniform'' along the manifold or the possibility that the distribution might have unbounded support. In the first case, it should be possible to weaken the assumptions to consider distributions which are simply absolutely continuous with respect to $\mathcal{U}_{\m M_{\sigma}}$. In the latter case, assuming sufficient decay of the distribution with respect to the distance from the manifold, one may essentially decompose the data into a small number of outliers and a large number of points drawn from a $(\sigma,\tau)$-model.
\end{remark}

Our partitioning scheme is based on the data structure known as the \textit{cover tree} introduced by \citet{LangfordICML06-CoverTree} \citep[see also][]{Karger:2002:FNN:509907.510013,Yianilos:1993:DSA:313559.313789,Ciaccia97indexingmetric}.
We briefly recall its definition and basic properties. 
Given a set of $n$ distinct points $S_n=\{x_1,\ldots,x_n\}$ in some metric space $(S,\rho)$, the cover tree $T$ on $S_n$ satisfies the following: 
let $T_j\subset S_n, \ j=0,1,2,\ldots$ be the set of nodes of $T$ at level $j$. 
%(following the standard cover tree terminology, the levels are indexed by negative integers). 
Then
\begin{enumerate}
\item $T_j\subset T_{j+1}$;
\item for all $y\in T_{j+1}$, there exists $z\in T_j$ such that $\rho(y,z)<2^{-j}$;
\item for all $y,z\in T_j$, $\rho(y,z)>2^{-j}$.
\end{enumerate}
\begin{remark}
\label{rem:2}
Note that these properties imply the following: for any $y\in S_n$, there exists $z\in T_j$ such that $\rho(y,z)<2^{-j+1}$.
\end{remark}
Theorem 3 in \citep{LangfordICML06-CoverTree} shows that the cover tree always exists; for more details, see the aforementioned paper.
% and may be constructed in an online fashion in time $O(C^d Dn\log n)$ on metric spaces with doubling dimension $d$, with $D$ the cost of computing one pairwise distance between two points in the space.  
%The role of the algorithm in the context of the computational aspects of GMRA is discussed in \cite{CM:MGM2}.

We will construct a cover tree for the collection $X_1,\ldots,X_n$ of i.i.d. samples from the distribution $\Pi$ with respect to the Euclidean distance 
$\rho(x,y):=\|x-y\|$. 
Assume that $T_j:=T_j(X_1,\ldots,X_n)=\{a_{j,k}\}_{k=1}^{N(j)}$. 
Define the indexing map
\[
k(x):=\argmin_{1\leq k\leq N(j)}\|x-a_{j,k}\|
\]
(ties are broken by choosing the smallest value of  $k$), 
and partition $\mb R^D$ into the Voronoi regions
%$$
%k_j(x) = \min \left\{ 1\leq k\leq N(j) \left\vert \Vert x-a_{j,k}\Vert_2 = \min_{1\leq k^\prime \leq N(j)}\Vert x-a_{j,k^\prime}\Vert_2\right.\right\}
%$$
\begin{align}
\label{eq:part}
C_{j,k} = \{x\in\mb R^D: k_j(x)=k\}.
\end{align}
%#################################################
Let $\eps(n,t)$ be the smallest $\eps>0$ which satisfies
\begin{align}
\label{eq:eps}
n\geq \frac{1}{\phi_1}\l(\frac{\tau+\sigma}{\tau-\sigma}\r)^d \beta_1\l(\log \beta_2+t\r),
\end{align}
where $\beta_1= \frac{\vol(\m M)}{\cos^d(\delta_1)\Vol(B_d(0,\eps/4))}$, $\beta_2= \frac{\vol(\m M)}{\cos^d(\delta_2)\Vol(B_d(0,\eps/8))}$, $\delta_1=\arcsin(\eps/8\tau)$, and $\delta_2=\arcsin(\eps/16\tau)$. 

\begin{remark}
For large enough $n$, this requirement translates into $n\geq C(\m M,d,\phi_1)\l(\frac 1 \eps\r)^d \left(\log \frac 1 \eps+t\right)$ for some constant $C(\m M,d,\phi_1)$.  
\end{remark}
%#################################################

\noindent We are ready to state the main result of this section. 

\begin{theorem}
\label{thm:MFD}
Suppose that $\Pi$ satisfies the $(\tau,\sigma)$-model assumption. 
Let $X_1,\ldots, X_n$ be an i.i.d. sample from $\Pi$, construct a cover tree $T$ from $\{X_i\}_{i=1}^n$, and define $C_{j,k}$ as in \eqref{eq:part}.
Assume that $\eps(n,t) < \sigma$. 
Then, for all $j\in \mb Z_+$ such that $2^{-j}>8\sigma$ and $3\cdot 2^{-j}+\sigma<\tau/8$, 
%and also that $n$ satisfies
%$$
%n\geq C(\Pi)2^{jd}\log 2^{j}
%$$
the partition $\{C_{j,k}\}_{k=1}^{N(j)}$ and $\Pi$ satisfy {\bf\em(A1)}, {\bf\em(A2)}, {\bf\em(A3)}, and {\bf\em(A4)} with probability $\geq 1-e^{-t}$ for
\begin{align*}
\theta_1&=\frac{\phi_1\Vol(B_d(0,1))}{2^{4d}\vol(\m M)}\l(\frac{\tau-\sigma}{\tau+\sigma}\r)^d,\\
\theta_2&=12,\\
\theta_3&=\frac{\phi_1/\phi_2}{2^{4d+8}\l(1+\frac{\sigma}{\tau}\r)^d},\\
\theta_4&=2\vee \frac{2^3 3^4}{\tau^2},\\
\theta_5&=\left(2\vee\frac{2^2 3^2}{\tau}\right)\left(1+ 3\cdot 2^5\sqrt{2d}\left(1+\frac{\sigma}{\tau}\right)^{d/2}
\left(\frac{1+\left(\frac{25}{71}\right)^2}{1-\frac{1}{9\cdot 2^{12}}}\right)^{d/4}\right),\\
\alpha &= 1.
\end{align*}
\end{theorem}

%We may combine the results of Theorem \ref{thm:MFD} and Theorem \ref{thm:FSB} as follows: use the first $\lceil \frac n2\rceil$ points  $\{X_1,\ldots,X_{\lceil \frac n2\rceil}\}$ to construct the partition $C_{j,k}$, while the remaining $\{X_{\lceil \frac n2\rceil+1},\ldots,X_{n}\}$ are used to obtain $\hat P_j$ (see (\ref{eq:a30})). 
%This makes our GMRA construction entirely (cover tree, partitions, affine linear projections)  data-dependent. 

%(\Stas{It would be much easier to avoid the complete theorem that was here before, since it would require more explicit statement with constants})
%Moreover, choosing $j_n$ so that $2^{j_n}\simeq n^{\frac{1}{d+2}}$, we can easily deduce that under assumptions of Theorem \ref{thm:FSB} 
%\Stas{Where did the $log$-factors come from? Are they needed?}
%\begin{align}
%\label{eq:rate}
%&
%\mb E\|x-\widehat{P}_{j_n}(x)\|^2\lesssim \left(\frac{\log n}n\right)^{\frac{4}{d+2}}\vee\sigma^2
%\end{align}
%and
%\begin{align}
%\label{eq:suprate}
%\left\|\Id-\widehat P_j\right\|_{\infty,\Pi}\lesssim \left(\frac{\log n}n\right)^{\frac{2}{d+2}}\vee\sigma
%\end{align}
%with high probability, where all subsumed constants depend only on $\tau,\sigma,d,\phi_1,\phi_2$ and $\vol(\m M)$.

One may combine the results of Theorem \ref{thm:MFD} and Theorem \ref{thm:FSB} as follows: 
given an i.i.d. sample $X_1,\ldots, X_n$ from $\Pi$, use the first $\lceil \frac n2\rceil$ points  $\{X_1,\ldots,X_{\lceil \frac n2\rceil}\}$ to obtain the partition $\{C_{j,k}\}_{k=1}^{N(j)}$, while the remaining $\{X_{\lceil \frac n2\rceil+1},\ldots,X_{n}\}$ are used to construct the operator
$\hat P_j$ (see (\ref{eq:a30})). 
This makes our GMRA construction entirely (cover tree, partitions, affine linear projections) data-dependent.
We observe that since our approximations are piecewise linear, they are insensitive to regularity of the manifold beyond first order, so the estimates saturate at $\alpha=1$.

When $\sigma$ is very small or equal to $0$, the bounds resulting from Theorem \ref{thm:FSB} can be ``optimized'' over $j$ to get the following statement (we present only the bounds for the $L_2(\Pi)$ error, but the results $\|\cdot\|_{\infty,\Pi}$ are similar).
\begin{corollary}
\label{cor:noiseless}
Assume that conditions of Theorem \ref{thm:MFD} hold, and that $n$ is sufficiently large. 
Then for all $A\geq 1$ such that $A\log n\leq c_4 n$, the following holds:
\begin{description}
\item[(a)]
if $d\in\{1,2\}$,
\[
\inf\limits_{j\in \mb Z: 2^{-j}<\tau/24}\mb E\|x-\widehat{P}_{j}(x)\|^2\leq C_1 \left(\frac{\log n}{n}\right)^{\frac{2}{d}};
\]
\item[(b)]
if $d\geq 3$, 
\begin{align}
\label{eq:rate}
\inf\limits_{j\in \mb Z: 2^{-j}<\tau/24}\mb E\|x-\widehat{P}_{j}(x)\|^2\leq C_2 \left(\frac{\log n}{n}\right)^{\frac{4}{d+2}}
\end{align}
\end{description}
with probability $\geq 1-c_3 n^{-A}$, where $C_1$ and $C_2$ depend only on $A,\tau,d,\phi_1/\phi_2, \vol(\m M)$ and $c_3,c_4$ depend only on $\tau,d,\phi_1/\phi_2, \vol(\m M)$.
\end{corollary}
\begin{proof}
In case (a), it is enough to set $t:=(A+1)\log n$, $2^{-j}:=\left(\frac{16t}{\theta_1 n}\right)^{1/d}$, and apply Theorem \ref{thm:FSB}. 
For case (b), set $t:=(A+1)\log n$ and $2^{-j}:=\left(\frac{A\log n}{n}\right)^{\frac{1}{d+2}}$.
\end{proof}
Finally, we note that the claims {\em ii.} and {\em iii.} stated in the beginning of Section \ref{sec:def} easily follow from our general results (it is enough to choose $n$ such that $\eps\simeq n^{-\frac{2}{d+2}}$ and $2^{-j}=\sqrt{\eps}$). 
%(with $\eps=n^{-\frac{2}{d+2}}$). 
Claim {\em i.} follows from assumption \textbf{(A1)} and Theorem \ref{thm:MFD}. 
Computational complexity bounds {\em iv.} follow from the associated computational cost estimates for the cover trees algorithm and the randomized singular value decomposition, and are discussed in detail in Sections 3 and 8 of \citep{CM:MGM2}.

\begin{remark}
\label{rmk:tau-sigma}
It follows from our proof that it is sufficient to assume a weaker (but somewhat more technical) form of $(\tau,\sigma)$-model condition for the conclusion of Theorem \ref{thm:MFD} to hold. 
Namely, let $\tilde \Pi$ be the pushforward of $\Pi$ under the projection $\proj{}_{\m M}:\m M_\sigma\rightarrow\m M$, and assume that there exists $\tilde \phi_1>0$ such that for any measurable $A\subseteq \m M$ 
\[
\tilde \Pi(A):=\Pi\l(\proj{}_{\m M}^{-1}(A)\r)\geq \tilde\phi_1 U_\m M(A).
\]
Moreover, suppose that there exists $\tilde \phi_2>0$ such that for any $y\in \m M$, any set $A\subset \m M_\sigma$ and any $\tau>r\geq 2\sigma$ such that 
$B(y,r)\cap \m M_\sigma \subseteq A\subseteq B(y,12r)$, we have
\[
\Pi(A)\leq \tilde \phi_2 \m U_{\m M_\sigma}(A).
\]
In some circumstances, checking these two conditions is not hard (e.g., when $\m M$ is a sphere, $Y$ is uniformly distributed on $\m M$, $\eta$ is spherically symmetric ``noise'' independent of $Y$ and such that $\|\eta\|\leq \sigma$, and $\Pi$ is the distribution of $Y+\eta$), but 
$(\tau,\sigma)$ - assumption does not need to hold with constants $\phi_1$ and $\phi_2$ independent of $D$. 
\end{remark}

\subsection{Connections to the previous work and further remarks}

It is useful to compare our rates with results of Theorem 4 in \citep{Canas2012Learning-Manifo00}. 
In particular, this theorem implies that, given a sample of size $n$ from the Borel probability measure $\Pi$ on the smooth $d$-dimensional manifold $\m M$, the $L_2(\Pi)$-error of approximation of $\m M$ by $k_n=C_1(\m M,\Pi)n^{d/(2(d+4))}$ affine subspaces is bounded by 
$C_2(\m M,\Pi)n^{-2/(d+4)}$. 
Here, the dependence of $k_n$ on $n$ is ``optimal'' in a sense that it minimizes the upper bound for the risk obtained in \citep{Canas2012Learning-Manifo00}. 
If we set $\sigma=0$ in our results, then it easily follows from Theorems \ref{thm:MFD} and \ref{thm:FSB} that the $L_2(\Pi)$-error achieved by our GMRA construction for $2^j\simeq n^{\frac{1}{2(d+4)}}$ (so that $N(j)\simeq k_n$ to make the results comparable) is of the same order $n^{-\frac{2}{d+4}}$. 
However, this choice of $j$ is not optimal in this case - in particular, setting $2^{j_n}\simeq n^{\frac{1}{d+2}}$, we obtain as in (\ref{eq:rate}) a $L_2(\Pi)$-error of order $n^{-\frac{2}{d+2}}$, which is a faster rate. 
Moreover, we also obtain results in the $\sup$ norm, and not only for mean square error.
We should note that technically our results require the stronger condition (\ref{eq:tau-sigma}) on the underlying measure $\Pi$, while theoretical guarantees in \citep{Canas2012Learning-Manifo00} are obtained assuming only the upper bound $\frac{d\Pi}{dU_\m M}\leq \phi_2<\infty$, where  $U_\m M:=\frac{d\vol}{\vol(\m M)}$ is the uniform distribution over $\m M$. 

The rate \eqref{eq:rate} is the same (up to log-factors) as the minimax rate obtained for the problem considered in \citep{Genovese:2012:MME:2503308.2343687} of estimating a manifold from the samples corrupted with the additive noise that is ``normal to the manifold''. 
Our theorems are stated under more general conditions, however, we only prove \emph{robustness-type} results and do not address the problem of \emph{denoising}. 
At the same time, the estimator proposed in \citep{Genovese:2012:MME:2503308.2343687} is (unlike our method) not suitable for applications. 
The paper \citep{Wasserman:ManifoldEstimationHausdorff} considers (among other problems) the noiseless case of manifold estimation under Hausdorff loss, and obtains the minimax rate of order $n^{-\frac2d}$.
%, and presents and estimator, that could be efficiently implemented, is constructed that achieves that rate up to logarithmic factors. 
Performed numerical simulation (see Section \ref{sec:numerical}) suggest that our construction also appears to achieve this rate in the noiseless case. 
However, our main focus is on the case $\sigma>0$. 

The work of \cite{TestingManifoldHyp} establishes the sampling complexity of testing the hypothesis if an unknown distribution is close to being on a manifold (with known reach, volume, dimension) in the Mean Squared sense, is also related to the work discussed in this section, and to the present one. While our results do imply that if we have enough points, as prescribed by our main theorems, and the MSE does not decay as prescribed, then the data with high probability does not satisfy the geometric assumptions in the corresponding theorem, this is still different from the hypothesis testing problem. There may distributions not satisfying our assumptions, such that GMRA still yields good approximations: in fact we welcome and do not rule out these situations. \cite{TestingManifoldHyp} also present an algorithm for constructing an approximation to the manifold; however such an algorithm does not seem easy to implement in practice. The emphasis in this work is on moving to a more general setting than the manifold setting, focusing on multiscale approaches that are robust (locally, because of SVD, as well as across scales), and fast, easily implementable algorithms.
%The work \cite{Wasserman:ManifoldEstimationHausdorff} also considers two noisy cases, the second of which is additive noise without the normality constraint imposed in \cite{Genovese:2012:MME:2503308.2343687}, for which a (slow!) minimax rate of $(\log n)^{-1}$ is proved.

%It is natural to ask if the dependence on the ``noise rate'' $\sigma$ in the bounds (\ref{eq:rate}), (\ref{eq:suprate}) is optimal. 
%We hope to answer this question in our future work. 

We remark that we analyze the case of one manifold $\mathcal{M}$, and its ``perturbation'' in the sense of having a measure supported in a tube around $\mathcal{M}$. Our construction however is multiscale and in particular local. Many extensions are immediate, for example to the case of multiple manifolds (possibly of different dimensions) with non-intersecting tubes around them. The case of unbounded noise is also of interest: if the noise has sub-Gaussian tails then very few points are outside a tube of radius dependent on the sub-Gaussian moment, and these ``outliers'' are easily disregarded as there are few and far away, so they do not affect the construction and the analysis at fine scales. Another situation is when there are many gross outliers, for example points uniformly distributed in high-dimension in, say, a cube containing $\mathcal{M}$. But then the volume of such cube is so large that unless the number of points is huge (at least exponential in the ambient dimension $D$), almost all of these points are in fact far from each other and from $\mathcal{M}$ with very high probability, so that again they do affect the analysis and the algorithms. These are some of the advantages of the multiscale approach, which would otherwise have the potential of corrupting the results (or complicating the analysis of) other global algorithms, such as $k$-flats.

\section{Preliminaries}
\label{sec:prelim}

This section contains the remaining definitions and preliminary technical facts that will be used in the proofs of our main results.

%Given a vector subspace $V\subset\mathbb{R}^D$, we let $\proj{}_V:\mathbb{R}^D\rightarrow\mathbb{R}^D$ denote the orthogonal projection onto $V$, and for an $x\in\mathbb{R}^D$ we shall let $\proj{}_{x+V}:\mathbb{R}^d\rightarrow\mathbb{R}^D$ denote the affine projection onto the affine subspace
%$$
%x+V = \{ y=x+v\in\mathbb{R}^D: v\in V\}.
%$$
Given a point $y$ on the manifold $\m M$, let $T_y \m M$ be the associated tangent space, and let $T_y^\perp\m M$ be the orthogonal complement of $T_y\m M$ in $\mathbb{R}^D$. 
We define the projection from the tube $\m M_\sigma$ (see (\ref{eq:tube})) onto the manifold $\proj{}_{\m M}:\m M_\sigma\rightarrow \m M$ by
\[
\proj{}_{\m M}(x) = \argmin_{y\in\m M}\Vert x-y\Vert
\]
and note that $\sigma<\tau$, together with (\ref{eq:reach}), implies that $\proj{}_{\m M}$ is well-defined  on $\m M_\sigma$, and
\[
\proj{}_{\m M}(y+\xi) = y
\]
whenever $y\in \m M$ and $\xi\in T_y^\perp \m M \cap B(0,\sigma)$.

Next, we recall some facts about the volumes of parallelotopes that will prove useful in Section \ref{sec:proofs}. For a matrix $A\in \mb R^{k\times l}$ with $l\leq k$, we shall abuse our previous notation and let $\Vol(A)$ also denote the volume of the parallelotope formed by the columns of $A$. Let $A$ and $B$ be $k\times l_1$ and $k\times l_2$ matrices respectively with $l_1+l_2\leq k$, and note that
$$
\Vol([A \,|\, B]) \leq \Vol(A)\Vol(B)
$$
where $([A \,|\, B])$ denotes the concatenation of $A$ and $B$ into a $k\times (l_1+l_2)$ matrix. 
Moreover, if the columns of $A$ and $B$ are all mutually orthogonal, we clearly have that
$
\Vol([A \,|\, B]) = \Vol(A)\Vol(B)
$.
Assuming that $I$ is the $l_1\times l_1$ identity matrix, we have the bound
$
\Vol\begin{pmatrix}
			A\\
			I
			\end{pmatrix}\geq 1
$.
The following proposition gives volume bounds for specific types of perturbations that we shall encounter. 
\begin{proposition}\label{prop:volume}
Suppose $Y=[y_1|\cdots |y_d]$ is symmetric $d$ by $d$ matrix such that $\Vert Y\Vert\leq q<1$. 
%for $i=1,\ldots, d$. 
Then
\begin{eqnarray*}
\Vol\begin{pmatrix}
I+Y\\
X
\end{pmatrix}
&\leq& (1+q)^d\,\Vol\begin{pmatrix}
I\\
X
\end{pmatrix}
\\
\Vol\begin{pmatrix}
I+Y& X^T\\
X & -I
\end{pmatrix}
&\geq&
 (1-q)^d\,\Vol\begin{pmatrix}
I & X^T\\
X & -I
\end{pmatrix}.
\end{eqnarray*}
\end{proposition}

\noindent This proof (as well as the proofs of our other supporting technical contributions) is given in the Appendix.
Finally, let us recall several important geometric consequences involving the reach: 

\begin{proposition}
\label{prop:reach}
The following holds:
\begin{enumerate}
\item[i.]
For all $x,y\in \m M$ such that $\|x-y\|\leq \tau/2$, we have 
$$
d_\m M(x,y)\leq \tau-\tau\sqrt{1-2\frac{\|x-y\|}{\tau}}\leq 2\| x-y\|.
$$
\item[ii.] Let $\gamma(t):[0,1]\mapsto \m M$ be the arclength-parameterized geodesic. 
Then $\|\gamma^{\prime\prime}(t)\|\leq \frac{1}{\tau}$ for all $t$. 
\item[iii.] Let $\phi$ be the angle between $T_x\m M$ and $T_y \m M$, in other words,  
\[
\cos(\phi):=\min_{u\in T_x \m M, \|u\|=1}\max_{v\in T_y\m M,\|v\|=1}\left| \dotp{u}{v}\right|.
\] 
If $\Vert x-y\Vert\leq \frac{\tau}{2}$, then 
$
\cos(\phi)\geq \sqrt{1-2\frac{\Vert x-y\Vert}{\tau}}.
$
\item[iv.] If $x$ is such that $\Vert x-y\Vert<\tau/2$, then $x$ is a regular point of $\proj{}_{y+T_y\m M}:\m B(y,\tau/2)\cap \m M\rightarrow y+T_y\m M$ (in other words, the Jacobian of $\proj{}_{y+T_y\m M}$ at $x$ is nonsingular).
\item[v.] Let $y\in \m M$, $r<\tau$ and $A=\m M\cap B(y,r)$. Then
\[
B_d(y,r\cos(\theta))\subseteq \proj{}_{y+T_y \m M}(A),
\]
where $\theta=\arcsin\left(\frac{r}{2\tau}\r)$.
\end{enumerate}
\end{proposition}
\begin{proof}
Part {\em i.} is the statement of Proposition 6.3  and part {\em ii.} - of Proposition 6.1 in \citep{Niyogi:homology}. Part {\it iii.} is demonstrated in Lemma 5.4 of the same paper, and this lemma coincides with {\it iv.} Part {\it v.} is proven in Lemma 5.3 of \citep{Niyogi:homology}. 
\end{proof}

%\section{GMRA for smooth manifolds}

\section{Proofs of the main results}
\label{sec:proofs}

The rest of the paper is devoted to the proofs of our main results. 

\subsection{Overview of the proofs}

We begin by providing an overview of the main steps of the proofs to aid comprehension. 
The proof of Theorem \ref{thm:FSB} begins by invoking the bias-variance decomposition:
\[
\Vert x - \widehat{P}_j(x)\Vert^2\leq 2\Vert x-P_j(x)\Vert^2 + 2 \Vert P_j(x)- \widehat{P}_j(x)\Vert^2.
\]
Remark \ref{rem:1}, part \textit{iii.} and the decomposition
\[
\mathbb{E}\Vert X-P_j(X)\Vert^2 = \sum_{k=1}^{N(j)} \Pi(C_{j,k})\mathbb{E}_{j,k}\Vert X-P_j(X)\Vert^2
\]
gives us the first term in the bound of Theorem \ref{thm:FSB}. Note that this contribution is deterministic.

The next step in the proof is to bound the stochastic error $\mathbb{E}\Vert P_j(x)- \widehat{P}_j(x)\Vert^2$ with high probability. 
We start with the bound
\begin{align}
\label{eq:initbnd}
\|P_j(x)-\widehat P_j(x)\|&=\|c_{j,k}-\widehat c_{j,k}+\proj{}_{V_{j,k}}(x-c_{j,k})-\proj{}_{\widehat V_{j,k}}(x-c_{j,k}+c_{j,k}-\widehat c_{j,k})\| \\
&\leq
2\|c_{j,k}-\widehat c_{j,k}\|+\|\proj{}_{V_{j,k}}-\proj{}_{\widehat V_{j,k}}\| \cdot \|x-c_{j,k}\|.
\end{align}
for $x\in C_{j,k}$. We then use concentration of measure results (matrix Bernstein-type inequality) to bound the terms
\[
\|c_{j,k}-\widehat c_{j,k}\|\text{ and }\l\|\widehat \Sigma_{j,k}-\Sigma_{j,k}\r\|
\]
with high probability. The latter bound and Assumption {\bf(A3)} allows us to invoke Theorem \ref{thm:daviskahan} to obtain a bound of the form
\[
\|\proj{}_{V_{j,k}}-\proj{}_{\widehat V_{j,k}}\|\leq C \l\|\widehat \Sigma_{j,k}-\Sigma_{j,k}\r\|.
\]
Finally, the term $\|x-c_{j,k}\|$ is controlled by Assumption {\bf(A2)}.

The proof of Theorem \ref{thm:MFD} is primarily supported by a volume comparison theorem that allows for the cancellation of the ``noisy" terms that would imply dependency on $D$. That is, supposing that $\text{Proj}_\mathcal{M}:\mathcal{M}_\sigma\rightarrow\mathcal{M}$ is the projection from the $\sigma$-tubular neighborhood onto the underlying manifold with reach $\tau$, if $U\subset\mathcal{M}$ is $\vol$-measurable with $\vol(U)>0$, we have that
\[
\left(1-\frac{\sigma}{\tau}\right)^d\leq\frac{\Vol(\text{Proj}_\mathcal{M}^{-1}(U))}{\vol(U)\Vol(B_{D-d}(0,\sigma))}\leq \left(1+\frac{\sigma}{\tau}\right)^d.
\]
This is encapsulated in Lemma \ref{lem:volUB}. This allows us to relate probabilities on the tubular neighborhood with probabilities on the manifold itself, which only involve $d$-dimensional volumes.

The first thing that this allows us to do is to ensure that a sufficiently large sample from $\mathcal{U}_{\mathcal{M}_\sigma}$, $\{X_i\}_{i=1}^N$, has that $\{\text{Proj}_\mathcal{M}(X_i)\}_{i=1}^N$ is an $\varepsilon$-net for $\mathcal{M}$. Running the cover tree algorithm at the appropriate scale and invoking the cover tree properties at this scale yields the constant for Assumption {\bf(A2)}. Cover tree properties also ensure that each partition element contains a large enough portion of the tubular neighborhood, which we then relate to a portions of the manifold whose volume is comparable to $d$-dimensional Euclidean volumes. This approach provides the constant for Assumption {\bf(A1)}. Finally, the constants from Assumption {\bf(A3)} and {\bf(A4)} are obtained from local moment estimates based upon these volume bounds.

Now, the volume comparison bounds themselves are proven by considering coordinate systems that locally invert orthogonal projections onto tangent spaces. The fact that the manifold has reach $\tau$ imposes bounds on the Jacobians and second-order terms for these local inversions. These bounds are ultimately used to bound volume distortions, and lead to the volume comparison result above.

\subsection{Proof of Theorem \ref{thm:FSB}}
				
Assumption {\bf(A3)} above controls the $L_2(\Pi)$ approximation error of $x\in M$ by $P_j(x)$ (see Remark \ref{rem:1}, part \textit{iii.}), hence we will concentrate on the stochastic error $\|\widehat P_j(x)-P_j(x)\|$. 
To this end, we will need to estimate $\|c_{j,k}-\widehat c_{j,k}\|$ and $\|\proj{}_{V_{j,k}}-\proj{}_{\widehat V_{j,k}}\|, \ k=1\ldots N(j)$. 

One of the main tools required to obtain this bound is the noncommutative Bernstein's inequality. 

\begin{theorem}\citep[Theorem 2.1]{Minsker2013On-Some-Extensi00}
\label{th:bernstein}
Let $Z_1,\ldots, Z_n\in \mb R^{D\times D}$ be a sequence of independent symmetric random matrices such that $\mb EZ_i=0$ and 
$\|Z_i\|\leq U$ a.s., $1\leq i\leq n$. 
Let  
\[
\sigma^2:=\l\|\sum\limits_{i=1}^n \mb E Z_i^2\r\|.
\] 
%Assume that $\|Z_i\|\leq U$ a.s., $1\leq i\leq n$. 
Then for any $t\geq 1$
\begin{align}
\label{eq:bernstein1}
\l\|\sum_{i=1}^n Z_i\r\|\leq 2\max\l(\sigma\sqrt{t+\log(\bar D)},U(t+\log(\bar D))\r)
\end{align}
with probability $\geq 1-e^{-t}$, 
where $\bar D:=4\frac{\tr\left(\sum\limits_{i=1}^n \mb E Z_i^2\right)}{\sigma^2}$.
\end{theorem}

				\begin{comment}

The version below is due to J. Tropp: 
\begin{theorem}[Theorem 1.4 in \cite{Tropp2012User-friendly-t00}]
\label{th:bernstein}
Let $Z_1,\ldots, Z_n\in \mb R^{D\times D}$ be symmetric independent random matrices such that $\mb EZ_i=0$, $\|Z_i\|\leq U$ a.s. and define 
\[
\sigma^2:=\Big\|\mb E\sum_{i=1}^n Z_i^2\Big\|.
\] 
Then, for all $t>0$, with probability $\geq 1-e^{-t}$
\[
\l\|\sum_{i=1}^n Z_i\r\|\leq 2\max\l(\sigma\sqrt{t+\log(2D)},U(t+\log(2D))\r).
\]
%where $C_1\leq 2$ is an absolute constant. 
\end{theorem}
		\end{comment}

\noindent Note that we always have $\bar D\leq 4D$. 
We use this inequality to estimate $\|\widehat \Sigma_{j,k}-\Sigma_{j,k}\|$:
let $\Pi(dx|A)$ be the conditional distribution of $X$ given that $X\in A$, and set $\Pi_{j,k}(dx):=\Pi(dx|C_{j,k})$.
Let $m_{j,k}:=\sum\limits_{i=1}^n I\{X_i\in C_{j,k}\}$ to be the number of samples in $C_{j,k}, \ k=1\ldots N(j)$. 
Let $I\subset\{1,\ldots,n\}$ be such that $|I|=m$.  
Conditionally on the event $A_I:=\{X_i\in C_{j,k}\text{ for } \ i\in I\,, \text{ and } X_i\notin C_{j,k}\text{ for }i\notin I\}$, the random variables $\{X_i,\ i\in I\}$ are independent with distribution $\Pi_{j,k}$. 
Then 
\begin{align}
\label{eq:a60}
\Pr\l(\l\|\widehat \Sigma_{j,k}-\Sigma_{j,k}\r\| \geq s \,|\, m_{j,k}=m\r)&=
\sum\limits_{I\subset \{1,\ldots,n\}, |I|=m}\Pr\l(\l\|\widehat \Sigma_{j,k}-\Sigma_{j,k}\r\| \geq s \,|\, A_I\r)\frac{1}{\binom{n}{m}} \\
&=\nonumber
\Pr\l(\l\|\widehat \Sigma_{j,k}-\Sigma_{j,k}\r\| \geq s \,|\, A_{\{1,\ldots,m\}}\r).
\end{align}
To estimate $\Pr\l(\l\|\widehat \Sigma_{j,k}-\Sigma_{j,k}\r\| \geq s \,|\, A_{\{1,\ldots,m\}}\r)$, we use the following inequality. 
Recall that 
\[
\bar d=4d^2\frac{\theta_2^4}{\theta_3^2},
\]
where $\theta_2,\theta_3$ are the constants in Assumptions {\bf(A2)} and {\bf(A3)}. 
\begin{lemma}
\label{lemma:cov}
Let $X,X_1,\ldots,X_m$ be an i.i.d. sample from $\Pi_{j,k}$. 
Set
\begin{align*}
\widehat c_{j,k}=\frac 1 m \sum_{i=1}^m X_i\quad\text{ and }\quad
\widehat\Sigma_{j,k}:=\frac 1 m \sum_{i=1}^m(X_i-\widehat c_{j,k})(X_i-\widehat c_{j,k})^T.
\end{align*}
Assume that $m\geq t+\log(\bar d\vee 8)$. 
Then with probability $\geq 1-2e^{-t}$,
\[
\l\|\widehat \Sigma_{j,k}-\Sigma_{j,k}\r\| \leq 6 r^2\sqrt{\frac{ t+\log(\bar d\vee 8)}{m}}.
\]
%(8C_1^2+2C_1)
\end{lemma}
\begin{proof}
 We want to estimate 
\begin{align}
\nonumber
\l\|\widehat \Sigma_{j,k}-\Sigma_{j,k}\r\|
&=\l\|\frac 1 m \sum_{i=1}^m (X_i-c_{j,k})(X_i-c_{j,k})^T-\Sigma_{j,k} + (c_{j,k}-\widehat c_{j,k})(c_{j,k}-\widehat c_{j,k})^T \r\|\\ 
&\leq\l\|\frac 1 m \sum_{i=1}^m (X_i-c_{j,k})(X_i-c_{j,k})^T-\Sigma_{j,k}\r\|+\l\|(c_{j,k}-\widehat c_{j,k})(c_{j,k}-\widehat c_{j,k})^T \r\|.
\label{e:empcovexp}
\end{align}

Set $r:=\theta_2\cdot 2^{-j}$. 
Recall that $\|x-c_{j,k}\|\leq r$ for all $x,y\in C_{i,j}$ by assumption {\bf(A2)}. 
It implies that
\begin{enumerate}
\item 
for all $1\leq i\leq m$, $\|(X_i-c_{j,k})(X_i-c_{j,k})^T\|\leq r^2$ almost surely, 
\item 
$\Big\|\mb E\Big[(X_i-c_{j,k})(X_i-c_{j,k})^T\Big]^2\Big\|=\Big\|\mb E\|X_i-c_{j,k}\|^2(X_i-c_{j,k})(X_i-c_{j,k})^T\Big\| 
\leq r^2 \|\Sigma_{j,k}\|$.  
\end{enumerate}

Therefore, by Theorem \ref{th:bernstein} applied to $Z_i:=\frac{1}{m}(X_i-c_{j,k})(X_i-c_{j,k})^T,  \ i=1\ldots m$,
\begin{align*}
\l\|\frac 1 m \sum_{i=1}^m (X_i-c_{j,k})(X_i-c_{j,k})^T-\Sigma_{j,k}\r\|
&\leq 
2\l(r\sqrt{\frac{(t+\log(\bar d))\|\Sigma_{j,k}\|}{m} }\vee r^2\frac{t+\log(\bar d)}{m} \r) \\
& =
2 r^2 \sqrt{\frac{(t+\log(\bar d))}{m}}\l(\sqrt{\frac{t+\log(\bar d)}{m}}\vee \sqrt{\l\|\frac{\Sigma_{j,k}}{r^2}\r\|}\r)
\end{align*}
with probability $\geq 1-e^{-t}$. 
Note that $\|\Sigma_{j,k}\|\leq \tr(\Sigma_{j,k})\leq r^2$. 
Moreover, 
\begin{align*}
&
\bar D=4 \frac{\tr (\mb E Z_1^2)}{\|\mb EZ_1^2\| }\leq 4\frac{\mb E(\tr Z_1)^2}{\l(\lambda_{d}^{j,k}\r)^2}\leq 4d^2\frac{r^4}{\theta_3^2 2^{-4j}}=
4d^2 \frac{\theta_2^4}{\theta_3^2}=\bar d
\end{align*}
by assumption {\bf(A3)} and the definition of $r$. 
Since $\frac{t+\log(\bar d)}{m}\leq 1$ by assumption,
\[
\l\| \frac 1 m \sum_{i=1}^m (X_i-c_{j,k})(X_i-c_{j,k})-\Sigma_{j,k} \r\|\leq 2r^2\sqrt{\frac{t+\log(\bar d)}{m}}.
\]
For the second term in \eqref{e:empcovexp}, note that $\l\|(c_{j,k}-\widehat c_{j,k})(c_{j,k}-\widehat c_{j,k})^T \r\|=\|c_{j,k}-\widehat c_{j,k}\|^2$. 
We apply Theorem \ref{th:bernstein} to the symmetric matrices
\[
G_i:=\begin{pmatrix}
0 & (X_i-c_{j,k})^T \\
X_i-c_{j,k} & 0
\end{pmatrix}.
\]
Noting that $\|G_i\|=\|X_i-c_{j,k}\|\leq r$ almost surely,
\[
\|\mb E G_i^2\|=\mb E\|X_i-c_{j,k}\|^2=\tr(\Sigma_{j,k})\leq r^2,
\]
and $\frac{\tr(\mb EG_i^2)}{\|\mb EG_i^2\|}=2$, 
we get that for all $t$ such that $t+\log 8\leq m$, with probability $\geq 1-e^{-t}$
\begin{align}
\label{eq:c10}
\|\widehat c_{j,k}-c_{j,k}\|\leq 2\left[r\sqrt{\frac{(t+\log 8)}{m}}\vee r\frac{t+\log 8}{m}\right]\leq 2 r\sqrt{\frac{t+\log 8}{m}},
\end{align}
hence with the same probability 
\begin{align*}
&
\l\|(c_{j,k}-\widehat c_{j,k})(c_{j,k}-\widehat c_{j,k})^T\r\|\leq 4 r^2 \frac{t+\log 8}{m},
\end{align*}	
and the claim follows.
\end{proof}
Given the previous result, we can estimate the angle between the eigenspaces of $\widehat \Sigma_{j,k}$ and $\Sigma_{j,k}$: 
\begin{theorem}\citep{davis1970rotation}, or \citep[Theorem 3]{Zwald2006On-the-Converge00}.\\
\label{thm:daviskahan}
Let $\delta_d=\delta_d(\Sigma_{j,k}):=\frac 1 2 (\lambda^{j,k}_{d}-\lambda^{j,k}_{d+1})$. 
If $\|\widehat\Sigma_{j,k}-\Sigma_{j,k}\|<\delta_d/2$, then 
%for $\|\cdot\|$ either the operator norm or the Frobenius norm,
\begin{align*}
&
\Big\|\proj{}_{V_{j,k}}-\proj{}_{\widehat V_{j,k}}\Big\|\leq \frac{\|\widehat\Sigma_{j,k}-\Sigma_{j,k}\|}{\delta_d},\\
%&
%\Big\|\proj{}_{V_{j,k}}-\proj{}_{\widehat V_{j,k}}\Big\|_{\rm F}\leq \frac{\|\widehat\Sigma_{j,k}-\Sigma_{j,k}\|_{\rm F}}{\delta_d}.
\end{align*}
\end{theorem}
Since $\delta_d\geq \frac{\theta_3}{2\theta_2^2}\frac{r^2}{d}$ by assumption {\bf(A3)}, the previous result implies that, conditionally on the event $\{m_{j,k}=m\}$, with probability $\geq 1-2e^{-t}$,
\[
\Big\|\proj{}_{V_{j,k}}-\proj{}_{\widehat V_{j,k}}\Big\|\leq 12d\frac{\theta_2^2}{\theta_3}\sqrt{\frac{ t+\log(\bar d\vee 8)}{m}}.
\]
It remains to obtain the unconditional bound. 
Set $n_{j,k}:=n\Pi(C_{j,k})$ and note that $n_{j,k}\geq \theta_1 n2^{-jd}$ by assumption {\bf(A1)}. 
To this end, we have
\begin{align*}
&
\Pr\l(\max_{k=1\ldots N(j)}\Big\|\proj{}_{V_{j,k}}-\proj{}_{\widehat V_{j,k}}\Big\|\geq 12\frac{\theta_2^2}{\theta_3}\sqrt{\frac{(t+\log(\bar d\vee 8))d^2}{n_{j,k}/2}}\r) \\
&\leq
\Pr\l(\max_{k=1\ldots N(j)}\Big\|\proj{}_{V_{j,k}}-\proj{}_{\widehat V_{j,k}}\Big\|\geq 12\frac{\theta_2^2}{\theta_3}\sqrt{\frac{(t+\log(\bar d\vee 8))d^2}{n_{j,k}/2}}\bigg| m_{j,k}\geq n_{j,k}/2, \ k=1\ldots N(j)\r)\\
& +\Pr\l(\bigcup_{k=1}^{N(j)}\{m_{j,k}<n_{j,k}/2\}\r)\leq N(j)e^{-t}+ \sum_{k=1}^{N(j)}\Pr\l(m_{j,k}<n_{j,k}/2\r).
\end{align*}
Recall that $m_{j,k}=\sum\limits_{i=1}^n I\{X_i\in C_{j,k}\}$, hence $\mb Em_{j,k}=n_{j,k}$ and $\var(m_{j,k})\leq n_{j,k}$. 
%Required bound can be obtained by combining the usual Bernstein inequality and the union bound. 
%Indeed, let $\xi_{i,j}=I\l\{X_i\in V_j\r\}$, so that $\mb E\xi_{i,j}=P(V_j)$. 
Bernstein's inequality  \citep[see Lemma 2.2.9 in][]{Vaart1996Weak-convergenc00} implies that
\[
\l|m_{j,k}-n_{j,k}\r|\leq \l(2\sqrt {s n_{j,k}}\vee \frac{4}{3}s\r)
\]
with probability $\geq 1-e^{-s}$. 
Choosing $s=\frac{n_{j,k}}{16}$, we deduce that $\Pr\l(m_{j,k}<n_{j,k}/2\r)\leq e^{-\frac{\theta_1}{16}n2^{-jd}}$, and, since $N(j)\leq \frac{1}{\theta_1}2^{jd}$ by assumption {\bf(A1)}, 
\[
\sum_{k=1}^{N(j)}\Pr\l(m_{j,k}<n_{j,k}/2\r)\leq \frac{1}{\theta_1}2^{jd}e^{-\frac{\theta_1}{16}n2^{-jd}}
\]
and 
\begin{align}
\label{eq:c20}
\Pr\l(\max_{k=1\ldots N(j)}\Big\|\proj{}_{V_{j,k}}-\proj{}_{\widehat V_{j,k}}\Big\|\geq 12\frac{\theta_2^2}{\theta_3}\sqrt{\frac{(t+\log(\bar d\vee 8))d^2}{n_{j,k}/2}}\r)\leq
\frac{2^{jd}}{\theta_1}\left(e^{-t}+e^{-\frac{\theta_1}{16}n2^{-jd}}\right)
\end{align}
%and applying the union bound over $1\leq j\lesssim 2^{dj}\simeq n^{\frac{d}{d+2}}$, we get 
%$$
%\Pr\l(\forall j \  n_j\geq \frac 1 2 nP(V_j)\r)\geq 1-cn^{\frac{d}{d+2}}\exp\l(-n^{\frac{2}{2+d}}\r).
%$$
A similar argument implies that 
\begin{align}
\label{eq:c30}
\Pr\l(\max_{k=1\ldots N(j)}\|c_{j,k}-\widehat c_{j,k}\|\geq 2 r\sqrt{\frac{t+\log(\bar d\vee 8)}{n_{j,k}/2}}\r)\leq
\frac{2^{jd}}{\theta_1}\left(e^{-t}+e^{-\frac{\theta_1}{16}n2^{-jd}}\right).
\end{align}
We are in position to conclude the proof of Theorem \ref{thm:FSB}. With assumption {\bf(A2)}, (\ref{eq:c20}), and (\ref{eq:c30}), the initial bound (\ref{eq:initbnd}) implies that, with high probability,
%It follows from Lemma \ref{lemma:cov} and (\ref{eq:c10}) that inequalities $\|c_{j,k}-\widehat c_{j,k}\|\leq C_3 r\sqrt{\frac{t+\log(2D)}{m}}$ and $$
%From , $\|c_{j,k}-\widehat c_{j,k}\|\leq C_3 r\sqrt{\frac{t+\log(2D)}{m}}$ with probability $\geq 1-e^{-t}$. 
\[
\|P_j(x)-\widehat P_j(x)\|\leq 
4\sqrt{2}\frac{\theta_2}{\sqrt{\theta_1}} 2^{-j}\sqrt{\frac{t+\log(\bar d\vee 8)}{n2^{-jd}}}+
12\sqrt{2}\frac{\theta_2^3}{\theta_3\sqrt{\theta_1}}  2^{-j}\sqrt{\frac{(t+\log(\bar d\vee 8))d^2}{n2^{-jd}}}.
\]
Combined with assumption {\bf(A3)} (see Remark \ref{rem:1}, part \textit{iii.}), this yields the result.  

\subsection{Proof of Theorem \ref{thm:MFD}}

%In this section, we prove Theorem \ref{thm:MFD}.
Recall that $\m M\hookrightarrow\mb R^D$ is a smooth (or at least $C^2$) compact manifold without boundary, with reach $\tau$, and
equipped with the volume measure $d\Vol_\m M$. 
Our proof is divided into several steps, and each of them is presented in a separate subsection to improve readability. 

\subsubsection{Local inversions of the projection}

%\Stas{$B(y,r)$ is defined with respect to the Euclidean distance - but it seems that we use the geodesic balls in some places..}

In this section, we introduce lemmas which ensure that (for $r<\tau/8$) the projection map $\proj{}_{y+T_y\m M}$ is injective on $B(y,r)\cap\m M$, and hence invertible by part {\it iv.} of Proposition \ref{prop:reach}. 
We also demonstrate that the derivatives of this inverse are bounded in a suitable sense. 
These estimates shall allow us to develop bounds on volumes in $\m M_\sigma$.
%, and thus bounds for probabilities the uniform measure on $\m M_\sigma$. 

We begin by proving a bound on the local deviation of the manifold from a tangent plane.

\begin{lemma}\label{lem:lengths}
Suppose $\eta\in T_y^\perp\m M$ with $\Vert\eta\Vert =1$ and $z\in B(y,r)\cap \m M$, where $r\leq \tau/2$. Then 
$$
\vert \langle \eta, z-y\rangle\vert\leq \frac{2r^2}{\tau}
$$
\end{lemma}

Our next lemma quantitatively establishes the local injectivity of the affine projections onto tangent spaces. 
\footnote{In an independent work, \citet{Eftekhari2013New-analysis-of00} prove a slightly stronger result that holds for $r<\tau/4$. }
\begin{lemma}\label{lem:inject}
Suppose $y\in \m M$ and $r<\tau/8$. Then $\proj{}_{y+T_y\m M}:B(y,r)\cap \m M\rightarrow y+T_y\m M$ is injective.
\end{lemma}

There are two important conclusions that Lemma \ref{lem:inject} provides. 
First of all, it indicates that, under a certain radius bound, the manifold does not ``curve back'' into particular regions. 
This is helpful when we begin to examine upper bounds on local volumes. 
More importantly, if we let $J_{y,r}=\proj{}_{y+T_y\m M}(B(y,r)\cap \m M)$, then there is a well-defined inverse map $f$ of $\proj{}_{y+T_y\m M}$, $f: J_{y,r}\rightarrow B(y,r)\cap\m M$, when $r<\tau/8$. 
Part {\it iv} of Proposition \ref{prop:reach} implies that $f$ is at least a $C^2$ function, and part {\it v} of Proposition \ref{prop:reach} implies that there is a $d$-dimensional ball inside of $J_{y,r}$ of radius 
$\cos(\theta)r$, where $\theta=\arcsin(r/2\tau)$. 

%\Stas{Corrected $C^2$ to $C^1$ - Lemma 5.4 in NiySmaWei does not claim that $f$ is $C^2$, as far as I can judge.}
%\Nate{We need at least $C^2$ since we have to look at $D^2f$, and $C^2$ follows from the fact that there are no critical points of $\proj{}_{y+T_y\m M}$ in $\m M_\sigma \cap B(y,\tau/2)$.}

Whenever we refer to such an $f$, we think of $J_{y,r}$ as a subset in the span of the first $d$ canonical directions, and we identify $f$ with the value $f$ takes in the span of the remaining $D-d$ directions. 
Thus, we identify $f$ with the function whose graph is a small part of the manifold. Such an identification is obtained via an affine transformation, so we may do this without any loss of generality. Using these assumptions, we may prove the following bounds.

\begin{proposition}
\label{prop:diffBounds}
Let $\eps < \tau/8$, and assume $f$ is defined above so that
$
v\longmapsto \begin{pmatrix}
						v\\
						f(v)
						\end{pmatrix}
$
is the inverse of $\proj{}_{y+T_y\m M}$ in $B(y,\eps)$ for some $y\in \m M$. Then
\begin{align}
\label{eq:jacoSUP}
\sup_{v\in B_d(0,\eps)}\Vert Df(v)\Vert \leq \frac{2\eps}{\tau-2\eps}
\end{align}
and 
\begin{align}
\label{eq:hessSUP}
\sup_{v\in B_d(0,\eps)}\sup_{u\in\m S^{D-d-1}} \l\Vert \sum_{i=1}^{D-d-1} u_i D^2f_i(v)\r\Vert \leq \frac{\tau^2}{(\tau-2\eps)^3}.
\end{align}
\end{proposition}

\subsubsection{Volume bounds}

The main result of this section is Lemma \ref{lem:volUB}, which allows us to compare volumes in $\m M_\sigma$ with volumes in $\m M$. It also establishes an upper bound on volumes, which is an essential ingredient when we control the conditional distribution of $\Pi$ subject to being in a particular $C_{j,k}$. The form of the bounds also allows us to cancel out noisy terms that would make the estimates depend upon the ambient dimension $D$.

\begin{lemma}
\label{lem:volUB}
Suppose $\sigma<\tau$, suppose $U\subseteq\m M$ is measurable, and define $P:\m M_\sigma\rightarrow\m M$ so that $x\mapsto \proj{}_{\m M}(x)$
 under $P$. Then
\begin{enumerate}
\item[i.] \small
$\displaystyle
\left(1-\frac{\sigma}{\tau}\right)^d\vol(U)\Vol(B_{D-d}(0,\sigma))\leq\Vol(P^{-1}(U))\leq \left(1+\frac{\sigma}{\tau}\right)^d\vol (U)\Vol(B_{D-d}(0,\sigma))
$\normalsize
\item[ii.] If $r+\sigma\leq\tau/8$, then \small
$$
\Vol(\m M_\sigma\cap B(y,r))\leq \left(1+\frac{\sigma}{\tau}\right)^d\left(1+\left(\frac{2(r+\sigma)}{\tau-2(r+\sigma)}\right)^2\right)^{d/2}\Vol(B_d(0,r+\sigma))\Vol(B_{D-d}(0,\sigma)).
$$\normalsize
\end{enumerate}
\end{lemma}

\subsubsection{Absolute continuity of the pushforward of $U_{\m M_\sigma}$ and local moments}

Recall that $U_{\m M_\sigma}$ is the uniform distribution over $\m M_\sigma$, and let $U_\m M:=\frac{d\vol}{\vol(\m M)}$ be the uniform distribution over $\m M$. 
In this section, we exploit the volume bounds of the previous subsection to obtain control over probabilities and local moments of $U_{\m M_\sigma}$.  Our first result allows us to get the lower bounds for $U_{\m M_\sigma}$ that are independent of the ambient dimension $D$. 

\begin{lemma}
Suppose $\sigma<\tau$, and let $\widetilde{U}_{\m M_\sigma}$ denote the pushforward of $U_{\m M_\sigma}$ under $\proj{}_{\m M}$. Then $\widetilde{U}_{\m M_\sigma}$ and $U_{\m M}$ are mutually absolutely continuous with respect to each other, and 
$$
\left(\frac{\tau-\sigma}{\tau+\sigma}\right)^d\leq \frac{d \widetilde{U}_{\m M_\sigma}}{d U_{\m M}}\leq\left(\frac{\tau+\sigma}{\tau-\sigma}\right)^d.
$$
\end{lemma}
\begin{proof}
This is a straightforward consequence of part {\it i.} of Lemma \ref{lem:volUB}.
\end{proof}

The next lemma quantitatively establishes the the decay of the local eigenvalues required in the second part of Assumption {\bf (A3)}.

\begin{lemma}
\label{lem:evalUB}
Suppose $\Pi$ is a distribution supported on $\m M_\sigma$, and let $r<\tau/2$. 
Further assume that $Z$ is the random variable drawn from $\Pi$ conditioned on the event $Z\in Q$ where $\m M_\sigma\cap Q\subset B(y,r)$ for some $y\in\m M$. 
If $\Sigma$ is the covariance matrix of $Z$, then
$$
\sum_{i=d+1}^D\lambda_i(\Sigma)\leq 2\sigma^2+\frac{8r^4}{\tau^2},
$$
where $\lambda_i(\Sigma)$ are the eigenvalues of $\Sigma$ arranged in the decreasing order. 
\end{lemma}
Finally, we derive a lower bound on the upper eigenvalues of the local covariance for the uniform distribution (needed to satisfy the first part of assumption {\bf(A3)}). 
This is done in the following lemma.

\begin{lemma}
\label{lem:evalLB}
Suppose that $Q\subseteq\mathbb{R}^D$ is such that
\[
B(y,r_1)\subseteq Q\text{ and } \m M_\sigma\cap Q\subset B(y,r_2)
\]
for some $y\in \m M$ and $\sigma<r_1<r_2<\tau/8-\sigma$. Let $Z$ be drawn from $U_{\m M_\sigma}$ conditioned on the event $Z\in Q$, and suppose $\Sigma$ is the covariance matrix of $Z$. Then
\[
\lambda_d(\Sigma)\geq \frac{1}{4\left(1+\frac{\sigma}{\tau}\right)^d}\left(\frac{r_1-\sigma}{r_2+\sigma}\right)^d\left(\frac{1-\left(\frac{r_1-\sigma}{2\tau}\right)^2}{1+\left(\frac{2(r_2+\sigma)}{\tau-2(r_2+\sigma)}\right)^2}\right)^{d/2}\frac{(r_1-\sigma)^2}{d}.
\]
\end{lemma}
The following statement is key to establishing the error bounds for GMRA measured in sup-norm.
\begin{lemma}
\label{lem:sup-norm}
Assume that conditions of Lemma \ref{lem:evalLB} hold, and let $V_d:=V_d(\Sigma)$ be the subspace corresponding to the first $d$ principal components of $Z$. Then 
\begin{align*}
&
\sup_{x\in Q}\left\|x-\mb EZ-\proj{}_{V_d}(x-\mb EZ)\right\|\leq 2\sigma+\frac{4r_2^2}{\tau}+
\frac{r_2}{r_1-\sigma}\sqrt{4\sigma^2+\frac{16r_2^4}{\tau^2}}\,\gamma(\sigma,\tau,d,r_1,r_2),
\end{align*}
where 
$\gamma(\sigma,\tau,d,r_1,r_2)=
4\sqrt{2d}\left(1+\frac{\sigma}{\tau}\right)^{d/2}\left(\frac{r_2+\sigma}{r_1-\sigma}\right)^{d/2}
\left(\frac{1+\left(\frac{2(r_2+\sigma)}{\tau-2(r_2+\sigma)}\right)^2}{1-\left(\frac{r_1-\sigma}{2\tau}\right)^2}\right)^{d/4}.$
%\frac{\sqrt d}{r_1-\sigma}.$
\end{lemma}
Notice that the term containing $\gamma(\sigma,\tau,d,r_1,r_2)$ is often of smaller order, so that the approximation is essentially controlled by the maximum of $\sigma$ and $\frac{r_2^2}{\tau}$.

\subsubsection{Putting all the bounds together}
In this final subsection, we prove Theorem \ref{thm:MFD}. 
We begin by translating Proposition 3.2 in \citep{Niyogi:homology} into our setting. 
As before, let $\m X_n=\{X_1,\ldots,X_n\}$ be an i.i.d. sample from $\Pi$, and the $\phi_1$ be the constant defined by (\ref{eq:tau-sigma}).   

\begin{proposition}\citep[Proposition 3.2]{Niyogi:homology}
\label{prop:net}
Suppose $0<\eps<\frac{\tau}{2}$, and also that $n$ and $t$ satisfy
\begin{align}
\label{eq:net_bound}
n\geq \eps^{-d}\frac{1}{\phi_1}\l(\frac{\tau+\sigma}{\tau-\sigma}\r)^d \beta_1\l(\log (\eps^{-d}\beta_2)+t\r),
\end{align}
where $\beta_1= \frac{\vol(\m M)}{\cos^d(\delta_1)\Vol(B_d(0,1/4))}$, $\beta_2= \frac{\vol(\m M)}{\cos^d(\delta_2)\Vol(B_d(0,1/8))}$, $\delta_1=\arcsin(\eps/8\tau)$, and $\delta_2=\arcsin(\eps/16\tau)$. Let ${\m E}_{\eps/2,n}$ be the event that 
$$
\m Y =\l\{Y_j=\proj{}_\m M(X_j)\r\}_{j=1}^n
$$ 
is $\eps/2$-dense in $\m M$
(that is, $\m M\subseteq \bigcup\limits_{i=1}^n B(Y_i,\eps/2)$). 
Then,  $\Pi^n({\m E}_{\eps,n})\geq 1-e^{-t}$, 
where $\Pi^n$ is the $n$-fold product measure of $\Pi$.
\end{proposition}
\begin{proof}
The proof closely follows the one given in  \citep{Niyogi:homology}. The only additional observation to make is that, if $\widetilde{\Pi}$ is the pushforward measure of $\Pi$ under $\proj{}_{\m M}:\m M_\sigma\rightarrow\m M$, then
\begin{align*}
\widetilde{\Pi}\l(\m M\cap B(y,\eps/8)\r)&=\Pi(\proj{}_{\m M}^{-1} (\m M\cap B(y,\eps/8)))\\
&\geq \phi_1 U_{\m M_\sigma}(\proj{}_{\m M}^{-1} (\m M\cap B(y,\eps/8)))\\
&= \phi_1 \widetilde{U}_{\m M_\sigma}(\m M\cap B(y,\eps/8))\\
&\geq \phi_1 \l(\frac{\tau-\sigma}{\tau+\sigma}\r)^d U_{\m M}(\m M\cap B(y,\eps/8)).
\end{align*}
by Lemma \ref{lem:volUB}.
\end{proof}

If $\eps\ll\tau$, previous proposition implies that we roughly need $n\geq {\rm Const}(\m M,d)\l(\frac 1 \eps\r)^d \log \frac 1 \eps$ points to get an $\eps$-net for $\m M$. 
For the remainder of this section, we identify $\eps:=\eps(n,t)$ with the smallest $\eps>0$ satisfying (\ref{eq:net_bound}) in the statement of Proposition \ref{prop:net}, and we also assume that $\eps < \sigma$. 
%Let $C_j\subseteq \tilde S$ be the union of the first $j$-th levels of the cover tree of $\tilde S$ (constructed with respect to the Euclidean metric $d$), for $j\ge 0$.
Take $j\in \mb Z_+$ such that 
\begin{align}
\label{eq:index}
\sigma< 2^{-j-2}<\tau.
\end{align} 
Let $C_{j,k}$ be the partition of $\mb R^D$ into Voronoi cells defined by (\ref{eq:part}). 
Recall that $T_j=\{a_{j,k}\}_{k=1}^{N(j)}\subset \m X_n$ is the set of nodes of the cover tree at level $j$, and set $z_{j,k}=\proj_\m M(a_{j,k})$.

\begin{lemma}
\label{lem:covering}
With probability $\geq 1-e^{-t}$, for all $j$ satisfying (\ref{eq:index}) and $k=1,\dots,N(j)$,
\begin{align}
\label{eq:inc}
B\left(z_{j,k}, 2^{-j-2}\right)\subseteq C_{j,k}\text{ and }C_{j,k}\cap \m M_\sigma\subseteq B(a_{j,k},3\cdot 2^{-j-2}+2^{-j+1})\subseteq B(z_{j,k},3\cdot 2^{-j}).
\end{align}
\end{lemma}

%The properties established above easily imply that the cardinality of $C_j$ satisfies 
%$$
%\l(\frac{C_1}{2^{-j}}\r)^d \leq N(\m M,d,3\cdot 2^{-j}) \leq |C_j| \leq M(\m M,d,\frac 1 4 2^{-j})\leq \l(\frac{C_2}{2^{-j}}\r)^d,
%$$ 
%where $N(\m M,d,\cdot)$ and $M(\m M,d,\cdot)$ stand for the covering number and the packing number of $\m M$ respectively. 

We now use Lemma \ref{lem:covering} to obtain bounds on the constants $\theta_i$ for $i=1,\ldots, 4$ and $\alpha$. We prove a lemma for each of the assumptions {\bf(A1)}, {\bf(A2)}, and {\bf(A3)} and then collect them as the proof of Theorem \ref{thm:MFD}.

\begin{proof}[Proof of Theorem \ref{thm:MFD}]
Since the hypotheses of Lemma \ref{lem:covering} are satisfied with high probability, we first obtain
\begin{align*}
\Pi(C_{j,k})&\geq \Pi(B(z_{j,k},2^{-j-2}))\\
&\geq \phi_1 U_{\m M_\sigma}(B(z_{j,k},2^{-j-2}))\\
&= \phi_1 \frac{\Vol(\m M_\sigma \cap B(z_{j,k},2^{-j-2}))}{\Vol(\m M_\sigma)}\\
&\geq \phi_1 \frac{\Vol(\proj{}_{\m M}^{-1}(\m M\cap B(z_{j,k},2^{-j-2}-\sigma)))}{\Vol(\m M_\sigma)}\\
&\geq \phi_1\l(\frac{\tau-\sigma}{\tau+\sigma}\r)^d\frac{\cos(\delta)^d\Vol(B_d(0,2^{-j-2}-\sigma))}{\vol(\m M)}\\
&\geq \frac{\phi_1\Vol(B_d(0,1))}{2^{4d}\vol(\m M)}\l(\frac{\tau-\sigma}{\tau+\sigma}\r)^d 2^{-jd}.
\end{align*}
where $\delta=\arcsin((2^{-j-2}-\sigma)/2\tau))$. Thus,
\[
\theta_1\geq \frac{\phi_1\Vol(B_d(0,1))}{2^{4d}\vol(\m M)}\l(\frac{\tau-\sigma}{\tau+\sigma}\r)^d
\]
Since the support is contained in a ball of radius $3\cdot 2^{-j}$, we easily obtain that $\theta_2\leq 12$. 
Finally, it is not difficult to deduce from Lemmas \ref{lem:evalUB} and \ref{lem:evalLB} that
\[
\theta_3\geq \frac{\phi_1/\phi_2}{2^{4d+8} \l(1+\frac{\sigma}{\tau}\r)^d},\: \theta_4\leq \left(2\vee \frac{2^3 3^4}{\tau^2}\right),\text{ and }\alpha=1.
\]
Lemma \ref{lem:sup-norm} together with Lemma \ref{lem:covering} imply that 
\[
\theta_5\leq \left(2\vee\frac{4\cdot 3^2}{\tau}\right)\left(1+ 3\cdot 2^5\sqrt{2d}\left(1+\frac{\sigma}{\tau}\right)^{d/2}
\left(\frac{1+\left(\frac{25}{71}\right)^2}{1-\frac{1}{9\cdot 2^{12}}}\right)^{d/4}\right).
\]
\end{proof}

\section{Numerical experiments}
\label{sec:numerical}

In this section, we present some numerical experiments consistent with our results.
%All the figures referenced in this section can be found in the Appendix.
%We conclude that the unknown constants (\Stas{which constants are unknown?}) and the number of samples required for the theory to be applicable are indeed very small.

\subsection{Spheres of varying dimension in $\mathbb{R}^D$}
\noindent We consider $n$ points $X_1,\dots,X_n$ sampled i.i.d. from the uniform distribution on the unit sphere in $\mathbb{R}^{d+1}$
$$
\m M=\mathbb{S}^d:=\{x\in\mathbb{R}^{d+1} : \Vert x\Vert=1\}\,.
$$
We then embed $\mathbb{S}^d$ into $\mathbb{R}^D$ for $D\in\{10,100\}$ by applying a random orthogonal transformation $\mathbb{R}^{d+1}\rightarrow\mathbb{R}^D$. 
%(sampled uniformly from the appropriate embedding of a Stiefel manifold) to each member of $\mathbb{S}^d$. 
Of course, the actual realization of this projection is irrelevant since our construction is invariant under orthogonal transformations. 
After performing this embedding, we add two types of noise. 
In the first case, we add Gaussian noise $\xi$ with distribution $\mathcal{N}(0,\frac{\sigma^2}D I_D)$: the scaling factor $\frac1D$ is chosen so that $\mathbb{E}\|\xi\|^2=\sigma^2$.
Since the norm of a Gaussian vector is tightly concentrated around its mean, this model is well-approximated by the ``truncated Gaussian'' model where the distribution of the additive noise is the same as the conditional distribution of $\xi$ given $\|\xi\|\leq C\sigma$, where $C$ is such that $C\sigma < 1$.   
% and in fact $\|\eta\|^2$ is highly concentrated around its mean $\sigma^2$. 
In this case, the constants in $(1,C\sigma)$-model assumption would be prohibitively large, so instead we can verify the conditions given in Remark \ref{rmk:tau-sigma} directly: due to symmetry, we have that for any $A\subset \mathbb S^d$, 
\begin{align}
\label{eq:meas}
\Pi\l(\proj{}_{\m M}^{-1}(A)\r)= U_\m M(A)=U_{\m M_{\sigma}}\l(\proj{}_{\m M}^{-1}(A)\r).
\end{align}
On the other hand, it is a simple geometric exercise to show that, for any $B$ such that  
$B(y,r)\cap \m M_\sigma \subseteq B\subseteq B(y,12r)$ and $\tau/2=1/2>r\geq 2C\sigma$,
\[
\proj{}_{\m M}^{-1}\l(\m M\cap B(y,\tilde r_1)\r)\supseteq B\supseteq \proj{}_{\m M}^{-1}\l(\m M\cap B(y,\tilde r_2)\r),
\]
where $\tilde r_1=\frac{r}{\sqrt{\frac{1+\sqrt{1-r^2}}{2}}}$ and $\tilde r_2 =  r \sqrt{\frac{3}{4(1+C\sigma)}}$. 
Lemma \ref{lem:volUB} and (\ref{eq:meas}) imply that 
\begin{align*}
\Pi(B)&
\leq U_{\m M_{\sigma}}\l(\proj{}_{\m M}^{-1}\l(\m M\cap B(y,\tilde r_1)\r)\r) \\
&
\leq\left(1+C\sigma\right)^d \vol (\m M\cap B(y,\tilde r_1))\frac{\Vol(B_{D-d}(0,C\sigma))}{\Vol(\m M_{C\sigma})}
\end{align*}
and 
\begin{align*}
U_{\m M_\sigma}(B)&
\geq U_{\m M_\sigma}\l( \proj{}_{\m M}^{-1}\l(\m M\cap B(y,\tilde r_2)\r) \r) \\
&
\geq
\left(1-C\sigma\right)^d\vol(\m M\cap B(y,\tilde r_2))\frac{\Vol(B_{D-d}(0,C\sigma))}{\Vol(\m M_{C\sigma})},
\end{align*}
hence $\Pi(B)\leq \tilde \phi_2 U_{\m M_\sigma}(B)$ for some $\tilde\phi_2$ independent of the ambient dimension $D$.
%In this way (up to a small number of samples for which $||\eta_i||^2\gg\sigma^2$), this data set almost satisfies the $(1,(1+\frac1{\sqrt{D}})\sigma)$-model assumption. 

We present the behavior of the $L^2(\Pi)$ error in this case in Figure \ref{f:spherenormalnoise}, and the rate of approximation at the optimal scale as the number of samples varies in Figure \ref{f:sphereoptimalMSE}, where it is compared to the rates obtained in Corollary \ref{cor:noiseless}. 
From Figure \ref{f:spherenormalnoise}, we see that the approximations obtained satisfy our bound, and are typically better even for a modest number of samples in dimensions non-trivially low (e.g. $8000$ samples on $\mathbb{S}^8$). 
In fact, the robustness with respect to sampling is such that the plots barely change from row to row.

The second type of noise is uniform in the radial direction, i.e. we let $\eta\sim\mathrm{Unif}[1-\sigma,1+\sigma]$ and each noisy point is generated by $\tilde X_i=X_i+\eta_i \frac{X_i}{||X_i||}$. 
This is an example where the noise is not independent of $X$. 
Once again, it is easy to check directly that conditions of Remark \ref{rmk:tau-sigma} hold (the argument mimics the approach we used for the truncated Gaussian noise). 
% and yet our $(1,\sigma)$-model assumptions are satisfied.
Simulation results for this scenario are summarized in Figure \ref{f:sphereunifnoise}, with the rate of approximation at the optimal scale again in Figure \ref{f:sphereoptimalMSE}.

\begin{figure}[ht!]
\begin{center}
\includegraphics[width=0.32\textwidth]{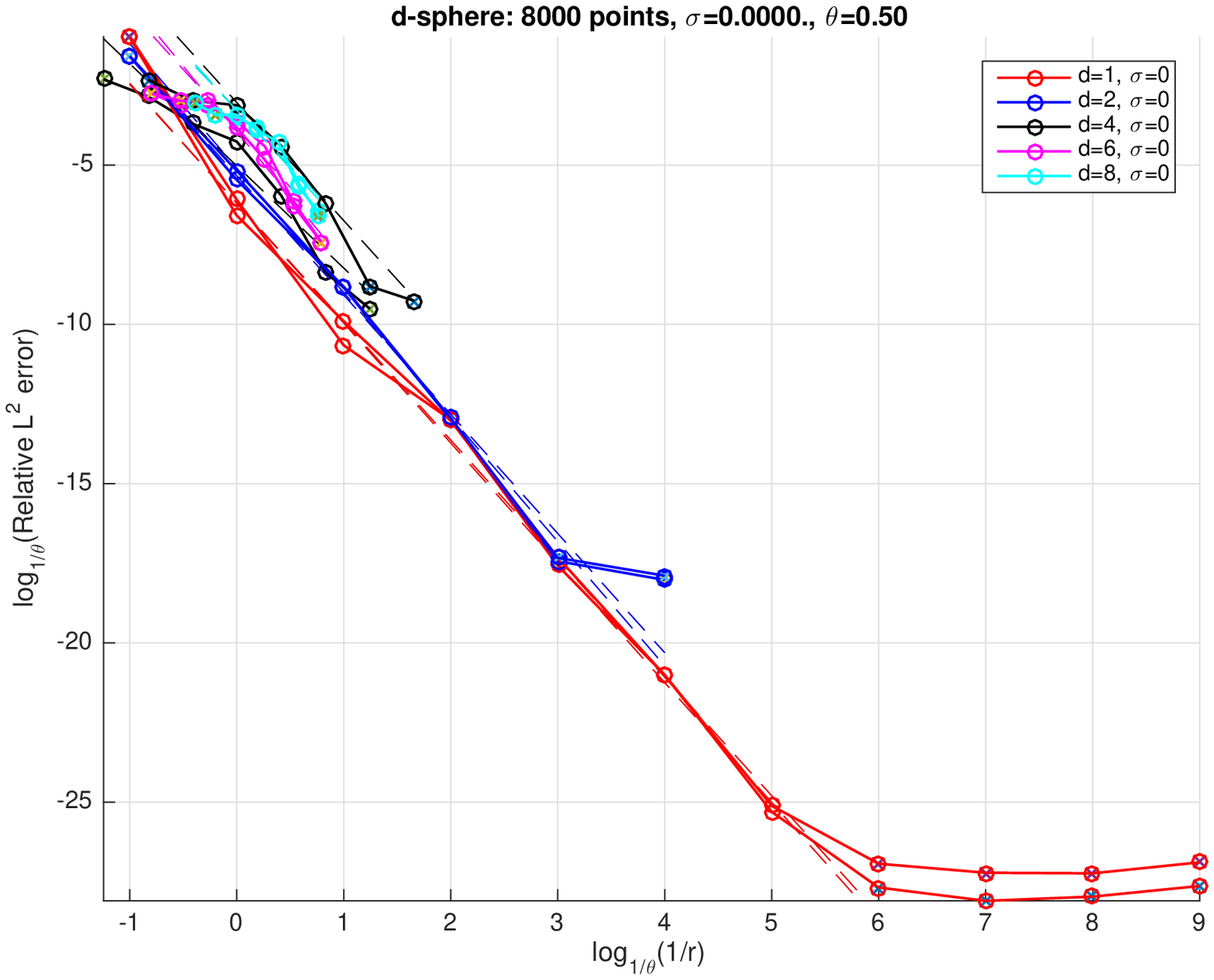}
\includegraphics[width=0.32\textwidth]{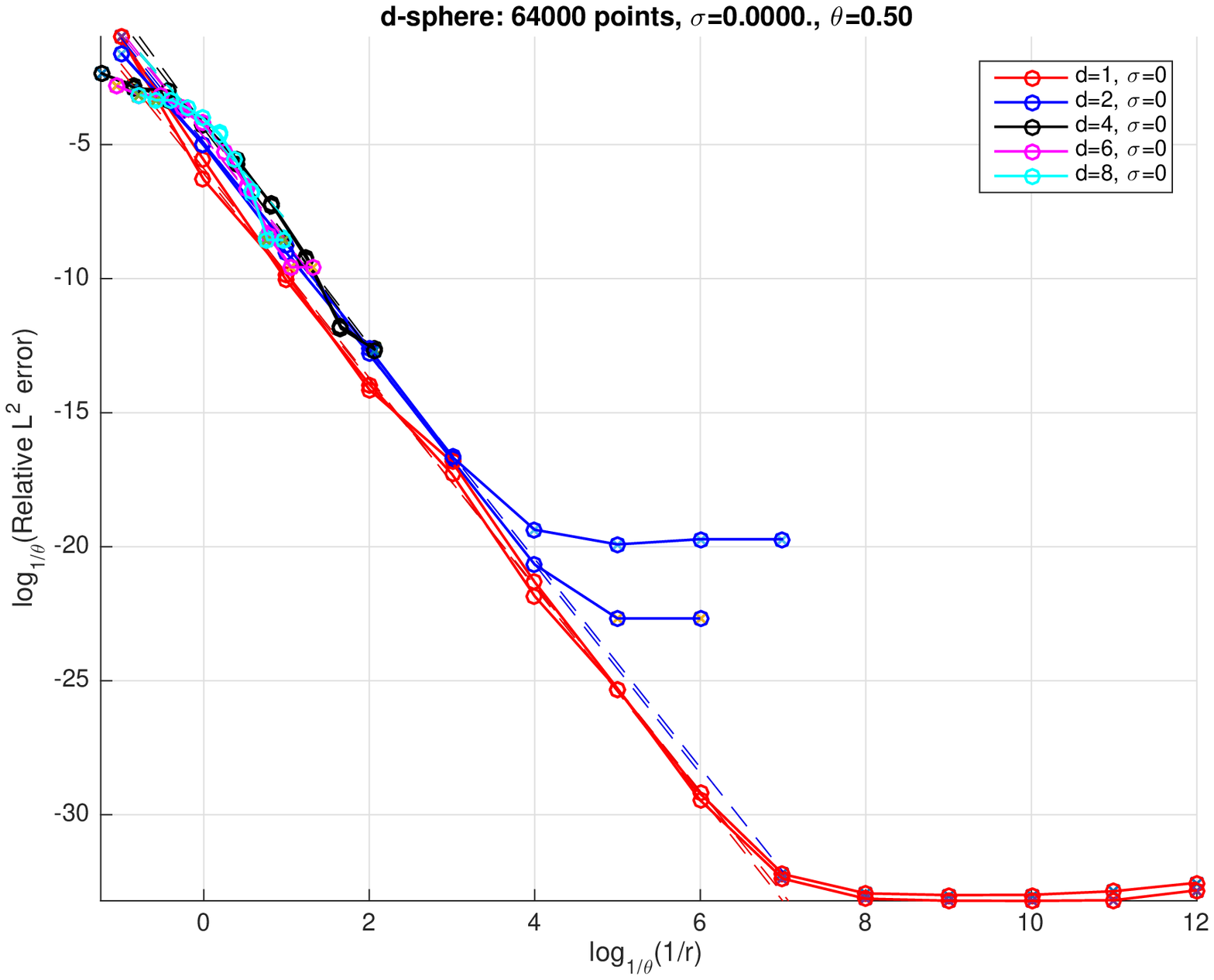}
\includegraphics[width=0.32\textwidth]{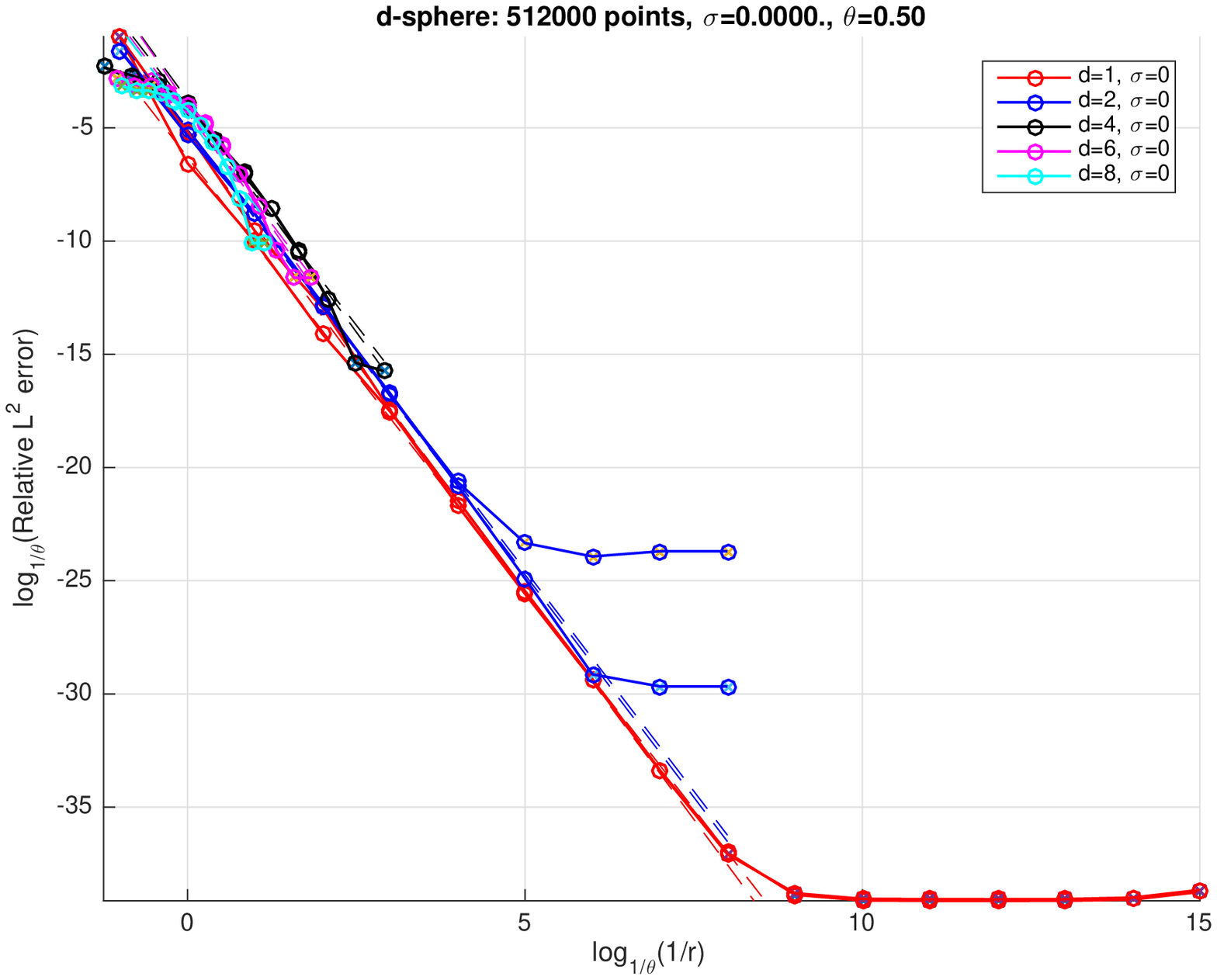}
\includegraphics[width=0.32\textwidth]{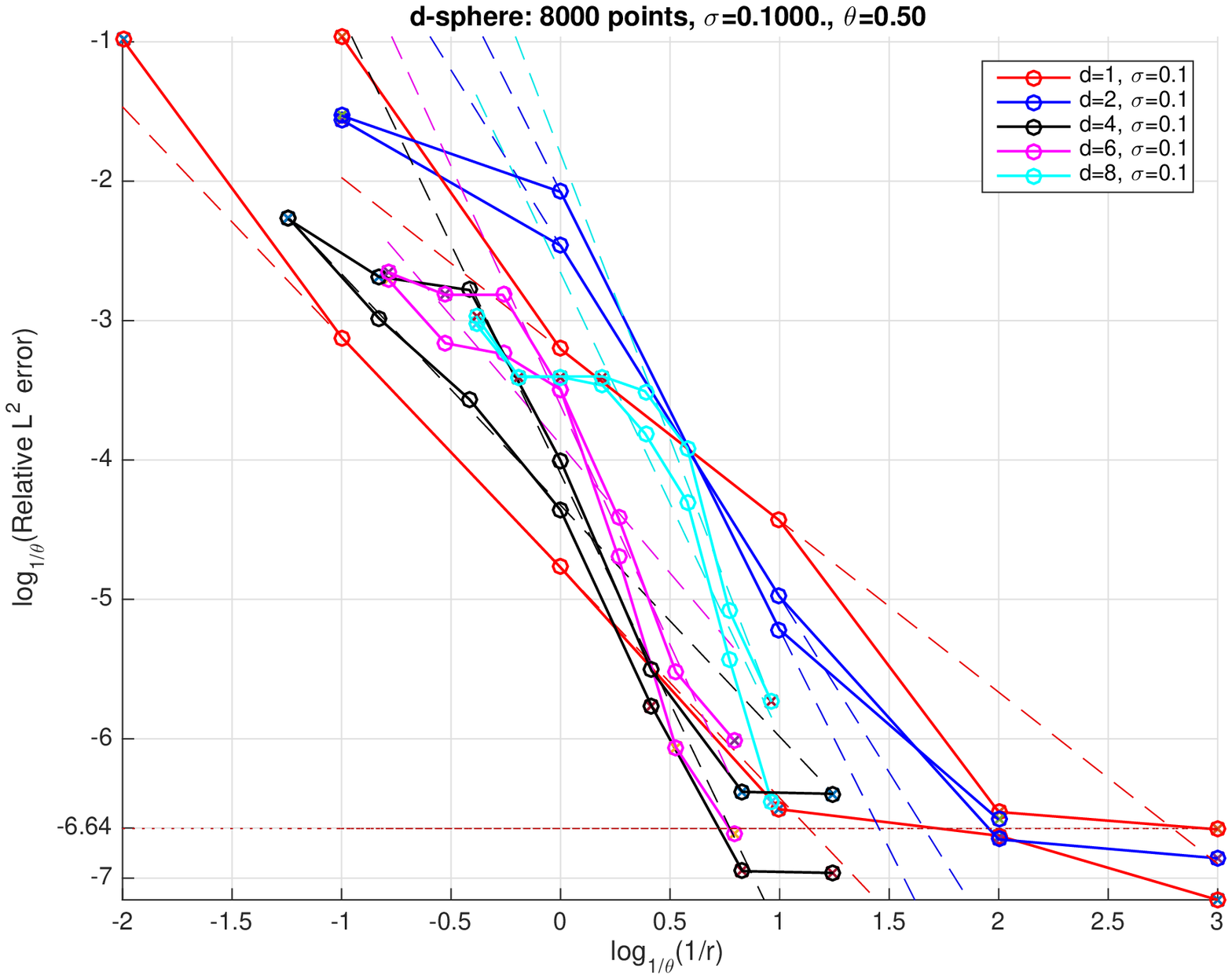}
\includegraphics[width=0.32\textwidth]{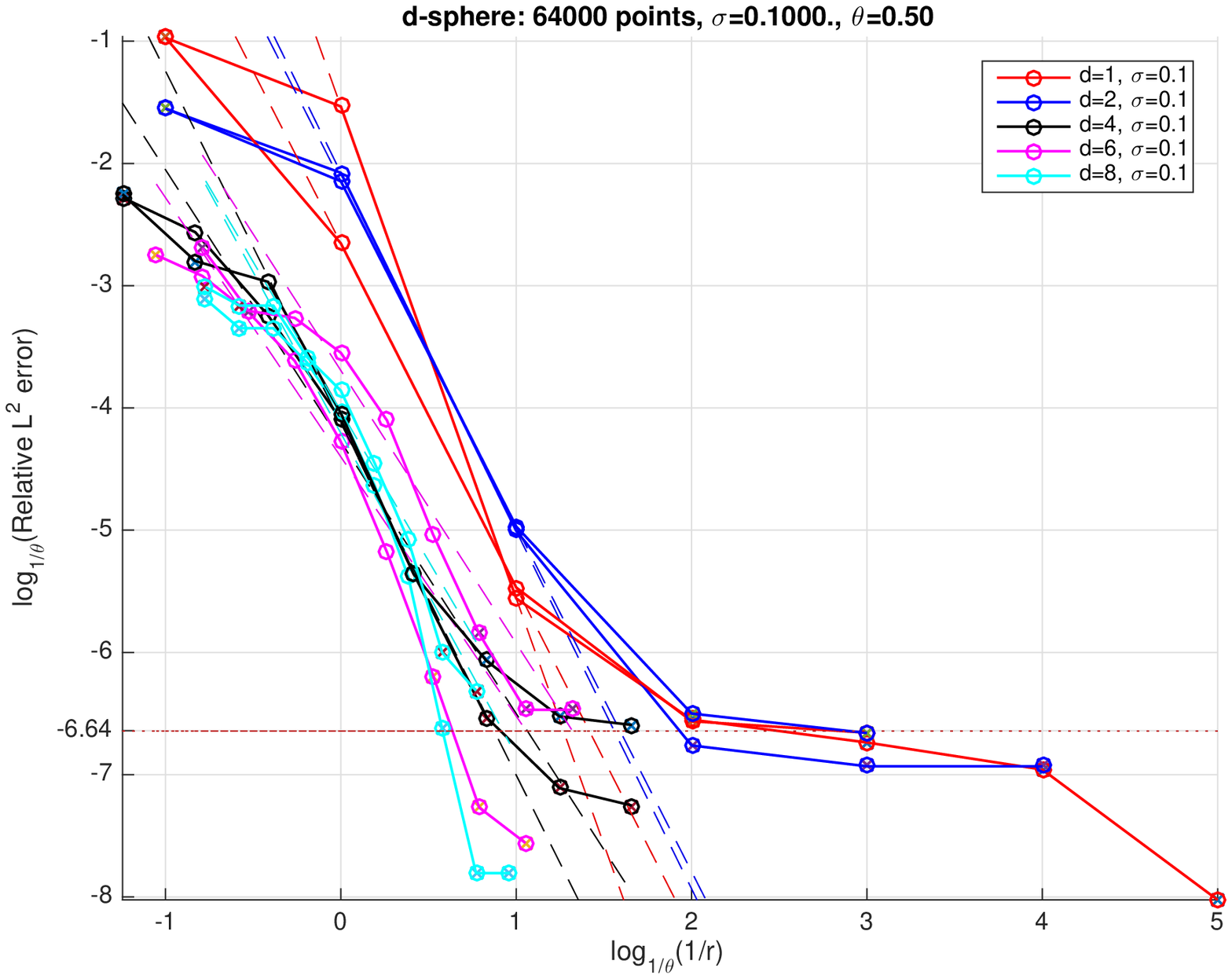}
\includegraphics[width=0.32\textwidth]{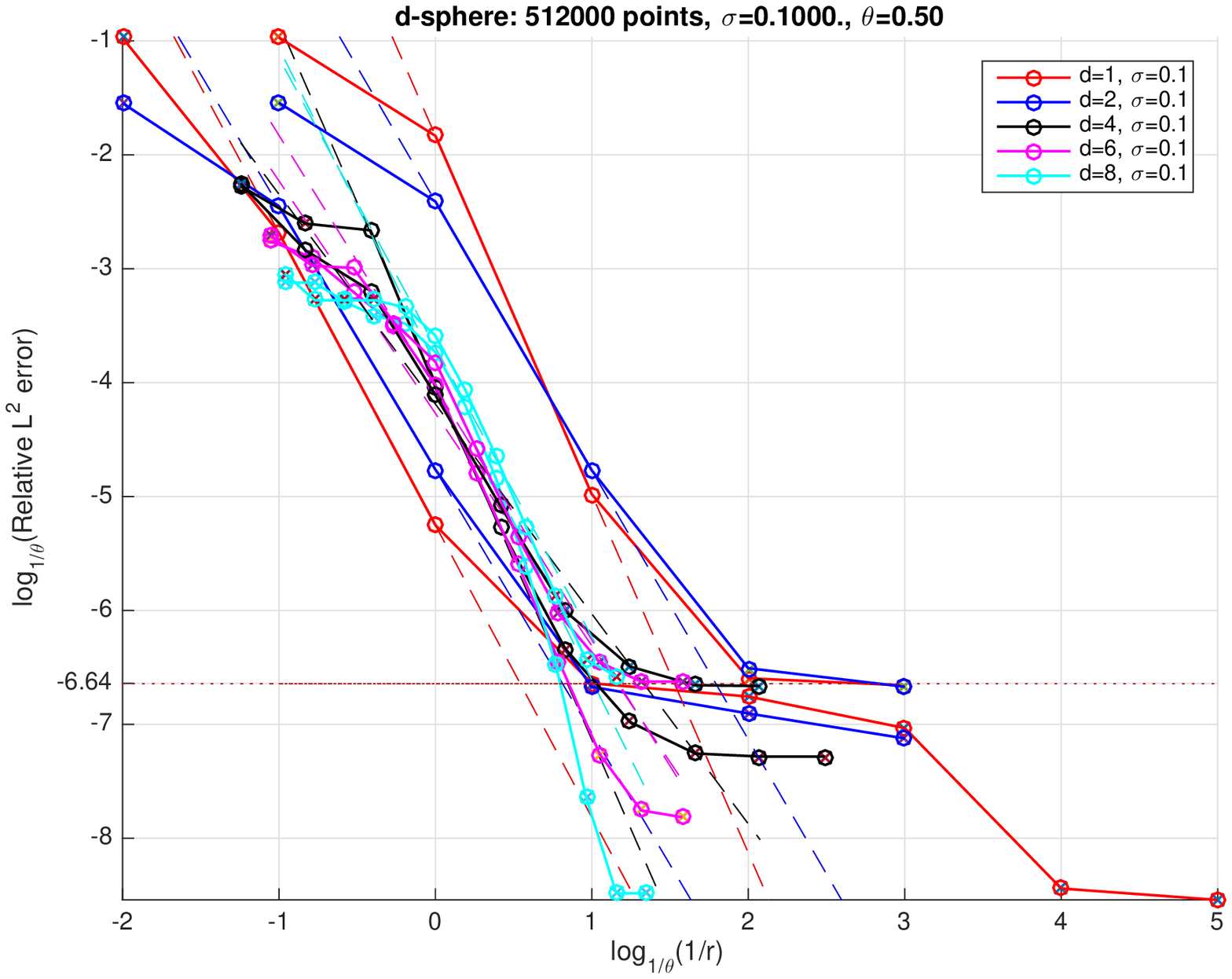}
\end{center}
\caption
{\small Experiment with $\mathbb{S}^d$, without (top row) and with Gaussian noise (bottom row). The columns correspond to different values of $n\in\{8000,64000,512000\}$. In the plots the dots represent the $L^2(\Pi)$ error squared (or MSE) of GMRA approximations (see \eqref{e:errdef}) as a function of the radius $r$ at scale $j$; more precisely the abscissa is in terms of $\log_{2}(1/r_j)$, where $r_j$ is the mean radius of $C_{j,k}$ for a fixed $j$, and the ordinate is $\log_{2}\text{MSE}_j$, where $\text{MSE}_j$ is the mean squared error of the GMRA approximation at scale $j$. Different colors correspond to different intrinsic dimensions $d$ (see legend). The two cases $D=10,100$ use the same colors for both the dots and the lines, all of which are essentially superimposed since our results are independent of the ambient dimension $D$. For each dimension we fit a line to measure the decay, which is $O(r^{-4})$ independently of $d$, consistently with our analysis. The horizontal dotted line, with corresponding tick mark on the $Y$ axis, represents the noise level $\sigma^2$: the approximation error flattens out at roughly that level, as expected.
}
\label{f:spherenormalnoise}
%\label{fig:dpshere}
\end{figure}

%\clearpage
We considered various settings of the parameters, namely all combinations of: $d\in\{1,2,4,6,8\}$, $n\in\{8000,16000,32000,64000,128000\}$, $D\in\{100,1000\}$, $\sigma\in\{0,0.05,0.1\}$.  We only display some of the results for reasons of space constraints.
\footnote{The code provided at \url{www.math.duke.edu/~mauro/code.html} can generate all the figures, re-create the data sets, and is easily modified to do more experiments.}

\begin{figure}[ht!]
\begin{center}
\includegraphics[width=0.32\textwidth]{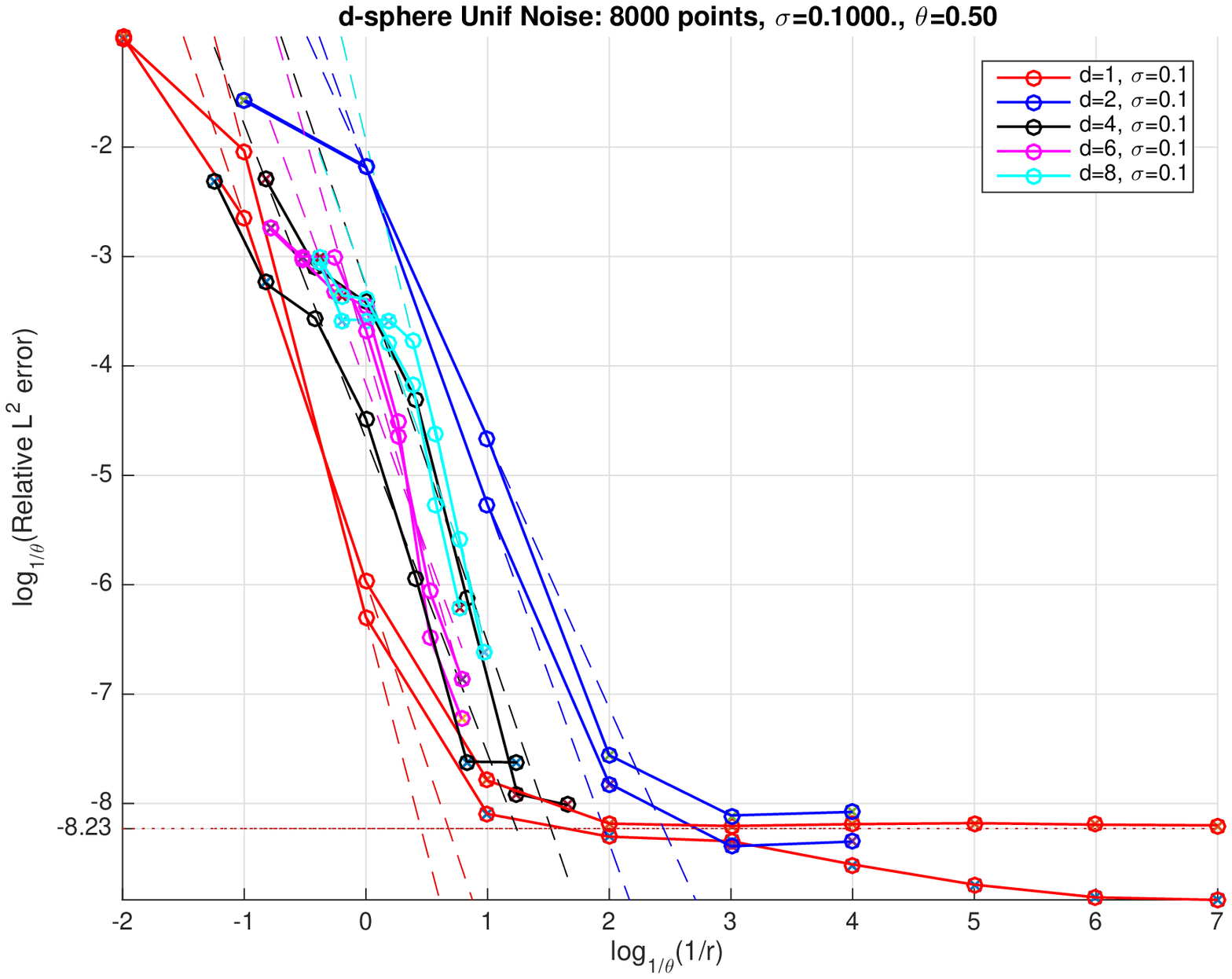}
\includegraphics[width=0.32\textwidth]{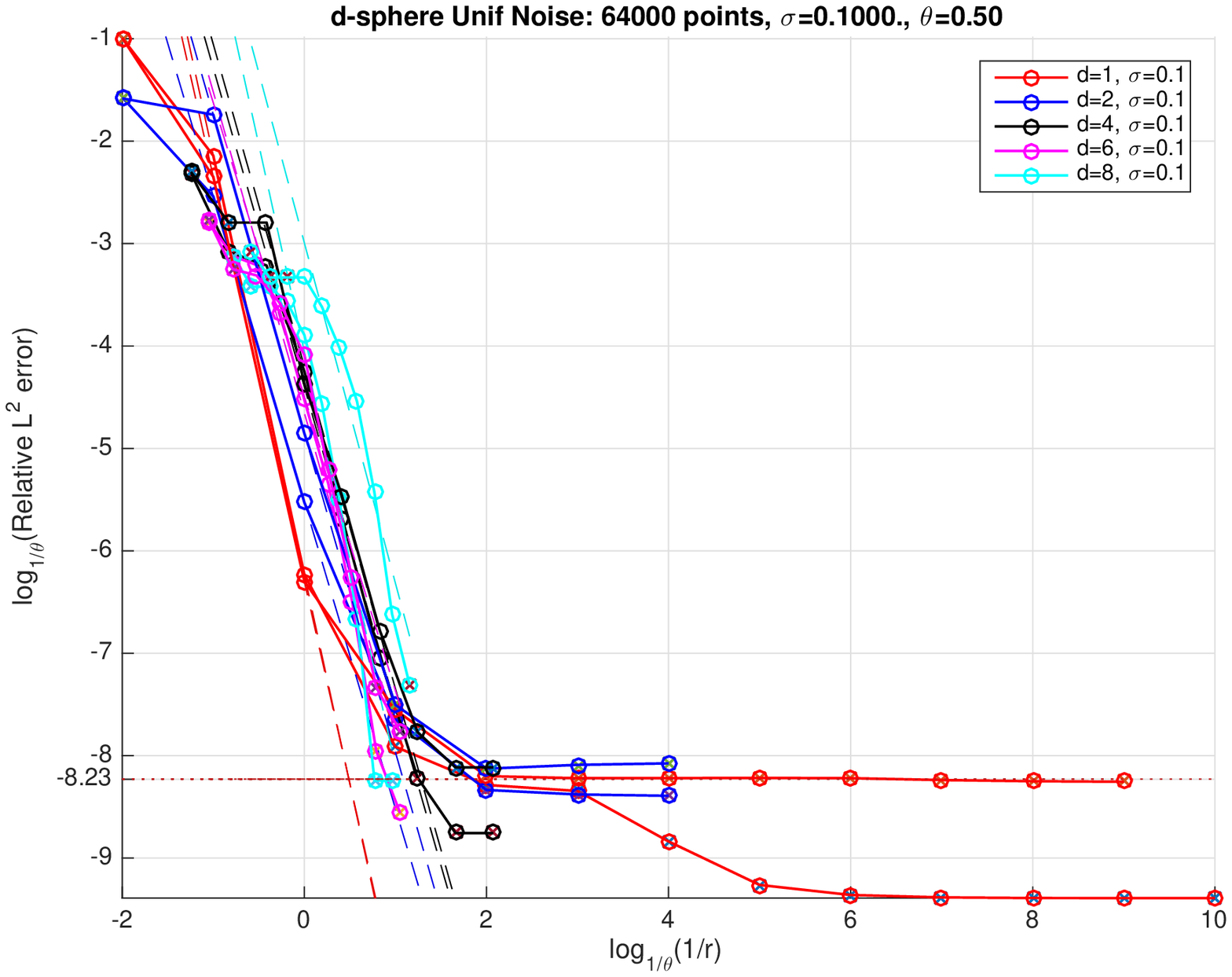}
\includegraphics[width=0.32\textwidth]{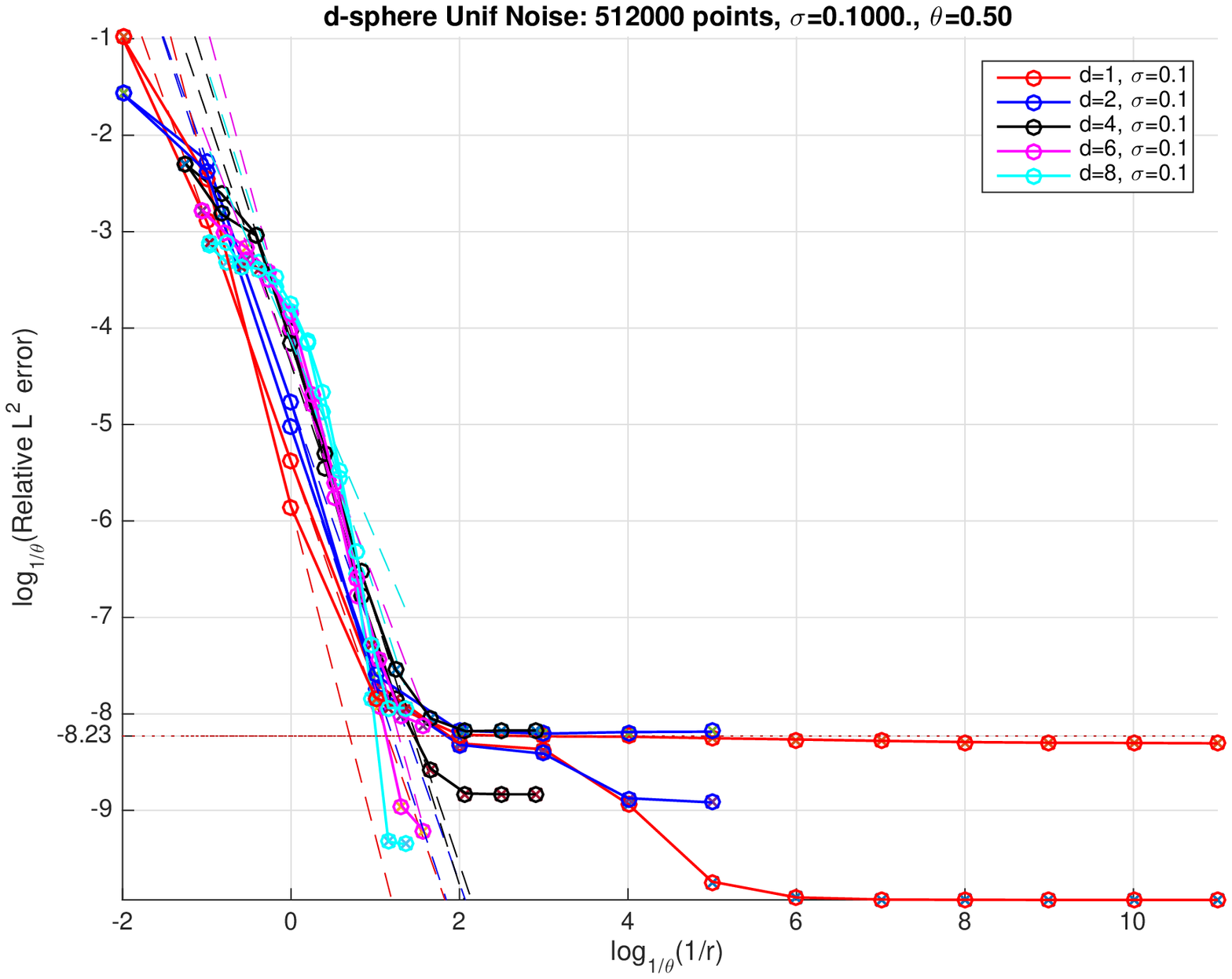}
\end{center}
\caption{\small This figure is as the second row of Figure \ref{f:spherenormalnoise}, but the noise is radially uniform with widht parameter $\sigma$. Note that the variance of the noise is $\sigma^2/3$, which is indicated in the figure by horizontal line and an extra tick mark on the $Y$-axis in the figures. The MSE converges quickly to that level as a function of scale.}
\label{f:sphereunifnoise}
\end{figure}

\begin{figure}[ht!]
\begin{center}
\includegraphics[width=0.32\textwidth]{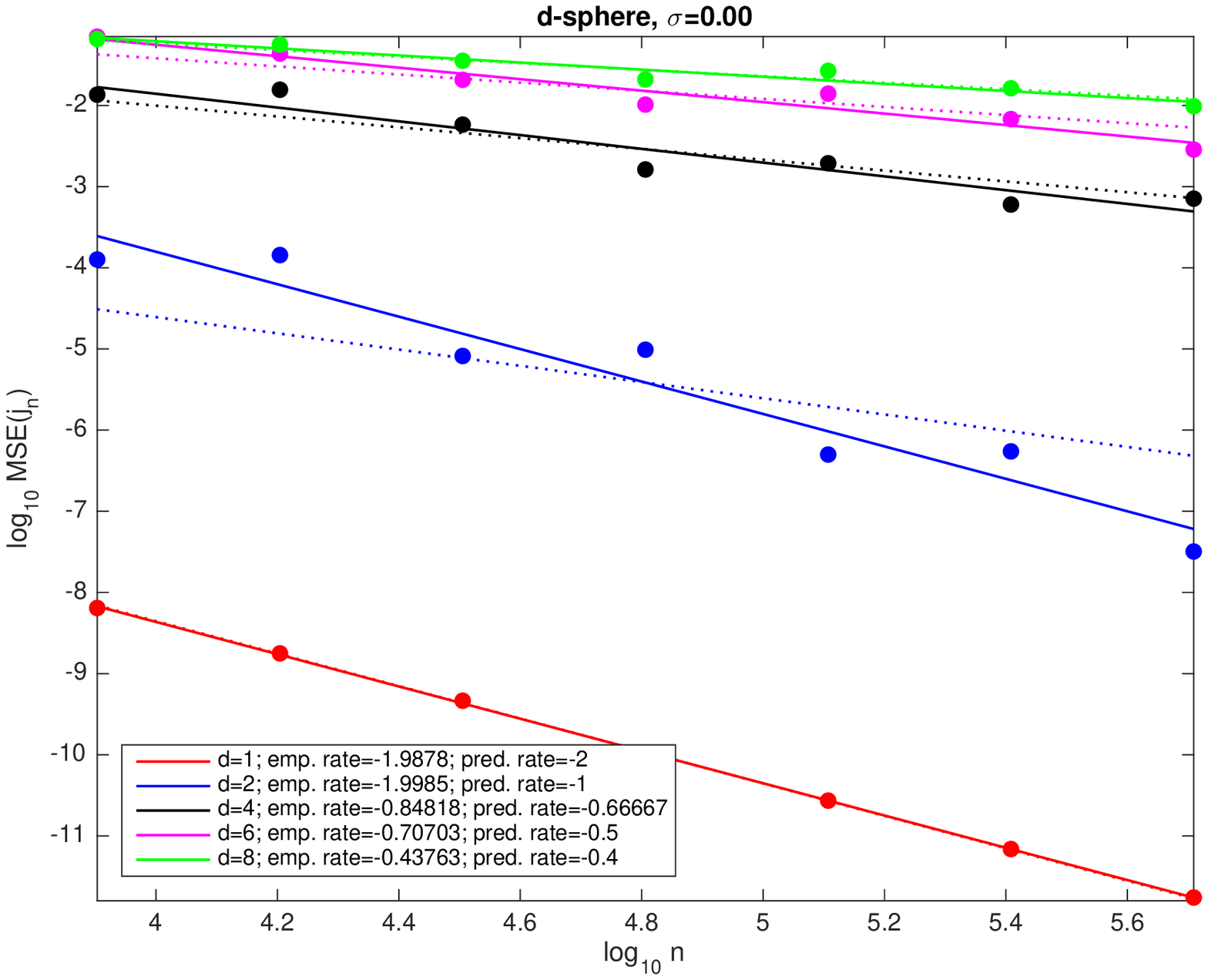}
\includegraphics[width=0.32\textwidth]{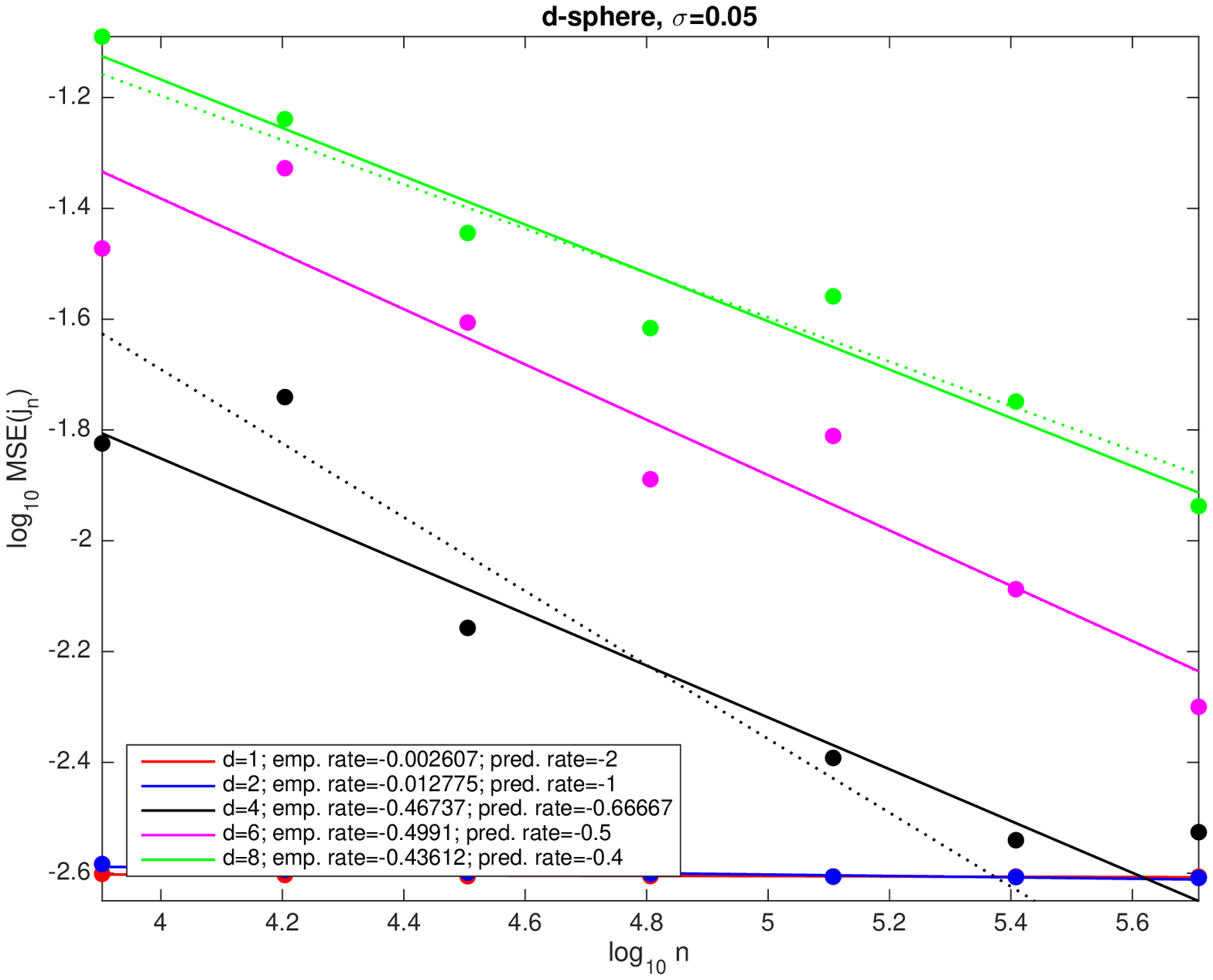}
\includegraphics[width=0.32\textwidth]{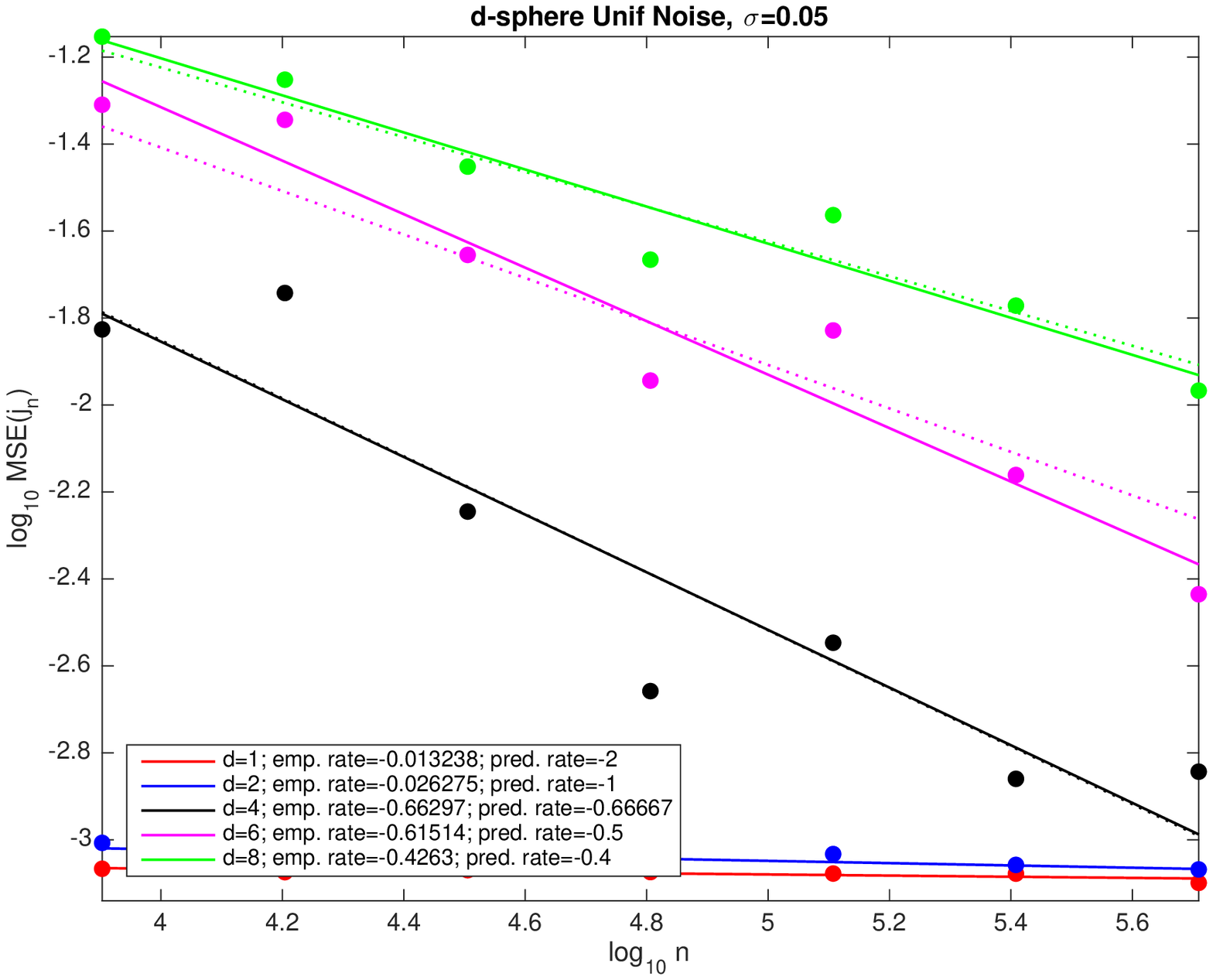}
\end{center}
\caption{For the example of $\mathbb{S}^d$ considered in this section we consider the MSE error, i.e. $L^2(\Pi)$ squared error (as defined in \eqref{e:errdef}) at the optimal scale $j_n$ (as in the proof of Corollary \ref{cor:noiseless}) as a function of the number of points $n\in\{8000,16000,32000,64000,128000,256000,512000\}$, and compare our empirical rates (solid linear, with the rate reported in the legend under ``emp. rate'') with the rates predicted by Corollary \ref{cor:noiseless} (dotted lines, with rate reported in the legend under ``pred. rate''), for various choices of the intrinsic dimension $d\in\{1,2,4,6,8\}$ and fixed ambient dimension $D=10$ (the results are independent of $D$, so we do note report the - very similar - results obtained for $D=100$). Left: noiseless case, middle: Gaussian noise, right: radial uniform noise (see text). The rates match our results quite well, except in the case $d=2$ where we seem to obtain the same convergence rate as in the $d=1$ case. Here we are choosing the optimal scale to be the finest scale such that, in every cell, we have at least $10d^2$ points. For the noisy cases, the approximation rates for $d=1,2$ are not meaningful simply because we have enough points to go the finest scale above the noise level.}
\label{f:sphereoptimalMSE}
\end{figure}

\subsection{Meyer's staircase}
We consider the ($d$-dimensional generalization of) Y. Meyer's staircase. 
Consider the cube $Q=[0,1]^d$ and the set of Gaussians $\mathcal{N}(x;\mu,\delta^2 I_d)$ where the mean $\mu$ is allowed to vary over $Q$, and the function is truncated to only accept arguments $x\in Q$. 
Varying $\mu$ in $Q$ in this manner induces a smooth embedding of a $d$-dimensional manifold into the infinite dimensional Hilbert space $L^2(Q)$. 
That is, the Gaussian density centered at $\mu\in Q$ and truncated to $x\in Q$ is a point in $L^2(Q)$. 
By discretizing $Q$, we may sample this manifold and project it into a finite dimensional space. In particular, a grid $\Gamma_D\subseteq Q$ of $D$ points (obtained by subdividing in $D^{-\frac1d}$ equal parts along each dimension) may be generated and considering the evaluations of the set of translated Gaussians on this grid produces an embedding of this manifold into $\mathbb{R}^D$. Sampling $n$ points from this manifold by randomly drawing $\mu_1,\dots, \mu_n$ uniformly from $Q$, we obtain a set $\{\mathcal{N}(x;\mu_i,\delta^2I_d)|_{\Gamma_D}\}_{i=1,\dots,n}$ of $n$ samples from the ``discretized'' Meyer's staircase in $\mathbb{R}^D$. 
This is what we call a sample from Meyer's staircase, which is illustrated in Figure \ref{f:meyer_example}. This example is not artificial: for example, translating a white shape on a black background produces a set of $2-D$ images with a similar structure to the $d$-dimensional Meyer's staircase for $d=2$.

\begin{figure}[ht!]
\begin{minipage}[c]{0.57\textwidth}
\includegraphics[width=\textwidth]{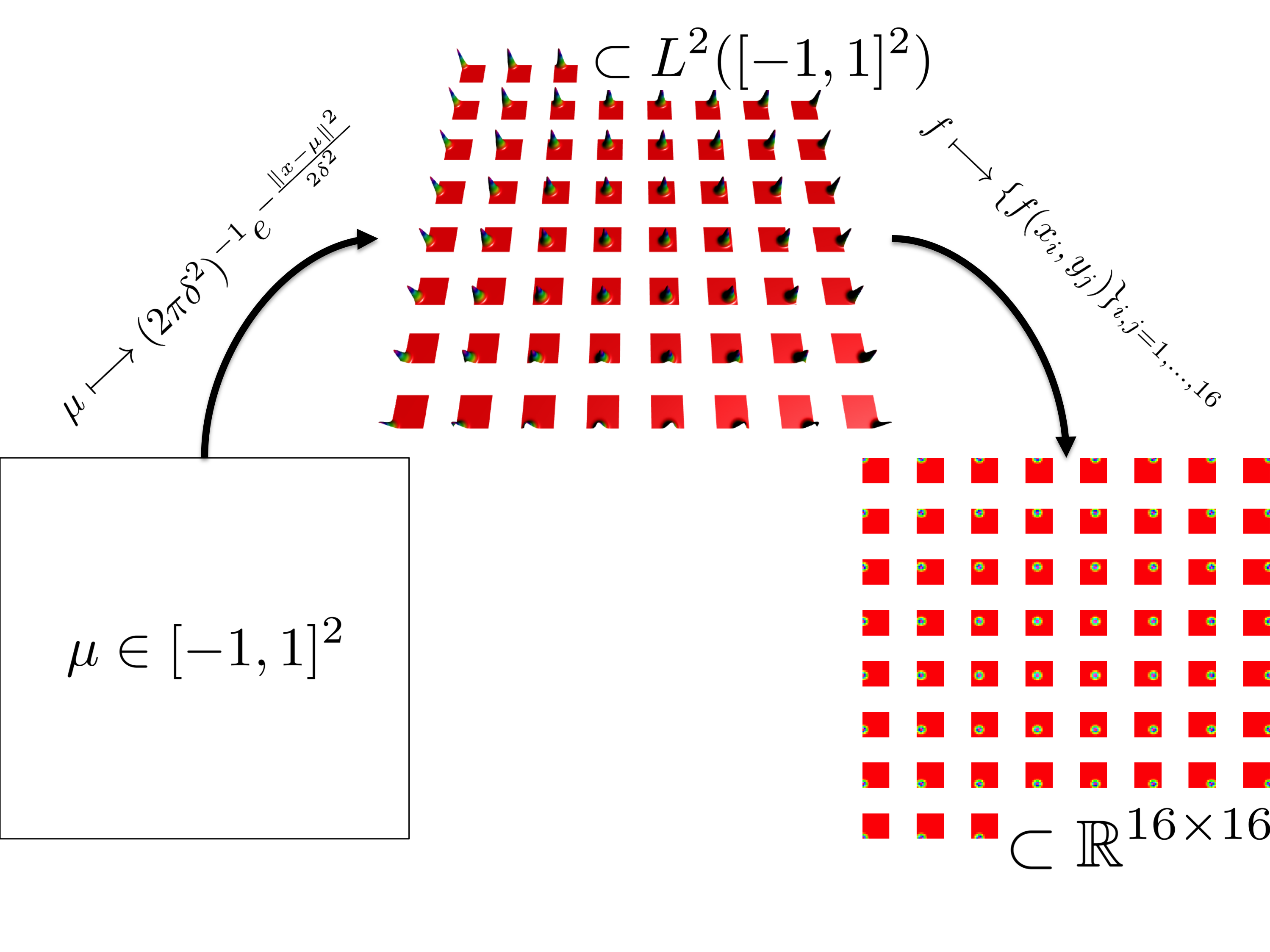}
\end{minipage}
\hfill
\begin{minipage}[c]{0.4\textwidth}
\caption{An illustration of Meyer's staircase for $d=2$. We see that the square is mapped into a subset of $L^2([-1,1]^2)$ consisting of truncated Gaussians. These are then sampled at points on a uniform, $16$ by $16$ grid to obtain an embedding of $[0,1]^2$ into $\mathbb{R}^{16\times 16}$. For small $\delta$, this embedding has a point very close to each coordinate axis in $\mathbb{R}^{16\times 16}$. Thus, it comes as no surprise that this embedding of $[-1,1]^2$ into $\mathbb{R}^{256}$ has a high degree of curvature.}
\label{f:meyer_example}
\end{minipage}
\end{figure}

The manifold associated with  Meyer's staircase is poorly approximated by subspaces of dimension smaller than $O(D\wedge 1/\delta^D)$, and besides spanning many dimensions in $\mathbb{R}^D$, it has a small reach, depending on $d,D,\delta$. In our examples we considered 
\[
n=8000,16000,32000,640000,128000, \  d=1,2,4, \ D=2000, \text { and } \delta=\frac5{100}.
\] 
We consider the noiseless case, as well as the case where Gaussian noise $\mathcal{N}(0,\frac{1}{D}I_D)$ is added to the data. Since this type of noise does not abide by the $(\sigma,\tau)$-model assumption and $\tau$ is very small for Meyer's staircase, Figure \ref{f:meyer} illustrates the behavior of the GMRA approximation outside of the regime where our theory is applicable.

\begin{figure}[ht!]
\begin{center}
\includegraphics[width=0.32\textwidth]{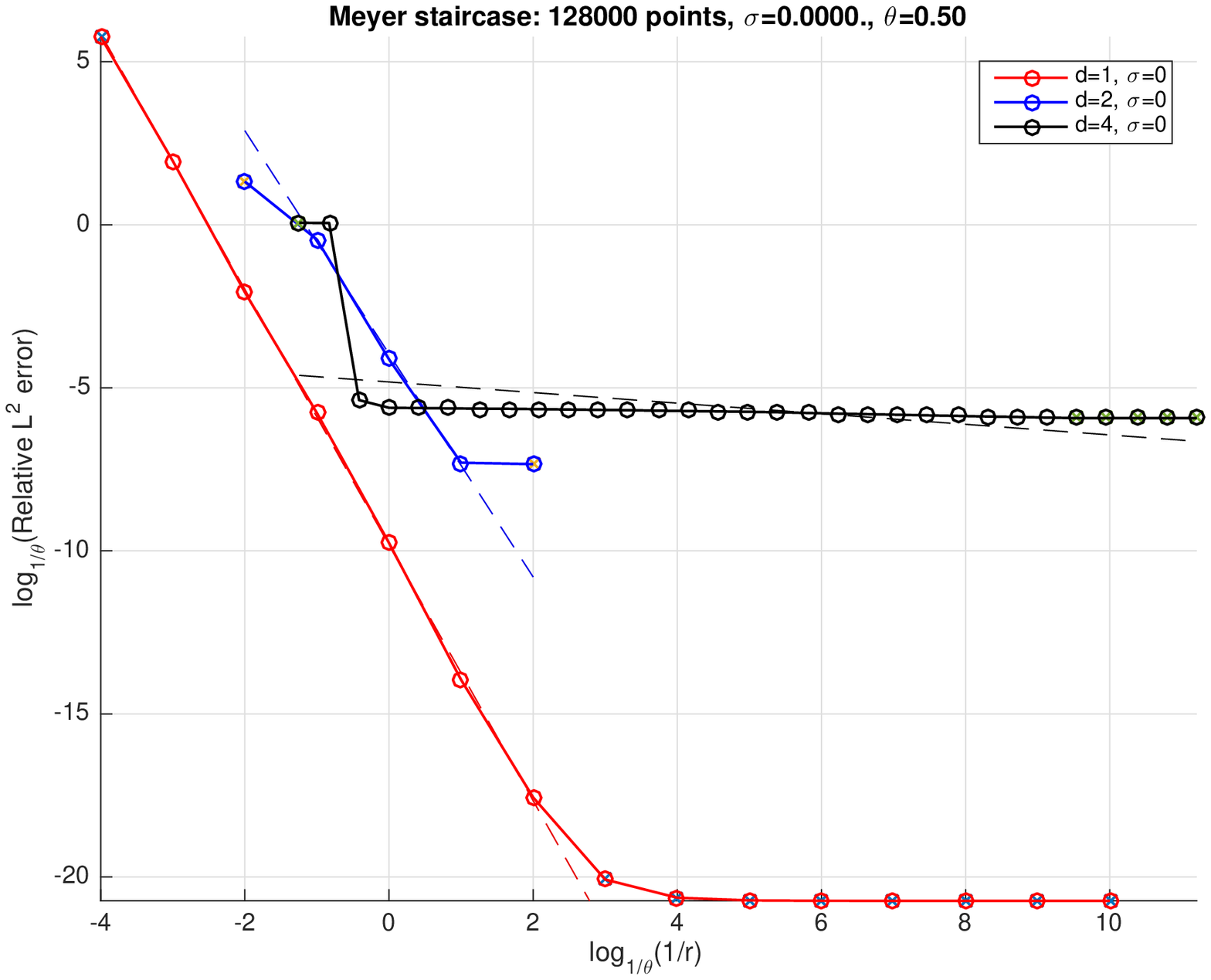}
\includegraphics[width=0.32\textwidth]{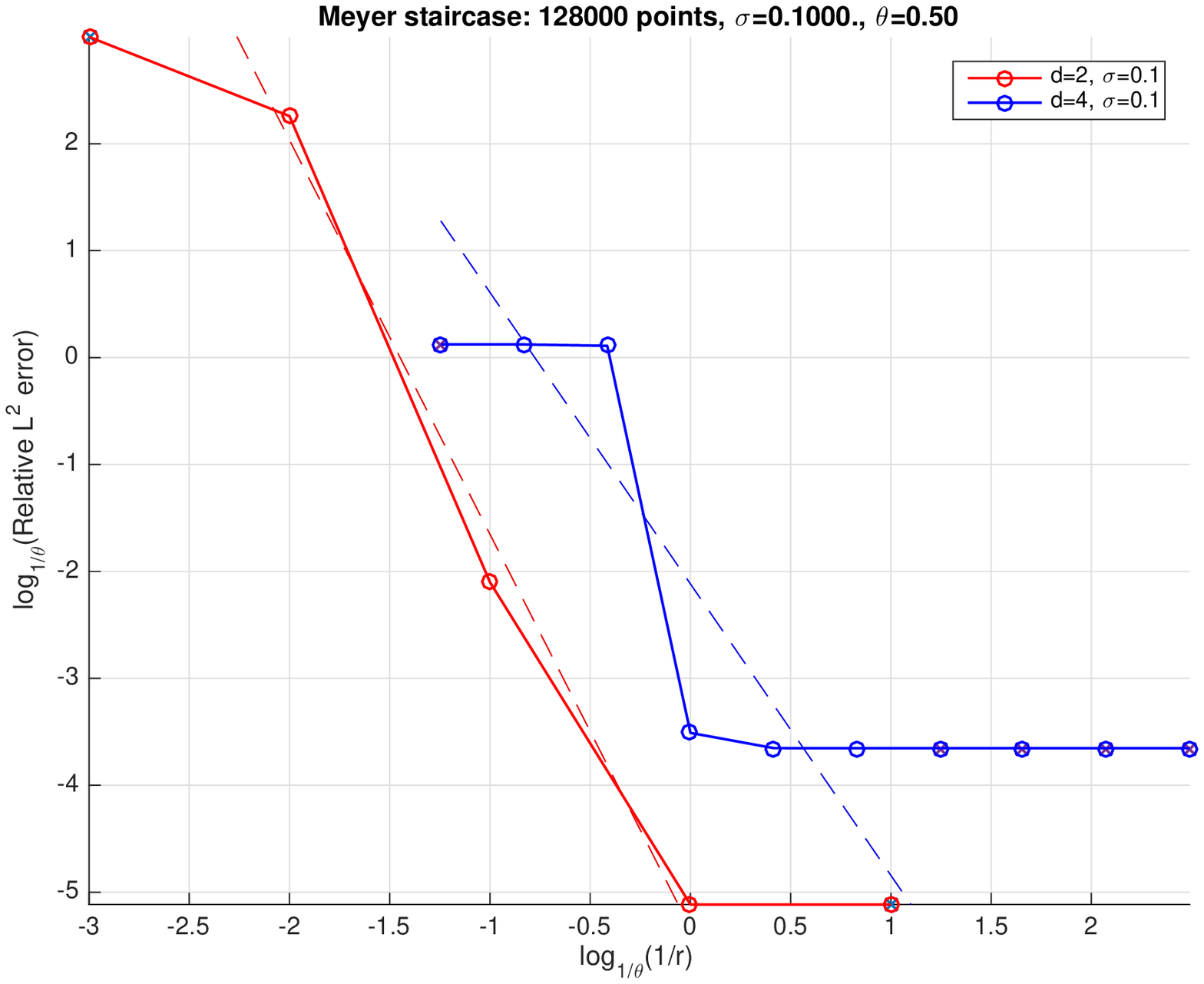}
\includegraphics[width=0.32\textwidth]{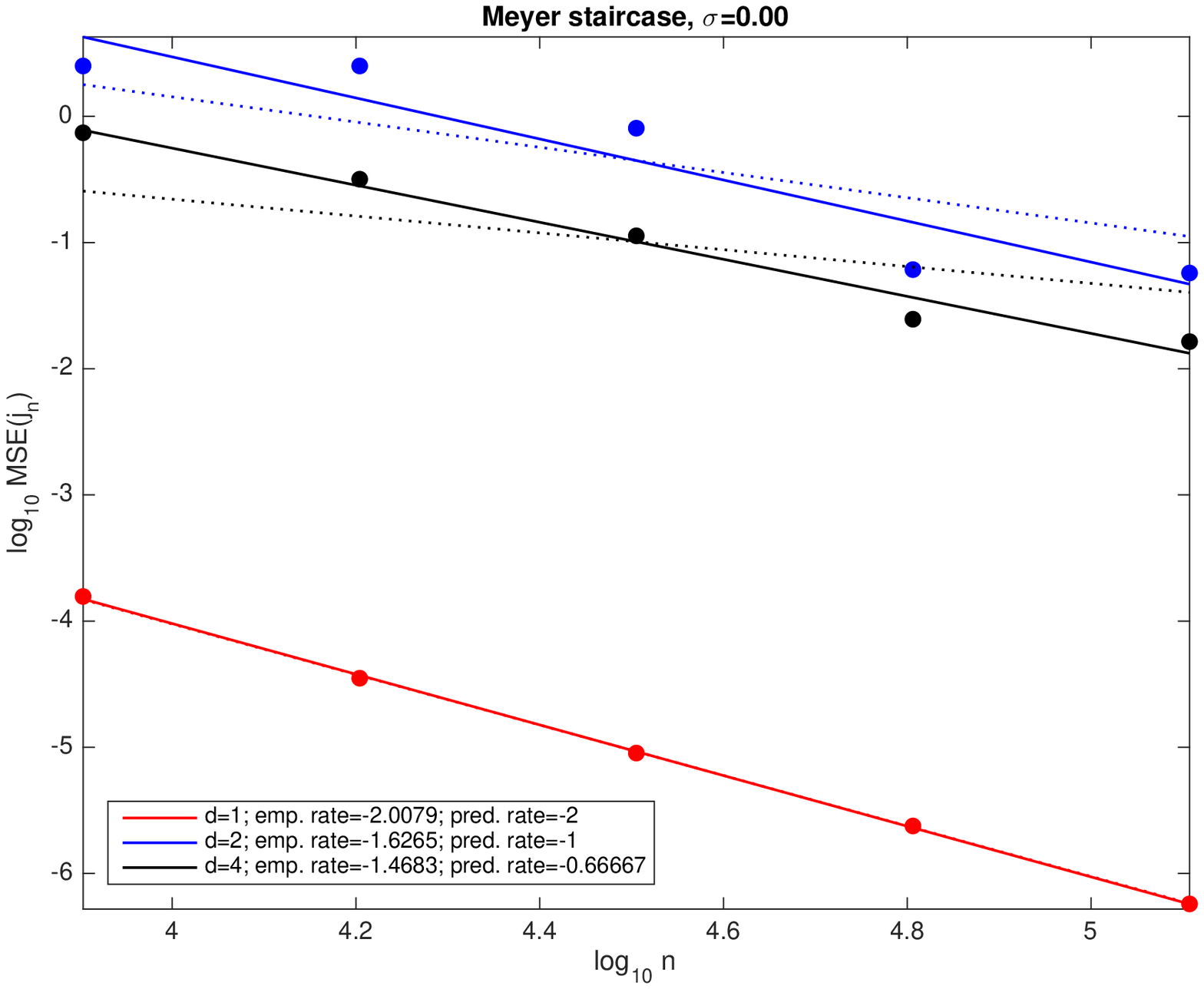}
\end{center}
\caption{Left and middle: MSE as a function of scale $r$ for the $d$-dimensional Meyer's staircase, for different values of $n=$, $d$ and $\sigma$,  standard deviation of Gaussian noise $\mathcal{N}(0,\frac{\sigma^2}{D})$. The small reach of Meyer's staircase makes it harder to approximate, and makes the approximation much more susceptible to noise. Moreover, Gaussian noise is unbounded, so this distribution violates the $(\sigma,\tau)$-model assumption (albeit only at a small number of points, with high probability). Right: MSE at the optimal scale, chosen so that every cell contains at least $10d^2$ points.}\label{f:meyernormalnoise}
\label{f:meyer}
\end{figure}

\subsection{The MNIST dataset of handwritten digits}
We consider the MNIST data set of images of handwritten digits\footnote{Available at \url{http://yann.lecun.com/exdb/mnist/.}}, each of size $28\times 28$, grayscale. There are total of $60,000$, from ten classes consisting of digits $0,1,\dots,9$.
The intrinsic dimension of this data set is variable across the data, perhaps because different digits have a different number of ``degrees of freedom'' and across scales, as it is observed in \cite{LMR:MGM1}.
We run GMRA by setting the cover tree scaling parameter $\theta$ equal to $0.9$ (meaning that we replace $1/2$ with $0.9$ in definition of cover trees in section \ref{sec:case}) in order to slowly ``zoom" into the data at multiple scales. As the intrinsic dimension is not well-defined, we set GMRA to pick the dimension of the planes $\mathbb{V}_{j,k}$ adaptively, as the smallest dimension needed to capture half of the ``energy'' of the data in $C_{j,k}$.  The distribution of dimensions of the subspaces $\mathbb{V}_{j,k}$ has median $3$ (consistently with the estimates in \cite{LMR:MGM1}) and is represented in figure \ref{f:MNIST_histdim}. We then compute the $L^2$ relative approximation error, and compute various quantiles: this is reported in the same figure. The running time on a desktop was few minutes.

\begin{figure}[ht!]
\begin{center}
\includegraphics[width=0.49\textwidth]{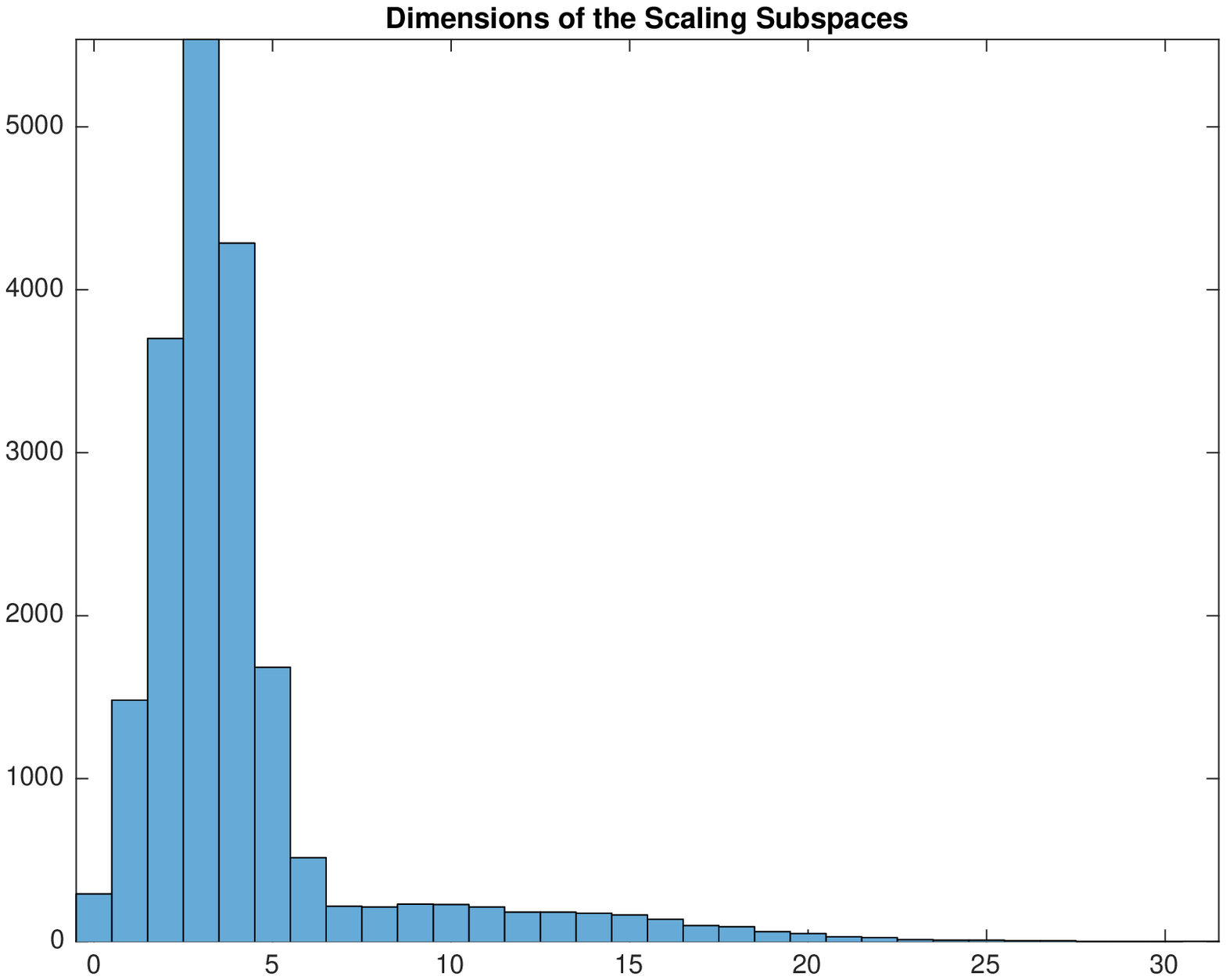}
\includegraphics[width=0.49\textwidth]{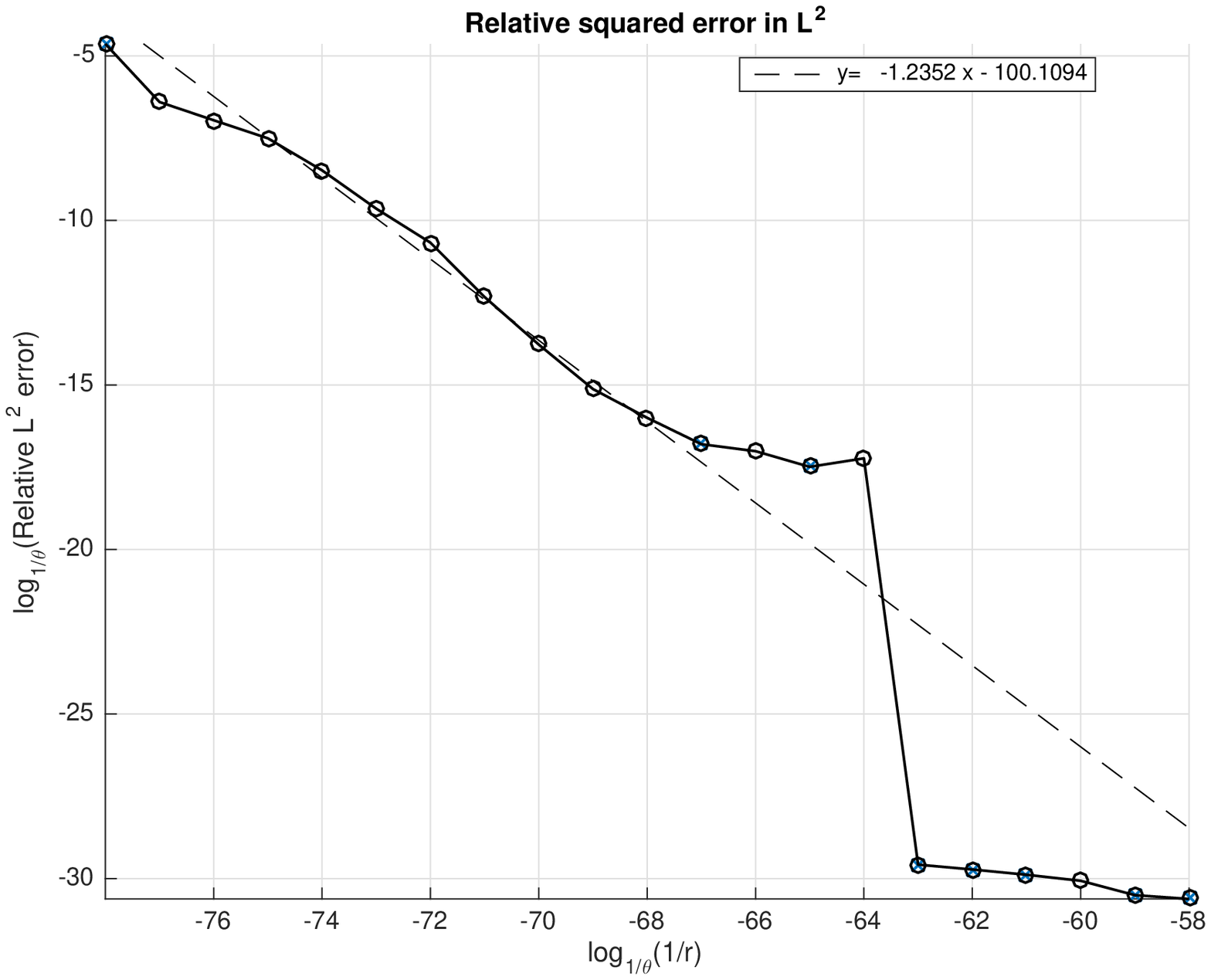}
\end{center}
\caption{Left: histogram of the dimension of the scaling function subspaces for the MNIST data set. We see that many of the subspaces are very low-dimensional, with dimensions mostly between $2$ and $5$. Right: $L^2$ relative approximation error squared as a function of scale. We do not plot the quantiles since many of them are many orders of magnitude smaller (which is a good thing in terms of approximation error), creating artifacts in the plots; they do indicate thought that the structure of the data is highly complex and not heterogeneous.  Note that the axis of this plot are in $\log_{1/\theta}$ scale, where $\theta=0.9$ is the cover tree scaling factor used in this example. Note how the approximation error decreases slowly at the beginning, as there are many classes, rather far from each other, so that it takes a few scales before GMRA starts focusing into each class, at which point the approximation error decreases more rapidly. This phenomenon does not happen uniformly over the data (figure not shown).}
\label{f:MNIST_histdim}
\end{figure}

\subsection{Sonata Kreutzer}
We consider a recording of the first movement of the Sonata Kreutzer by L.V. Beethoven, played by Y. Pearlman (violin) and V. Ashkenazy (piano) (EMI recordings). The recording is stereo, sampled at 44.1kHz. We map it to mono by simply summing the two audio channels, and then we generate a high-dimensional dataset as follows. We consider windows of width $w$ seconds, overlapping by $\delta w$ seconds, and consider the samples in each such time window $[i\delta w,i\delta w + w)$ as a high-dimensional vector $X'_i$, of dimension equal to the sampling rate times $w$. In our experiment we choose $w=0.1$ seconds, $\delta w=0.05$ seconds, and the resulting vectors $X'_i$ are $D'=551$-dimensional. Since Euclidean distances between the $X'_i$ are far from being perceptually relevant, we transform each $X'_i$ to its cepstrum (see \cite{OS:DSP}), remove the central low-pass frequency, and discard the symmetric part of the spectrum (the signal is real), obtaining $X_i$, a vector with $D=275$ dimensions, and $i$ ranges from $0$ to about $130,000$. The running time on a desktop was few minutes.

\begin{figure}[ht!]
\begin{center}
\includegraphics[width=0.49\textwidth]{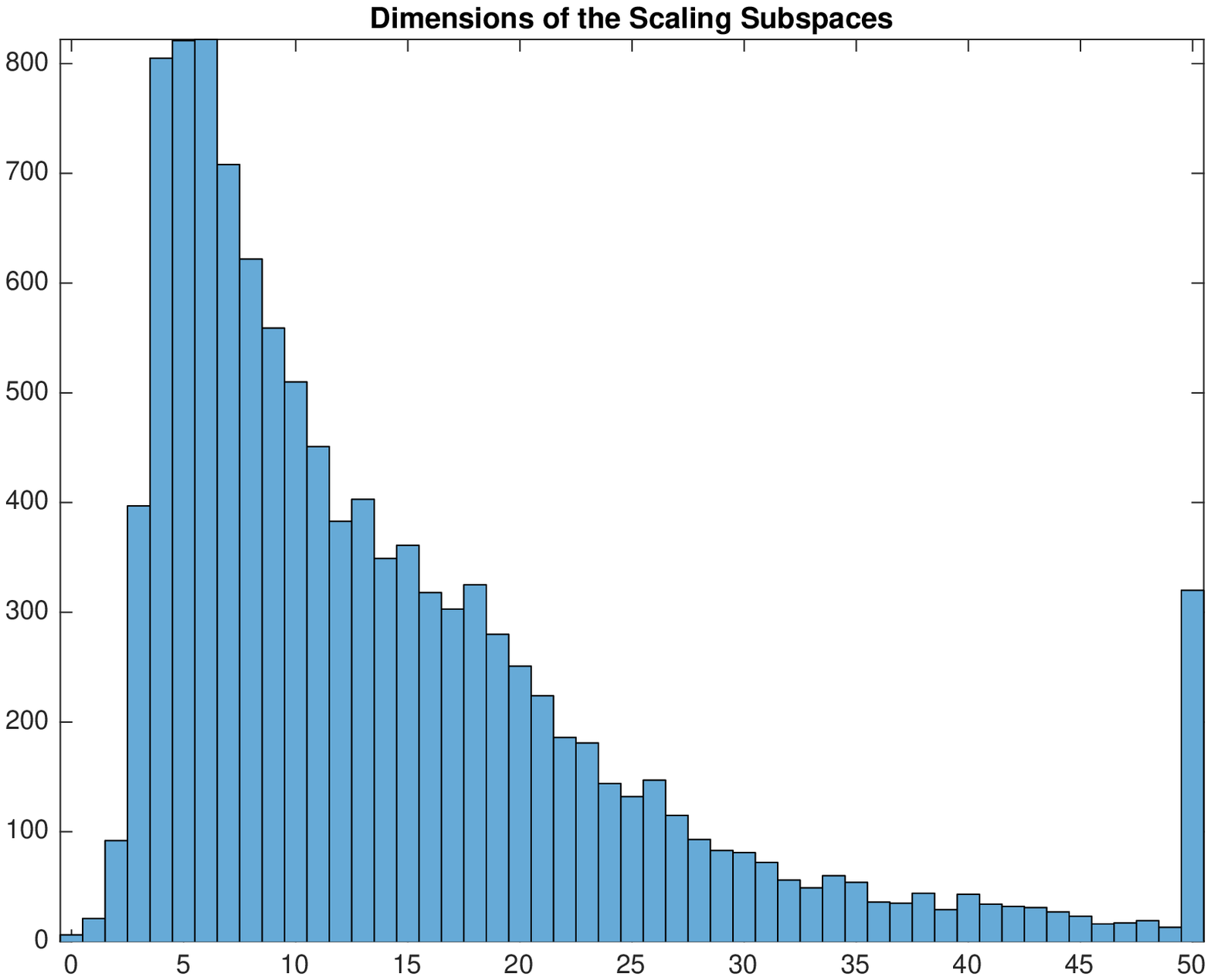}
\includegraphics[width=0.49\textwidth]{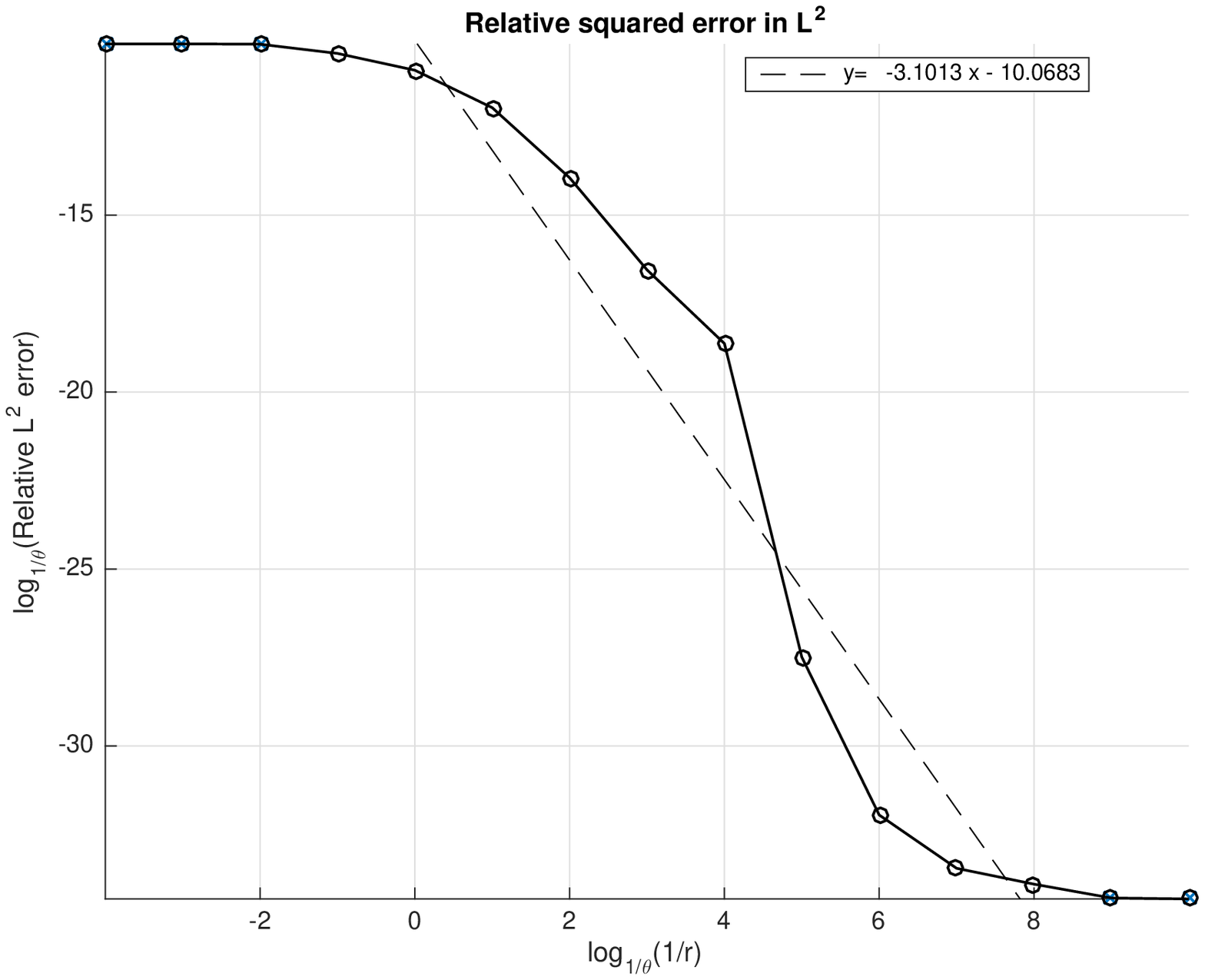}
\end{center}
\caption{Left: histogram of the dimension of the scaling function subspaces for the Kreutzer sonata dataset. We see that the dimension of the scaling function subspaces is mostly between $4$ and $25$. Right: mean  $L^2$ relative approximation error squared as a function of scale. We do not plot the quantiles since many of them are many orders of magnitude smaller (which is a good thing in terms of approximation error), creating artifacts in the plots; they do indicate that the structure of the data is highly complex and non-heterogeneous. Note that the axes of this plot are in $\log_{1/\theta}$ scale, where $\theta=0.9$ is the scaling factor used in this example.}
\label{f:Kreutzer_histdim}
\end{figure}

%\section*{Acknowledgements}
\acks{The authors gratefully acknowledge support from NSF DMS-0847388, NSF DMS-1045153, ATD-1222567, CCF-0808847, AFOSR FA9550-14-1-0033, DARPA N66001-11-1-4002. 
We would also like to thank Mark Iwen for his insightful comments.}

%#######################

\begin{comment}
	
\section{Extensions}
It is well known that many data sets that arise in applications are not well-modeled by smooth manifolds \cite{}
Observe that the key geometric properties used in the proof are the following:

Discuss OUTLIERS?

Finite reach implies $\m M$ is at least $C^{1,1}$ (differentiable with first derivative Lipschitz) \footnote{Thanks to W.K. Allard for pointing this out, and for K. Vixie for a short proof piecing together results from \cite{Federer:GeometricMeasureTheory}.}

			\end{comment}
%#######################################

%\newpage
%
%\begin{thebibliography}{9}
%\bibitem{ZB}Zwald, L., Blanchard, G. ``On the Convergence of Eigenspaces in Kernel Principal Component Analysis". Preprint, 2006.
%\bibitem{Niyogi:homology}Niyogi, P, Smale, S, Weinberger, S. ``Finding the Homology of Submanifolds with High Confidence from Random
%Samples". Technical Report No. TR-2004-08, Department of Computer Science, University of Chicago (2006).
%\bibitem{CM:MGM2} Allard, W., Guangliang Chen, and Mauro Maggioni. ``Multiscale geometric methods for data sets II: Geometric wavelets." Appl. Comp. Harm. Anal., accepted (2011).
%\bibitem{Wasserman:ManifoldEstimationHausdorff}Genovese, Christopher, Marco Perone-Pacifico, Isabella Verdinelli, and Larry Wasserman. ``Minimax manifold estimation." Journal of Machine Learning Research 13 (2012): 1263-1291.
%\end{thebibliography}

\appendix
\addcontentsline{toc}{section}{Appendix: Proofs of geometric propositions and lemmas}
\section*{Appendix: Proofs of geometric propositions and lemmas}
\label{pf:volume}
\begin{proof}[Proof of Proposition \ref{prop:volume}]
For the first inequality, let
$$
A=\begin{pmatrix}
I\\
X
\end{pmatrix}\text{ and } B=\begin{pmatrix}
Y\\
0
\end{pmatrix},
$$
and for every $T\subset[d]$, we let $V_T$ denote the volume of $\{a_i\}_{i\in T^c}\cup\{b_i\}_{i\in T}$, where $a_i$ and $b_i$ denote the $i$th columns of $A$ and $B$ respectively.
By submultilinearity of the volume we have
$$
\Vol(A+B)\leq \sum_{T\in 2^{[d]}} V_T,
$$
where $2^{[d]}=\{S: S\subset\{1,\ldots,d\}\}$. We now show that $V_T\leq q^{\vert T\vert}\Vol(A)$ for every $T\in 2^{[d]}$. The bound $\Vert Y\Vert\leq q$ implies $\Vert y_i\Vert\leq q$ for all $i=1,\ldots,d$, and so the fact that the volume is a submultiplicative function implies that
$$
V_T\leq q^{\vert T\vert} \Vol (A_{T^c}).
$$
On the other hand, letting $a_1^\perp$ be the orthogonal projection of $a_1$ onto $\text{span}^\perp\{a_i\}_{i=2}^d$, we note that $\Vert a_1^\perp\Vert\geq 1$, and thus
$$
\Vol(A_{\{1\}^c})\leq \Vert a_1^\perp\Vert\Vol(A_{\{1\}^c})=\Vol(A).
$$
By induction and invariance of the volume under permutations, we see that $\Vol (A_{T^c})\leq \Vol(A)$ for all $T\in 2^{[d]} $. Thus,
$$
\Vol(A+B)\leq \sum_{T\in 2^{[d]}} q^{\vert T\vert}\Vol(A)=(1+q)^d\Vol(A).
$$

For the second inequality, since $Y$ is symmetric, we can represent it as 
$Y=F-G$ where $F$ and $G$ are symmetric positive semidefinite, $FG=GF=0$, and $\Vert F\Vert,\Vert G\Vert\leq \Vert Y\Vert$. 
Indeed, if $Y=Q\Lambda Q^T$ is the eigenvalue decomposition of $Y$ with $\Lambda=\diag(\lambda)$, set 
$\lambda_+:=(\max(0,\lambda_1),\ldots,\max(0,\lambda_d))^T$, $\lambda_-:=\lambda_+ - \lambda$, and define 
$F:=Q\diag(\lambda_+)Q^T$, $G=Q\diag(\lambda_-)Q^T$.  

Recall the \emph{matrix determinant lemma}: let $T\in \mb R^{k\times k}$ be invertible, and let $U,V\in \mb R^{k\times l}$. 
Then
\[
\Vol(T+UV^T)=\Vol(I+V^T T^{-1}U)\Vol(T).
\]
Applying it in our case with $U=\begin{pmatrix} \sqrt{F}-\sqrt{G}\\
												0
												\end{pmatrix}$, 
$V=\begin{pmatrix} \sqrt{F}+\sqrt{G}\\
							0
							\end{pmatrix}$, and 
$T=\begin{pmatrix} I & X^T \\
				X & -I
							\end{pmatrix}$, 
we have that
$$
\Vol \begin{pmatrix}
			I+Y & X^T\\
			X & -I\\
			\end{pmatrix} = \Vol \left( I + \begin{pmatrix}
												\sqrt{F}+\sqrt{G}\\
												0
												\end{pmatrix}^T \begin{pmatrix}
			I & X^T\\
			X & -I\\
			\end{pmatrix}^{-1}\begin{pmatrix}
												\sqrt{F}-\sqrt{G}\\
												0
												\end{pmatrix}\right) \Vol \begin{pmatrix}
			I & X^T\\
			X & -I\\
			\end{pmatrix}.	
$$
By orthogonality of the columns in
$
\begin{pmatrix}
		I\\
		X
		\end{pmatrix}
$
with the columns in 
$
\begin{pmatrix}
	X^T\\
	-I
	\end{pmatrix},
$
we have that
$$
\left\Vert \begin{pmatrix}
			I & X^T\\
			X & -I\\
			\end{pmatrix} \begin{pmatrix}
								u\\
								v
								\end{pmatrix}\right\Vert \geq \left\Vert  \begin{pmatrix}
								u\\
								v
								\end{pmatrix}\right\Vert,
$$
and hence
$$
\left\Vert \begin{pmatrix}
												\sqrt{F}+\sqrt{G}\\
												0
												\end{pmatrix}^T\begin{pmatrix}
			I & X^T\\
			X & -I\\
			\end{pmatrix}^{-1}\begin{pmatrix}
												\sqrt{F}-\sqrt{G}\\
												0
												\end{pmatrix}\right\Vert\leq \sqrt{q}\cdot 1\cdot\sqrt{q}=q.
$$
Therefore, we conclude that
$$
\Vol \left( I + \begin{pmatrix}
												\sqrt{F}+\sqrt{G}\\
												0
												\end{pmatrix}^T \begin{pmatrix}
			I & X^T\\
			X & -I\\
			\end{pmatrix}^{-1}\begin{pmatrix}
												\sqrt{F}-\sqrt{G}\\
												0
												\end{pmatrix}\right)\geq (1-q)^d,
$$
and combining this with the expression from the matrix determinant lemma completes the proof.									
\end{proof}

\begin{proof}[Proof of Lemma \ref{lem:lengths}]
Let $\gamma:[0,d_{\m M}(z,y)]\rightarrow\m M$ denote the arclength-parameterized geodesic connecting $y$ to $z$ in $\m M$. 
Since $\gamma$ is a geodesic, there is a $v\in T_y \m M$ with $\Vert v\Vert=1$ such that the Taylor expansion
$$
z = y + d_{\m M}(z,y) v + \int_0^{d_{\m M}(z,y)} \gamma^{\prime\prime}(t)\l(d_{\m M}(z,y)-t\r)dt.
$$
By Proposition \ref{prop:reach}, $\Vert \gamma^{\prime\prime}(t)\Vert_2\leq 1/\tau$ for all $t$ and $ d_{\m M}(z,y)\leq 2 r$, so we have that
\begin{align*}
\vert \langle \eta, z-y\rangle\vert&= \l\vert\l\langle \eta, \int_0^{d_{\m M}(z,y)} \gamma^{\prime\prime}(t)\l(d_{\m M}(z,y)-t\r)dt\r\rangle\r\vert\\
&\leq \int_0^{d_{\m M}(z,y)} \vert\langle \eta,\gamma^{\prime\prime}(t)\rangle\vert \l(d_{\m M}(z,y)-t\r)dt\\
&\leq \frac{1}{\tau}\int_0^{d_{\m M}(z,y)} \l(d_{\m M}(z,y)-t\r)dt\\
&\leq \frac{d_{\m M}(z,y)^2}{2\tau}\\
&\leq \frac{2r^2}{\tau}.
\end{align*}
\end{proof}

\begin{proof}[Proof of Lemma \ref{lem:inject}]
Suppose $a$ and $b$ are distinct in $B(y,r)\cap {\m M}$. 
Now, $b-a=v+w$ where $v\in T_a\m M$ and $w\in T_a^\perp\m M$, and note that $\Vert w\Vert\leq \frac{2\Vert b-a\Vert^2}{\tau}\leq 4\frac{r}{\tau}$ by Lemma \ref{lem:lengths}. This also implies that
\begin{align*}
\Vert v\Vert &= \sqrt{\Vert a-b\Vert^2-\Vert w\Vert^2}
			\geq \sqrt{\Vert a-b\Vert^2 - 4\frac{\Vert a-b\Vert^4}{\tau^2}}
			\geq \Vert a-b\Vert \sqrt{1- 16 \frac{r^2}{\tau^2}}
			\geq \Vert a-b\Vert\sqrt{1-4\frac{r}{\tau}}\,.
\end{align*}
By part {\it iii.} of Proposition \ref{prop:reach}, there is a $u\in T_y\m M$ such that $\langle u, v\rangle\geq\Vert v\Vert\cos(\phi)$ where $\phi$ is the angle between $T_y\m M$ and $T_a\m M$. Then
\begin{align*}
\vert\langle u, b-a\rangle\vert&\geq \vert\langle u,v\rangle\vert-\vert\langle u, w\rangle\vert\\
&\geq \Vert v\Vert\cos(\phi) - \Vert w\Vert\\
&\geq \Vert a-b\Vert \sqrt{1-4\frac{r}{\tau}}\sqrt{1-2\frac{r}{\tau}} - 2\frac{\Vert a-b\Vert^2}{\tau}\\
&\geq \Vert b-a\Vert\l(\sqrt{1-4\frac{r}{\tau}}\sqrt{1-4\frac{r}{\tau}}-4\frac{r}{\tau}\r)\\
&\geq \Vert b-a\Vert\left(1-8\frac{r}{\tau}\right).
\end{align*}
It then follows from $r<\tau/8$ that $\proj{}_{T_y\m M}(b-a)\not = 0$, and hence $\proj{}_{y+T_y\m M}(a)\not = \proj{}_{y+T_y\m M}(b)$ and injectivity holds.
\end{proof}

\begin{proof}[Proof of Proposition \ref{prop:diffBounds}]
For $\varepsilon<\tau/8$, we may define the embedding
$$
\begin{pmatrix}
		v\\
		\beta
		\end{pmatrix}\longmapsto \begin{pmatrix}
											v\\
											f(v)
											\end{pmatrix} + \begin{pmatrix}
																Df(v)^T\\
																-I\\
																\end{pmatrix}\beta
$$
where we have assumed (without loss of generality) that $y=0$ and $T_y\m M$ coincides with the span of the first $d$ canonical orthonormal basis members. The domain of this map is the set 
$$
\Omega=\{(v,\beta)\in \mathbb{R}^d\times\mathbb{R}^{D-d}:v\in T_y\m M\cap B(0,\varepsilon),\Vert \beta\Vert^2+\Vert Df(v)^T\beta\Vert^2<\tau^2\}
$$
and the Jacobian of this map is
$$
\begin{pmatrix}
I +\sum_{i=1}^{D-d}\beta_i D^2f_i(v) & Df(v )^T\\
Df(v) & -I
\end{pmatrix}.
$$
It is clear that the inverse of the above map is given by
$$
x\longmapsto (\proj{}_{y+T_y\m M} (\proj{}_{\m M}(x)), \proj{}_{T_y^\perp\m M}(x-\proj{}_{\m M}(x))),
$$
which is at least a $C^1$ map. Thus, a necessary condition for the $\tau$-radius normal bundle to embed is that the Jacobian exhibited above is invertible, which in turn implies that
$$
\begin{pmatrix}
I +\sum_{i=1}^{D-d}\beta_i D^2f_i(v) & Df(v)^T\\
Df(v) & -I
\end{pmatrix}\begin{pmatrix}
		\zeta\\
		Df(v)\zeta
		\end{pmatrix}\not = 0
$$
for all $\zeta\not=0$ when $(v,\beta)\in\Omega$. This reduces to $(I+\sum\beta_i D^2f_i(v)+Df(v)^TDf(v))\zeta\not = 0$, and so a necessary condition for embedding is then that the norm of $\sum_{i=1}^{D-d}\beta_i D^2f_i(v)$ does not exceed $1+\Vert Df(v)\Vert^2$ whenever
$$
\l\Vert\begin{pmatrix}
			Df(v)^T\\
			-I
			\end{pmatrix}\beta\r\Vert^2=\Vert \beta\Vert^2+\Vert Df(v)^T\beta\Vert^2 < \tau^2.
$$
In particular, this must be true if $\Vert\beta\Vert < \tau/\sqrt{1+\Vert Df(v)\Vert^2}$. This reduces to the condition that the operator norm
\begin{align}
\label{eq:v10}
\sup_{u\in \m S^{D-d-1}}\left\Vert \sum_{i=1}^{D-d}u_i D^2 f_i(v)\right\Vert < \frac{(1+\Vert Df(v)\Vert^2)^{3/2}}{\tau} < \frac{1}{\tau}\left( 1+\Vert Df(v)\Vert\right)^3.
\end{align}
By the fundamental theorem of calculus, we have that
$$
Df(v)x = Df(0)x +\int_0^{\Vert v\Vert} [ u_v^T D^2f_i(t u_v)x]dt = \int_0^{\Vert v\Vert} [u_v^T D^2f_i(t u_v) x] dt,
$$ 
where $u_v=v/\Vert v\Vert$ and $[ u_v^T D^2f_i(t u_v)x]$ indicates a vector with $i$th component $u_v^T D^2f_i(t u_v)x$. 
Consequently, for any $x\in\mathbb{R}^{d}$, we have that
\begin{equation}
\begin{aligned}
\Vert D f(v)x\Vert &\leq \int_0^{\Vert v\Vert} \left\Vert [u_v^T D^2f_i(t u_v) x]\right\Vert dt
%&\nonumber
\leq \Vert v\Vert \sup_{t\in [0,\Vert v\Vert]}\left\Vert[u_v^T D^2f_i(t u_v) x]\right\Vert\\
&
\leq \varepsilon\sup_{t\in [0,\varepsilon]}\left\Vert [u_v^T D^2f_i(t u_v) x]\right\Vert.
\end{aligned}
\label{eq:v20}
\end{equation}
Now, 
\begin{align*}
\left\Vert [u_v^T D^2f_i(t u_v) x]\right\Vert &= \sup_{u\in \m S^{D-d-1}} \langle u, [u_v^T D^2f_i(t u_v) x]\rangle= \sup_{u\in \m S^{D-d-1}} \sum_{i=1}^{D-d} u_i(u_v^T D^2f_i(t u_v) x)\\
										   &=  \sup_{u\in \m S^{D-d-1}} u_v^T \l(\sum_{i=1}^{D-d} u_iD^2f_i(t u_v)\r) x\\
										   &\leq  \sup_{u\in \m S^{D-d-1}} \Vert u_v\Vert  \l\Vert\sum_{i=1}^{D-d} u_iD^2f_i(t u_v)\r\Vert \Vert x\Vert\\
										   &= \Vert x\Vert \sup_{u\in \m S^{D-d-1}} \l\Vert\sum_{i=1}^{D-d} u_iD^2f_i(t u_v)\r\Vert,
\end{align*}
which together with (\ref{eq:v20}) and (\ref{eq:v10}) yields the bound
$$
\Vert Df(v)\Vert < \frac{\eps}{\tau}\l(1+\sup_{t\in[0,\eps]}\Vert Df(t u_v)\Vert\r)^3.
$$
Since this inequality also holds for any $v^\prime$ with $\Vert v\Vert\leq \eps^\prime$, taking a supremum yields
\begin{align*}
\sup_{\eps^\prime \in[0,\eps]}\Vert Df(t u_v)\Vert &\leq \sup_{\eps^\prime\in[0,\eps]}\frac{\eps^\prime}{\tau}\l(1+\sup_{t\in[0,\eps^\prime ]}\Vert Df(t u_v)\Vert\r)^3
\leq \frac{\eps }{\tau}\l(1+\sup_{\eps^\prime \in[0,\eps ]}\Vert Df(t u_v)\Vert\r)^3,
\end{align*}
and hence
$$
\sup_{v\in B_d(0,\eps)}\Vert Df(v)\Vert \leq \frac{\eps }{\tau}\l(1+\sup_{v\in B_d(0,\eps)}\Vert Df(v)\Vert\r)^3.
$$
Setting $a(\eps^\prime)=\sup_{v\in B_d(0,\eps^\prime )}\Vert Df(v)\Vert$, we have that $a(0)=0$,
$$
a(\eps^\prime)\leq \frac{\eps^\prime}{\tau}\l(1+a(\eps^\prime)\r)^3,
$$
for all $\eps^\prime \geq 0$, and $a$ is continuous by continuity of $\Vert D f(v)\Vert$. 
Setting $b(\eps^\prime)= a(\eps^\prime)/(1+a(\eps^\prime))$, we get
$$
b(\eps^\prime)(1-b(\eps^\prime))^2\leq \frac{\eps^\prime}{\tau}.
$$
Examining the polynomial $x(1-x)^2$, we see that the sublevel set $x(1-x)^2\leq \omega$ consists of two components when $\omega<4/27$. Also note that if $\omega<1/8$, then
$$
2(1-2\omega)^2= 2-8\omega+8\omega^2 > 2-1 = 1,
$$
and hence
$$
2\omega(1-2\omega)^2> \omega.
$$
Consequently, if $x$ is such that $x(1-x)^2\leq \omega$ and is in the interval containing zero in the sublevel set $x(1-x)^2\leq \omega<1/8$, then $x\leq 2\omega$.

By these observations, continuity of $b(\eps^\prime)$, and the fact that $b(0)=0$, 
we have that $a(\eps^\prime)\leq \frac{2\frac{\eps^\prime}{\tau}}{1-2\frac{\eps^\prime}{\tau}}$, and thus
\begin{align*}
\sup_{v\in B_d(0,\eps)}\Vert Df(v)\Vert \leq \frac{2\eps}{\tau-2\eps}.
\end{align*}
From the bound in (\ref{eq:v10}) we now acquire the bound
\begin{align*}
\sup_{v\in B_d(0,\eps)}\sup_{u\in\m S^{D-d-1}} \l\Vert \sum_{i=1}^{D-d-1} u_i D^2f_i(v)\r\Vert \leq \frac{\tau^2}{(\tau-2\eps)^3}.
\end{align*}
\end{proof}

\begin{proof}[Proof of Lemma \ref{lem:volUB}]
%[Proof of lemma \ref{lem:volUB}]
We first prove part {\it i.} Let $\eps>0$ satisfy $\eps<\tau/8$. Because of (\ref{eq:hessSUP}) and the fact that $\Vert\beta\Vert\leq\sigma$, we have that 
$$
\left\Vert \sum_{i=1}^{D-d}\beta_i D^2 f_i(v)\right\Vert\leq \frac{\sigma\tau^2}{(\tau - 2\varepsilon)^3}.
$$
Since this is also a bound for the columns of $\sum \beta_i D^2f_i(v)$, Proposition \ref{prop:volume} implies that
\begin{align*}
\text{Vol}\begin{pmatrix}
			I+\sum\beta_i D^2f_i(v) & Df(v)^T\\
			Df(v) & - I
			\end{pmatrix} &\leq \left(1+\frac{\sigma\tau^2}{(\tau-2\varepsilon)^3}\right)^d\Vol\begin{pmatrix}
			I & Df(v)^T\\
			Df(v) &-I
			\end{pmatrix}
\end{align*}
in $T^\perp(\m M\cap B(y,\varepsilon))\cap \m M_\sigma$.

On the other hand, we have that

$$
\Vol\begin{pmatrix}
			Df^T(v)\\
			-I
			\end{pmatrix} \leq \prod_{i=1}^{D-d}\sqrt{1+\Vert\nabla f_i(v)\Vert^2}\leq \left(1+\frac{4\varepsilon^2}{(\tau-2\varepsilon)^2}\right)^{(D-d)/2}
$$
since (\ref{eq:jacoSUP}) implies the bounds $\Vert\frac{\partial f(v)}{\partial v_i}\Vert\leq \frac{2\eps}{\tau-2\eps}$ for each $i=1,\ldots, d$, and the above is the largest this quantity may be subject to these bounds.

When these estimates are joined together, we have an inequality
\begin{align*}
\Vol\begin{pmatrix}
I +\sum_{i=1}^{D-d}\beta_i D^2f_i(v) & Df(v)^T\\
Df(v) & -I
\end{pmatrix}&\leq \left(1+\frac{\sigma\tau^2}{(\tau-2\varepsilon)^3}\right)^d\text{Vol}\begin{pmatrix}
			I & Df(v)^T\\
			Df(v) &-I
			\end{pmatrix}\\
			&\leq \left(1+\frac{\sigma\tau^2}{(\tau-2\varepsilon)^3}\right)^d\left(1+\frac{4\varepsilon^2}{(\tau-2\varepsilon)^2}\right)^{(D-d)/2}\Vol\begin{pmatrix}
			I\\
			Df(v)
			\end{pmatrix}.
\end{align*}
For an arbitrarily small $\varepsilon>0$, let $\{U_\gamma\}_{\gamma\in\Gamma}$ denote a finite partition of $U$ into measurable sets such that there for each $\gamma\in\Gamma$, there is a $y_\gamma$ satisfying $U_\gamma\subset \m M\cap B(y_\gamma,\varepsilon)$. Let $f_\gamma$ denote the inverse of $P_\gamma=\proj{}_{y_\gamma+T_{y_\gamma}\m M}$ in $U_\gamma$, and set
$$
E_{\gamma,v}=\{\beta\in \mathbb{R}^{D-d}: \Vert\beta\Vert^2+\Vert Df_\gamma(v)\beta\Vert^2\leq\sigma^2\}
$$
for all $v\in P_\gamma(U_\gamma)$. Thus, 
\begin{align*}
\int_{P_\gamma^{-1}(U_\gamma)} d\Vol(x) &= \int_{P_\gamma(U_\gamma)}\int_{E_{\gamma,v}}\Vol\begin{pmatrix}
I +\sum_{i=1}^{D-d}\beta_i D^2f_i(v) & Df(x)^T\\
Df(v) & -I
\end{pmatrix}d\beta dv\\
&\leq \int_{P_\gamma(U_\gamma)}\int_{E_{\gamma,v}}\left(1+\frac{\sigma\tau^2}{(\tau-2\varepsilon)^3}\right)^d\left(1+\frac{4\varepsilon^2}{(\tau-2\varepsilon)^2}\right)^{(D-d)/2}\Vol\begin{pmatrix}
I \\
Df(v)
\end{pmatrix}d\beta dv\\
&\leq \left(1+\frac{\sigma\tau^2}{(\tau-2\varepsilon)^2}\right)^d\left(1+\frac{4\varepsilon^2}{(\tau-2\varepsilon)^2}\right)^{(D-d)/2}\vol(U_\gamma)\Vol(B_{D-d}(0,\sigma))
\end{align*}
since $E_{\gamma,v}\subset B_{D-d}(0,\sigma)$. 
Consequently, we have that
\begin{align*}
\Vol( P^{-1}(U))&=\sum_{\gamma\in\Gamma}\Vol( P_\gamma^{-1}(U_\gamma))\\
&\leq\sum_{\gamma\in\Gamma}\left(1+\frac{\sigma\tau^2}{(\tau-2\varepsilon)^3}\right)^d\left(1+\frac{4\varepsilon^2}{(\tau-2\varepsilon)^2}\right)^{(D-d)/2}\vol(U_\gamma)\Vol(B_{D-d}(0,\sigma))\\
&=\left(1+\frac{\sigma\tau^2}{(\tau-2\varepsilon)^3}\right)^d\left(1+\frac{4\varepsilon^2}{(\tau-2\varepsilon)^2}\right)^{(D-d)/2}\vol(U)\Vol(B_{D-d}(0,\sigma)).
\end{align*}
Since $\varepsilon>0$ was arbitrary, we obtain
$$
\Vol(P^{-1}(U))\cap \m M_\sigma)\leq\left(1+\frac{\sigma}{\tau}\right)^d\vol(U)\Vol(B_{D-d}(0,\sigma)).
$$
This completes the proof of upper bound in part {\it i.} Using a similar partition strategy, we have that
\begin{align*}
\int_{P_\gamma^{-1}(U_\gamma)} d\Vol(x) &= \int_{P_\gamma(U_\gamma)}\int_{E_{\gamma,v}}\Vol\begin{pmatrix}
I +\sum_{i=1}^{D-d}\beta_i D^2f_i(v) & Df(x)^T\\
Df(v) & -I
\end{pmatrix}d\beta dv\\
&\geq \int_{P_\gamma(U_\gamma)}\int_{E_{\gamma,v}}\left(1-\frac{\sigma\tau^2}{(\tau-2\varepsilon)^3}\right)^d\Vol\begin{pmatrix}
I &Df(v)^T\\
Df(v) &- I
\end{pmatrix}d\beta dv\\
&= \int_{P_\gamma(U_\gamma)}\int_{E_{\gamma,v}}\left(1-\frac{\sigma\tau^2}{(\tau-2\varepsilon)^3}\right)^d\Vol\begin{pmatrix}
I \\
Df(v)
\end{pmatrix}\Vol\begin{pmatrix}
Df(v)^T\\
- I
\end{pmatrix}d\beta dv\\
&\geq \int_{P_\gamma(U_\gamma)}\int_{E_{\gamma,v}}\left(1-\frac{\sigma\tau^2}{(\tau-2\varepsilon)^3}\right)^d\Vol\begin{pmatrix}
I \\
Df(v)
\end{pmatrix}d\beta dv\\
&\geq \int_{P_\gamma(U_\gamma)}\int_{B_{D-d}\l(0,\frac{\sigma}{1+\frac{\eps}{\tau-\eps}}\r)}\left(1-\frac{\sigma\tau^2}{(\tau-2\varepsilon)^3}\right)^d\Vol\begin{pmatrix}
I \\
Df(v)
\end{pmatrix}d\beta dv\\
&= \left(1-\frac{\sigma\tau^2}{(\tau-2\varepsilon)^3}\right)^d\vol(U_\gamma)\Vol\left(B_{D-d}\left(0,\left(1-\frac{\eps}{\tau}\right)\sigma\right)\right)
\end{align*}
In the inequalities above, we have used the fact that there is a ball of radius $\left(1-\frac{\eps}{\tau}\right)\sigma$ inside of $E_{\gamma, v}$ for each $\gamma$ and each $v$. 
Aggregating all of the sums and letting $\eps\to 0$ yields the lower bound in part {\it i.}

We now prove part {\it ii.} Note that
$$
\Vol(\m M_\sigma\cap B(y,r))\leq \Vol( P^{-1}(\m M\cap B(y,r+\sigma)))
$$
since $\Vert \proj{}_{\m M}(x)-y\Vert\leq \Vert x-y\Vert + \Vert \proj{}_{\m M}(x)-x\Vert\leq r+\sigma$. Part {\it ii.} now follows from part {\it i.} and the fact that
\begin{align*}
\vol(\m M\cap B(y,r+\sigma))&\leq \int\limits_{P(\m M\cap B(y,r+\sigma))} \Vol \begin{pmatrix}
																					I\\
																					Df(v)
																			\end{pmatrix}dv\\
							&\leq \left(1+\left(\frac{2(r+\sigma)}{\tau-2(r+\sigma)}\right)^2\right)^{d/2}\Vol(B_d(0,r+\sigma)).																		
\end{align*}

\end{proof}

\begin{proof}[Proof of Lemma \ref{lem:evalUB}]
By the variational characterization of eigenvalues, we have that
\begin{align*}
\sum_{i=d+1}^D\lambda_i(\Sigma)&=\argmin_{\dim(V)=D-d} \tr(\proj{}_V^T\Sigma \proj{}_V)\\
&=\argmin_{\dim(V)=D-d} \mathbb{E}\Vert \proj{}_V(Z-\mathbb{E}Z)\Vert^2\\
&=\argmin_{\dim(V)=d}\mathbb{E}\Vert Z-\mathbb{E}Z-\proj{}_V(Z-\mathbb{E}Z)\Vert^2.
\end{align*}
Thus, we have that 
$\sum\limits_{i=d+1}^D\lambda_i(\Sigma)\leq \mathbb{E}\Vert Z-\mathbb{E}Z-\proj{}_{T_y\m M}(Z-\mathbb{E}Z)\Vert^2$. 
Observe that
\begin{align*}
%\label{eq:f10}
\nonumber
\mathbb{E}\Vert Z-\mathbb{E}Z-\proj{}_{T_y\m M}(Z-\mathbb{E}Z)\Vert^2=&\mathbb{E}\Vert Z-y+(y-\mathbb{E} Z) - \proj{}_{T_y\m M}((Z-y)+(y-\mathbb{E}Z))\Vert^2\\
\nonumber
=&\mathbb{E}\Vert Z-y-\proj{}_{T_y\m M}(Z-y)\Vert^2\\
& - \Vert (y-\mathbb{E} Z) - \proj{}_{T_y\m M}(y-\mathbb{E}Z)\Vert^2 \\
\nonumber
&\leq\mathbb{E}\Vert Z-y-\proj{}_{T_y\m M}(Z-y)\Vert^2.
\end{align*}
Now for any $z\in \m M_\sigma\cap B(y,r)$, we have that $z=\beta+x$ where $x\in \m M$, and $\beta\in T_x^\perp\m M$ satisfies $\Vert\beta\Vert\leq\sigma$. 
Moreover, there is a unique decomposition $x=\eta+v+y$ where $\eta\in T_ y^\perp\m M$ and $v\in T_y\m M$.  Thus,
\begin{align}
\label{eq:f01}
\Vert z-y-\proj{}_{T_y\m M}(z-y)\Vert=\Vert \beta+\eta -\proj{}_{T_y\m M}\beta\Vert\leq \Vert\beta-\proj{}_{T_y\m M}(\beta)\Vert+\Vert\eta\Vert\leq \sigma+\frac{2r^2}{\tau}, 
\end{align}
by Lemma \ref{lem:lengths}, and we obtain the bound
\begin{align}
\label{eq:f02}
\mathbb{E}\Vert Z-\mathbb{E}Z- \proj{}_{T_y\m M}(Z-\mathbb{E}Z)\Vert^2\leq 2\sigma^2+\frac{8r^4}{\tau^2}.
\end{align}
This establishes the required estimate.
\end{proof}

\begin{proof}[Proof of Lemma \ref{lem:evalLB}]
For any unit vector $u\in T_y\m M$ we have
\begin{align*}
\mathbb{E}\langle u,Z-\mathbb{E}Z\rangle^2&= \frac{1}{\Vol(Q\cap \m M_\sigma)}\int_{Q\cap \m M_\sigma}\langle u, Z-\mathbb{E}Z\rangle^2 d\Vol(Z)\\
&\geq \frac{1}{\Vol(B(y,r_2)\cap \m M_\sigma)}\int_{B(y,r_1)\cap \m M_\sigma}\langle u, (Z-y)-\mathbb{E}(Z-y)\rangle^2 d\Vol(Z)
\end{align*}
using the inclusion assumptions, and by adding and subtracting the constant vector $y$. 

We now seek to reduce the domain of integration and perform a change of variables. Since $r_1\leq\tau/8$, the inverse of the affine projection onto $y+T_y\m M$ is injective. Without loss of generality, we assume $y=0$ and $T_y\m M$ is the span of the first $d$ standard orthonormal vectors. Letting $f$ denote the inverse of the affine projection onto $y+T_y\m M$, we see that the map 
\[
\begin{pmatrix}
	v\\
	\beta
	\end{pmatrix}\longmapsto \begin{pmatrix}
									v\\
									f(v)+\beta
									\end{pmatrix}
\]
is well-defined and injective on $\proj{}_{T_y\m M}(\m M\cap B(y,r_1-\sigma))\times (T_y^\perp \m M\cap B(0,\sigma))$. Let $g$ denote this map, note that
\[
\Vert x +\beta\Vert\leq \Vert x\Vert+\Vert\beta\Vert\leq(r-\sigma)+\sigma=r,
\]
for $x\in \m M \cap B(y,r_1-\sigma)$, and hence the image of $g$ is contained in $\m M_\sigma\cap B(y,r_1)$. Since the absolute value of the determinant of the Jacobian of $g$ is always $1$ (it is lower triangular with ones on the diagonal), employing the change of coordinates in the reduced domain of integration yields
\[
\mathbb{E}\langle u,Z-\mathbb{E}Z\rangle^2\geq \frac{1}{\Vol(B(y,r_2)\cap\m M_\sigma)}\int_{\m A} \int_{\m B} \l\langle \begin{pmatrix}
																																u\\
																																0
																																\end{pmatrix}, \begin{pmatrix}
																																					v\\
																																					f(v)+\beta
																																					\end{pmatrix}-\mathbb{E}(Z-y)\r\rangle^2d\beta dv,
\]
where 
\[
\m A=\proj{}_{T_y\m M}(B(y,r_1-\sigma)\cap\m M),\: \m B=T_y^\perp \m M\cap B(0,\sigma) .
\]
Note that $B(y,\cos(\theta)(r_1-\sigma))\cap(y+ T_y\m M)\subset \m A$. Setting ${\m Q}=\proj{}_{T_y\m M}$, this immediately reduces to
\begin{align*}
\mathbb{E}\langle u,Z-\mathbb{E}Z\rangle^2&\geq\frac{1}{\Vol(B(y,r_2)\cap\m M_\sigma)}\int_{\m A} \int_{\m B} \langle u, v-\mathbb{E}{\m Q}(Z-y)\rangle^2d\beta dv\\
&=\frac{\Vol(B_{D-d}(0,\sigma))}{\Vol(B(y,r_2)\cap \m M_\sigma)}\int_{\m A} \langle u, v-\mathbb{E}{\m Q}(Z-y)\rangle^2 dv\\
&\geq\frac{\Vol(B_{D-d}(0,\sigma))}{\Vol(B(y,r_2)\cap \m M_\sigma)}\int_{B_d(0,q)} \langle u, v-\mathbb{E}{\m Q}(Z-y)\rangle^2 dv,
\end{align*}
where $q=\cos(\delta)(r_1-\sigma)$ and $\delta=\arcsin((r_1-\sigma)/2\tau)$. Noting that $\int_{B_d(0,q)}\langle u,v\rangle dv=0$ by symmetry, we now use linearity of the inner product to further reduce the integrand:
\begin{align*}
\mathbb{E}\langle u,Z-\mathbb{E}Z\rangle^2&\geq\frac{\Vol(B_{D-d}(0,\sigma))}{\Vol(B(y,r_2)\cap \m M_\sigma)}\int_{B_d(0,q)} \left(\langle u, v\rangle^2-2\langle u,v\rangle\langle u,\mathbb{E}{\m Q}(Z-y)\rangle+\langle u,\mathbb{E}{\m Q}(Z-y)\rangle^2\right) dv\\
&=\frac{\Vol(B_{D-d}(0,\sigma))}{\Vol(B(y,r_2)\cap \m M_\sigma)}\int_{B_d(0,q)} \left(\langle u, v\rangle^2+\langle u,\mathbb{E}{\m Q}(Z-y)\rangle^2\right) dv\\
&\geq \frac{\Vol(B_{D-d}(0,\sigma))}{\Vol(B(y,r_2)\cap \m M_\sigma)}\int_{B_d(0,q)} \langle u, v\rangle^2 dv\\
&=\frac{\Vol(B_{D-d}(0,\sigma))\Vol(B_d(0,q))}{\Vol(B(y,r_2)\cap \m M_\sigma)}\frac{q^2}{d}.
\end{align*}
By Lemma \ref{lem:volUB}, we then obtain
\begin{align}
\label{eq:eig_lb}
\nonumber
\mathbb{E}\langle u,Z-\mathbb{E}Z\rangle^2&\geq\left(\left(1+\frac{\sigma}{\tau}\right)\sqrt{1+\left(\frac{2(r_2+\sigma)}{\tau-2(r_2+\sigma)}\right)^2}\right)^{-d}\frac{\Vol(B_d(0,q))}{\Vol(B_d(0,r_2+\sigma))}\frac{q^2}{d}\\
&
\geq \frac{1}{4\left(1+\frac{\sigma}{\tau}\right)^d}\left(\frac{r_1-\sigma}{r_2+\sigma}\right)^d
\left(\frac{1-\left(\frac{r_1-\sigma}{2\tau}\right)^2}{1+\left(\frac{2(r_2+\sigma)}{\tau-2(r_2+\sigma)}\right)^2}\right)^{d/2}\frac{(r_1-\sigma)^2}{d}.
\end{align}
Let $V_{d-1}(\Sigma)$ be a subspace corresponding to the first $d-1$ principal components of $Z$: 
\[
V_{d-1}=\argmin_{\dim(V)=d-1}\mb E\|Z-\mb EZ-\proj{}_{V}(Z-\mb EZ)\|,
\]
and note that $\lambda_d(\Sigma)=\max_{0\ne u\in V_{d-1}^\perp}\mb E\dotp{\frac{u}{\|u\|}}{Z-\mb EZ}^2$. 
Since $\dim(V_{d-1}^\perp)=D-d+1$ and $\dim(T_y \m M)=d$, it is easy to see that $V_{d-1}^\perp\cap T_y\m M\ne \emptyset.$ For any $u_\ast\in V_{d-1}^\perp\cap T_y\m M$ such that  $\|u_\ast\|=1$ it follows from Courant-Fischer characterization of $\lambda_d(\Sigma)$ that
\[
\lambda_d(\Sigma)\geq \mb E\dotp{u_\ast}{Z-\mb EZ}^2,
\]  
and (\ref{eq:eig_lb}) implies the desired bound. 
\end{proof}

\begin{proof}[Proof of Lemma \ref{lem:sup-norm}]
Let $Q\subset \mb R^D$ be such that $B(y,r_1)\subset Q$ and $\m M_\sigma\cap Q\subset B(y,r_2)$ for some $y\in \m M$ and $\sigma<r_1<r_2<\tau/8-\sigma$. 
Assume that $Z$ is drawn from $U_{\m M_\sigma\cap Q}$, let $\Sigma$ be the covariance matrix of $Z$ and $V_d:=V_d(\Sigma)$ - the subspace corresponding to the first $d$ principal components of $Z$. 

Let $\alpha\in [0,1]$ be such that $\cos(\phi):=\min_{u\in V_d, \|u\|=1}\max_{v\in T_y \m M,\|v\|=1}\left| \dotp{u}{v}\right|=\sqrt{1-\alpha^2}$ is the cosine of the angle between $T_y \m M$ and $V_d$. 
Then there exists a unit vector $u_\ast\in (V_d)^\perp$ such that 
\[
\max_{v\in T_y \m M, \|v\|=1}|\dotp{u_\ast}{v}|\geq\alpha.
\] 
Indeed, let $u'\in V_d,\ v'\in T_y \m M$ be unit vectors such that $\cos(\phi)=\dotp{u'}{v'}$,
Note that $\sqrt{1-\alpha^2}$ is equal to the smallest absolute value among the nonzero singular values of the operator $\proj{}_{T_y \m M}\proj{}_{V_d}$. 
Since the spectra of the operators $\proj{}_{T_y \m M}\proj{}_{V_d}$ and $\proj{}_{V_d}\proj{}_{T_y \m M}$ coincide by well-known facts from linear algebra, we have that 
\[
\min_{u\in V_d, \|u\|=1}\max_{v\in T_y \m M,\|v\|=1}\left| \dotp{u}{v}\right|=
\min_{v\in T_y \m M,\|v\|=1}\max_{u\in V_d, \|u\|=1}\left| \dotp{u}{v}\right|.
\]
In other words, $\proj{}_{T_y \m M}(u')=\dotp{u'}{v'}v'$ and $\proj{}_{V_d}(v')=\dotp{u'}{v'}u'$. 
This implies that there exists a unit vector $u_\ast\in (V_d)^\perp$ such that $v'=\dotp{v'}{u'}u'+\dotp{v'}{u_\ast}u_\ast$, hence 
$\dotp{u_\ast}{v'}^2=1-\dotp{v'}{u'}^2=\alpha^2$, so $u_\ast$ satisfies the requirement.  

To simplify the expressions, let
$$
\zeta = \frac{1}{\Vol(Q\cap \m M_\sigma)}.
$$
We shall now construct upper and lower bounds for 
\begin{align*}
\zeta\int_{Q\cap \m M_\sigma} \dotp{u_\ast}{x-\mb EZ-\proj{}_{V_d}(x-\mb EZ)}^2 d\Vol(x)=\zeta\int_{Q\cap \m M_\sigma} \dotp{u_\ast}{x-\mb EZ}^2 d\Vol(x)
\end{align*}
which together yield an estimate for $\alpha$.
%Note that for any $u\in (V_d)^\perp$ (in particular, for $u=u_\ast$)
%\begin{align*}
%&
%\frac{1}{\Vol(Q\cap \m M_\sigma)}\int_{Q\cap \m M_\sigma} \|x-\mb EZ-\proj{}_{V_d}(x-\mb EZ)\|^2 d\Vol(x)\geq \\
%&
%\frac{1}{\Vol(Q\cap \m M_\sigma)}\int_{Q\cap \m M_\sigma} \dotp{u}{x-\mb EZ-\proj{}_{V_d}(x-\mb EZ)}^2 d\Vol(x)=\\
%&
%\frac{1}{\Vol(Q\cap \m M_\sigma)}\int_{Q\cap \m M_\sigma} \dotp{u}{x-\mb EZ}^2 d\Vol(x).
%\end{align*}
Write $u_\ast=u_\ast^{||}+u_\ast^\perp$, where $u_\ast^{||}\in T_y \m M$ and $u_\ast^\perp\in T_y^\perp \m M$. 
By our choice of $u_\ast$, we clearly have that $\|u_\ast^{||}\|=\max_{v\in T_y \m M, \|v\|=1}\dotp{u_\ast}{v}\geq \alpha$. 
Using the elementary inequality $(a+b)^2\geq \frac{a^2}{2}-b^2$, we further deduce that
\begin{align}
\label{eq:app1}
\zeta\int_{Q\cap \m M_\sigma} \dotp{u_\ast}{x-\mb EZ}^2 d\Vol(x)
&\geq \zeta\int_{Q\cap \m M_\sigma} \frac 1 2\dotp{u_\ast^{||}}{x-\mb EZ}^2 d\Vol(x)\\
&\nonumber \hspace{10pt}-\zeta \int_{Q\cap \m M_\sigma} \dotp{u_\ast^\perp}{x-\mb EZ}^2 d\Vol(x).
\end{align}
%As in the proof of Lemma \ref{lem:evalLB}, we define $f$ to be the inverse of the affine projection onto $y+T_y \m M$. 
%Using the notation and following the proof of Lemma \ref{lem:evalLB}), we have 
It follows from the proof of Lemma \ref{lem:evalLB} that 
\begin{align*}
\zeta\int_{Q\cap \m M_\sigma} \frac 1 2\dotp{u_\ast^{||}}{x-\mb EZ}^2 d\Vol(x)
& \geq 
\frac{\alpha^2}{8\left(1+\frac{\sigma}{\tau}\right)^d}\left(\frac{r_1-\sigma}{r_2+\sigma}\right)^d\left(\frac{1-\left(\frac{r_1-\sigma}{2\tau}\right)^2}{1+\left(\frac{2(r_2+\sigma)}{\tau-2(r_2+\sigma)}\right)^2}\right)^{d/2}\frac{(r_1-\sigma)^2}{d}.
\end{align*}
For the last term in (\ref{eq:app1}), Lemma \ref{lem:evalUB} (see equation (\ref{eq:f02})) gives
\begin{align*}
\zeta\int_{Q\cap \m M_\sigma} \dotp{u_\ast^\perp}{x-\mb EZ}^2 d\Vol(x)&\leq
\zeta\int_{Q\cap \m M_\sigma} \|x-\mb EZ-\proj{}_{T_y \m M}(x-\mb EZ)\|^2 d\Vol(x)\\
& \leq 2\sigma^2+\frac{8r_2^4}{\tau^2},
\end{align*}
hence (\ref{eq:app1}) yields
\begin{align}
\label{eq:app2}
\zeta\int_{Q\cap \m M_\sigma} \dotp{u_\ast}{x-\mb EZ}^2 d\Vol(x) \geq& \nonumber
\frac{\alpha^2}{8\left(1+\frac{\sigma}{\tau}\right)^d}\left(\frac{r_1-\sigma}{r_2+\sigma}\right)^d\left(\frac{1-\left(\frac{r_1-\sigma}{2\tau}\right)^2}{1+\left(\frac{2(r_2+\sigma)}{\tau-2(r_2+\sigma)}\right)^2}\right)^{d/2}\frac{(r_1-\sigma)^2}{d}\\
&-
2\sigma^2-\frac{8r_2^4}{\tau^2}.
\end{align}
On the other hand, invoking (\ref{eq:f02}) once again, we have
\begin{align*}
\zeta\int_{Q\cap \m M_\sigma} \dotp{u_\ast}{x-\mb EZ}^2 d\Vol(x)\leq 2\sigma^2+\frac{8r_2^4}{\tau^2}.
\end{align*}
Combined with (\ref{eq:app2}), this gives
\begin{align}
\label{eq:app3}
&
\frac{\alpha^2}{8\left(1+\frac{\sigma}{\tau}\right)^d}\left(\frac{r_1-\sigma}{r_2+\sigma}\right)^d\left(\frac{1-\left(\frac{r_1-\sigma}{2\tau}\right)^2}{1+\left(\frac{2(r_2+\sigma)}{\tau-2(r_2+\sigma)}\right)^2}\right)^{d/2}\frac{(r_1-\sigma)^2}{d}\leq 
4\sigma^2+\frac{16r_2^4}{\tau^2},
\end{align}
and the upper bound for $\alpha$ follows. 

Notice that for any $x\in Q\cap \m M_\sigma$,
\begin{align}
\label{eq:app4}
x-\mb EZ-\proj{}_{V_d}(x-\mb EZ)&=
x-y-\proj{}_{T_y \m M}(x-y)+\underbrace{y-\mb EZ-\proj{}_{T_y \m M}(y-\mb EZ)}_{\proj{}_{(T_y\m M)^\perp}(y-\mb EZ)}\\
& \nonumber \hspace{10pt} +
(\proj{}_{T_y \m M}-\proj{}_{V_d})(x-\mb EZ).
%\proj{}_{T_y \m M} x -\proj{}_{V_d}x = \proj{}_{T_y \m M}x - \proj{}_{V_d} \proj{}_{T_y \m M} x -\proj{}_{V_d}\proj{}_{(T_y \m M)^{\perp}}x,
\end{align}
It follows from (\ref{eq:f01}) that 
\begin{align*}
&
\|x-y-\proj{}_{T_y \m M}(x-y)\|=\left\|\proj{}_{T_y^\perp \m M}(x-y)\right\|
\leq \sigma+\frac{2r_2^2}{\tau}.
%\|\proj{}_{T_y \m M} x -\proj{}_{V_d}x\|\leq \| \proj{}_{T_y \m M}x - \proj{}_{V_d} \proj{}_{T_y \m M} x\|+\|\proj{}_{(T_y \m M)^{\perp}}x\|.
\end{align*}
Next, 
\begin{align*}
\|\proj{}_{(T_y\m M)^\perp}(y-\mb EZ)\|&=
\frac{1}{\Vol(Q\cap \m M_\sigma)}\left\|\int_{Q\cap \m M_\sigma}\proj{}_{T_y^\perp \m M}(y-z)d\Vol(z)\right\|\\ 
&\leq
\frac{1}{\Vol(Q\cap \m M_\sigma)}\int_{Q\cap \m M_\sigma}\left\|\proj{}_{T_y^\perp \m M}(z-y)\right\| d\Vol(z) \\
&\leq
\sigma+\frac{2r_2^2}{\tau}.
\end{align*}
Finally, it is easy to see that 
\begin{align*}
\|(\proj{}_{T_y \m M}-\proj{}_{V_d})(x-\mb EZ)\|\leq & \|\proj{}_{T_y \m M}(x-\mb EZ)-\proj{}_{V_d}\proj{}_{T_y \m M}(x-\mb EZ)\|\\
&+
\|\proj{}_{T_y^\perp \m M}(x-y)\|+\|\proj{}_{T_y^\perp \m M}(\mb E Z- y)\|.
\end{align*}
Let $u_x:=\frac{\proj{}_{T_y \m M}(x-\mb EZ)}{\|\proj{}_{T_y \m M}(x-\mb EZ)\|}$ and note that for any $x\in Q\cap \m M_\sigma$, $\|\proj{}_{T_y \m M}(x-\mb EZ)\|\leq 2r_2$, hence 
\begin{align*}
\|\proj{}_{T_y \m M}(x-\mb EZ)-\proj{}_{V_d}\proj{}_{T_y \m M}(x-\mb EZ)\|^2&\leq (2r_2)^2\left(1-\|\proj{}_{V_d}u_x\|^2\right)\\
&\leq4r_2^2\l(1-\min_{u\in T_y \m M, \|u\|=1}\max_{v\in V_d, \|v\|=1}\dotp{u}{v}^2\r)\\
&=4r_2^2 \alpha^2.
\end{align*}
Combining the previous bounds with (\ref{eq:app3}) and (\ref{eq:app4}), we obtain the result. 
\end{proof}

\begin{proof}[Proof of Lemma \ref{lem:covering}]
%The second statement is obvious from the definition of $V_{j,k}$. 
%The first follows from the properties of $C_j$ and $\tilde S$: 
Assume the event $\m E_{\eps/2,n}=\left\{\{Y_1,\ldots,Y_n\} \text { is an } \eps/2 \text{ - net in } \m M\right\}$ occurs.
By Proposition \ref{prop:net}, $\Pr(\m E_{\eps/2,n})\geq 1-e^{-t}$.    

Since the elements of $T_j$ are $2^{-j}$-separated, for any $1\leq k\leq N(j)$, $B(a_{j,k},2^{-j-1})\subseteq C_{j,k}$. 
Moreover, since $\sigma\leq 2^{-j-2}$ and $\|a_{j,k}-z_{j,k}\|\leq \sigma$, 
\[
B(z_{j,k},2^{-j-1}-2^{-j-2})\subseteq B(z_{j,k},2^{-j-1}-\sigma)\subseteq B(a_{j,k},2^{-j-1}),
\]  
hence the inclusion $B\left(z_{j,k}, 2^{-j-2}\right)\subseteq C_{j,k}$ follows. 

To show that $C_{j,k}\cap \m M_\sigma\subseteq B(a_{j,k},3\cdot 2^{-j-2}+2^{-j+1})$, pick an arbitrary $z\in\m M_\sigma$.  
%such that 
%\[
%\|z-a_{j,k}\|>3\cdot 2^{-j-2}+2^{-j+1}.
%\] 
Note that on the event $\m E_{\eps/2,n}$, there exists $y\in \{Y_1,\ldots,Y_n\}$ satisfying $\|z-y\|\leq \eps/2+\sigma$. 
Let $x(y)\in \m X_n$ be such that $y=\proj{}_\m M (x(y))$. 
By properties of the cover trees (see Remark \ref{rem:2}), there exists $x_\ast \in T_j$ such that $\|x(y)-x_\ast\|\leq 2^{-j+1}$. 
Then 
\begin{align*}
\|z-x_\ast\|\leq & \|z-y\|+\|y-x(y)\|+\|x(y)-x_\ast\|\leq \eps/2+2\sigma+2^{-j+1}\leq 3\cdot 2^{-j-2}+2^{-j+1}.
%<\|z - a_{j,k}\|,
\end{align*} 
%hence $z\notin C_{j,k}$.   
Since $z$ was arbitrary, the result follows. 
Finally, $B(a_{j,k},3\cdot 2^{-j-2}+2^{-j+1})\subset B(z_{j,k},3\cdot 2^{-j})$ holds since $\|a_{j,k}-z_{j,k}\|\leq 2^{-j-2}$. 
\end{proof}

\vskip 0.1in
\bibliography{MyPublications,BigBib}

\end{document}